\documentclass{article}
\usepackage{amsmath,amssymb,amsfonts,amsthm,url}
\usepackage{graphicx}
\usepackage{enumitem,caption}
\usepackage{amssymb,amsbsy,amsmath}
\usepackage{a4}

\newcommand{\xLtwo}{{\mathrm{L}^2}}
\newcommand{\xHone}{{\mathrm{H}^1}}
\newcommand{\xHn}[1]{{\mathrm{H}^{#1}}}
\newcommand{\xCzero}{{\mathrm{C}^0}}
\newcommand{\xCone}{{\mathrm{C}^1}}

\newcommand{\xWn}[1]{{\mathrm{W}^{#1}}}

\DeclareMathOperator{\xim}{{\mathrm{im}}}
\newcommand{\xdif}{{\mathrm{d}}}

\newcommand{\xA}{{\mathbb{A}}}
\newcommand{\xN}{{\mathbb{N}}}
\newcommand{\xR}{{\mathbb{R}}}
\newcommand{\xC}{{\mathbb{C}}}

\newcommand{\etc}{{\textit{etc}}}
\newcommand{\afortiori}{{\textit{a~fortiori}}}
\newcommand{\cf}{{\textit{cf.}}}
\newcommand{\ie}{{\textit{i.e.}}}
\newcommand{\vg}{{\textit{v.g.}}}

\let\Rref=\ref

\newtheorem{thrm}{Theorem}[section]
\newtheorem{prpstn}[thrm]{Proposition}
\newtheorem{lmm}[thrm]{Lemma}
\newtheorem{crllr}[thrm]{Corollary}

\newenvironment{acknowledgement}{{\it Acknowledgement.\quad}}{}

\DeclareMathOperator{\curl}{{\mathbf{curl}}}

\let\divg=\divergence
\DeclareMathOperator{\grad}{{\mathbf{grad}}}

\theoremstyle{definition}
\newtheorem{hpthss}{Hypothesis}
\newtheorem{rmrk}[thrm]{Remark}
\newtheorem{dfntn}[thrm]{Definition}

\newcommand{\ii}{{\mathrm{i}}}
\newcommand{\ee}{{\mathrm{e}}}

\newcommand{\III}{|\!|\!|}

\title{Linearised electrodynamics and stabilization\\ of a cold magnetized plasma}
\author{Simon Labrunie$^1$ and Ibtissem Zaafrani$^{1,2}$\thanks{The second author thanks the Campus France Eiffel Excellence Programme for its financial support.}\\
$^1$ Universit\'e de Lorraine, CNRS, IECL, F-54000 Nancy, France \\
$^2$ Universit\'e de Sousse, Laboratoire de Math\'ematiques, \\
\'Ecole Sup\'erieure des Sciences et de la technologie,\\ 4011 Hammam Sousse, Tunisie}

\date{May 26, 2021}

\begin{document}
\maketitle

\begin{abstract}
We consider a linearized Euler--Maxwell model for the propagation and absorption of electromagnetic waves in a magnetized plasma. We present the derivation of the model, and we show its well-posedeness, its strong and polynomial stability under suitable and fairly general assumptions, its exponential stability in the same conditions as the Maxwell system, and finally its convergence to the time-harmonic regime. No homogeneity assumption is made, and the topological and geometrical assumptions on the domain are minimal. These results appear strongly linked to the spectral properties of various matrices describing the anisotropy and other plasma properties.
\smallbreak
\smallbreak
\end{abstract}

\section{Introduction}
\label{sec-intro}
Electromagnetic wave propagation in plasmas, especially magnetized ones, is a vast subject~\cite{S92}. Even in a linear framework, the equations that describe it are generally highly anisotropic and, in many practical settings, highly inhomogeneous as well. The bewildering array of phenomena and parameters involved in this modelling requires to derive simplified models tailored to the phenomenon under study, and to the theoretical or computational purpose of this study.

\medbreak

Wave-plasma interaction is of paramount importance, for instance, in tokamak technology. According to their frequency, electromagnetic waves can be used in a wide range of processes: to stabilize or heat the plasma and thus bring it closer to the conditions needed for nuclear fusion, for instance, or to probe various properties such as density and temperature. These interactions involve many phenomena, such as propagation, absorption, refraction, scattering, \etc. The basic physics is well understood~\cite{S92}; nevertheless, efficient and robust mathematical models have to be derived in order to do reliable numerical simulations in realistic settings, or to properly interpret experimental results. 

\medbreak

A first, time-harmonic model focused on propagation and absorption has been derived in~\cite{BHL+15,H+14}. This article constitutes the time-dependent counterpart of those works. We consider a general linearized Euler--Maxwell model describing the interaction between a strongly magnetized, pressureless, totally ionized gas and an electromagnetic wave; it can be particularized to various physical settings. The waves accelerate the charged particles that make up the plasma, and transfer some of their energy to them through collisions, which act as friction. We study the well-posedness of the model, investigate various stability properties (strong, polynomial, and exponential), and finally check that no inconsistency stems from the time-harmonic modelling. The latter appears, as expected, as a particular solution and a limit of the general solution (under reasonable physical assumptions) in presence of a time-harmonic forcing. Not only the time-dependent model is more general, but it also appears more robust. The well-posedness of the time-harmonic model rests upon absorption; more exactly, the proof fails in the absence of absorption, and serious qualitative arguments suggest that the limiting model is actually ill-posed~\cite{BHL+15}.

\medbreak

On the other hand, we shall see that the time-dependent model is well-posed even without absorption. Nevertheless, and unsurprisingly, the convergence (exponential or polynomial) of the time-dependent model toward the time-harmonic one does depend on absorption. The mathematical tools used in this analysis are well-known theorems on semigroups and operator spectra~\cite{AB88,LV88,H85,P84}.
The main difficulties are: first, the resolvent of the evolution operator is not compact; then, absorption only acts on some variables, namely, the hydrodynamic ones; finally, one has to handle with various technicalities linked to inhomogeneity, anisotropy, and topology.
More or less similar models have been studied by various authors~\cite{N+12, N+18, VKLS+13}; but they did not include anisotropy or inhomogeneity, and they generally considered simpler topologies or boundary conditions than we do. On the contrary, we have tried to keep our model as general as possible, by assuming neither any homogeneity in the plasma properties, nor in the external magnetic field, nor any strong topological or geometrical condition on the domain. 

\medbreak

Generally speaking, stabilization and controllability of Maxwell's equations and coupled models involving them may be rooted in two main physical mechanisms: boundary stabilization of the sourceless Maxwell system through an absorbing boundary condition on all or part of the boundary \cite[among others]{BaHa97,Komo94,Phung00,ELNi02,NiPi03}; or internal stabilization by some resistive source term on all or part of the domain, typically that given by Ohm's law \cite[among others]{Phung00}. 
Boundary absorption is usually sufficient to have energy decay and convergence toward an equilibrium state (strong stability~\cite{BaHa97}), but the precise decay rate (polynomial or exponential stability) strongly depends on the global shape of the domain (various star-shapedness conditions such as~\cite{Komo94,NiPi03}) and/or the absorbing part of the boundary (geometric control condition~\cite{Phung00}). Internal absorption, when it only holds on part of the domain, also requires a geometric control condition~\cite{Phung00}.
Mathematically, the issue has been tackled by several approaches: the multiplier method~\cite[etc.]{Komo94,ELNi02}, microlocal analysis~\cite[etc.]{Phung00}, or frequency-domain analysis~\cite{N+18}. In this paper we use the later method, and we focus on internal stabilisation; the internal absorption mechanism, however, is different from Ohm's law. This allows us to consider fairly arbitrary geometries and topologies, provided some physically reasonable (in the framework of tokamak plasmas) assumption holds (Hypothesis~\Rref{hyp-2} below; cf.~\cite[Remark~4.1]{BHL+15}). 
The complementary approach, where the physical hypotheses are weakened at the price of stricter conditions on the geometry, is reserved for future work.

\medbreak

The outline of the article is as follows. 
In \S\Rref{sec-model}, we present the derivation of the model, and recall some classical results on the functional analysis of Maxwell's equations in~\S\Rref{sec-prelim}. Section~\Rref{sec-well-posed} is devoted to the proof of the well-posedness of the model, in three variants: with a perfectly conducting condition on the whole boundary, and with a Silver--M\"uller one (homogeneous or not) on part of it.
Section~\Rref{sec-prelim-2} recalls or introduces some more advanced results of functional analysis, which are needed in the sequel. In~\S\Rref{sec-spectral}, we study the spectral properties of various matrices describing the anisotropy and other plasma properties, which will be essential in the stability proofs of \S\S\Rref{sec-strong-stab} and~\Rref{sec-expoly-stab}. The former is dedicated to strong stability, the latter to unconditional polynomial and conditional exponential stability. Though it happens that the polynomial stability does not entail stronger hypotheses than strong one, we have chosen to present them sequentially: the results which allow us to prove strong stability are also the starting point for the finer properties needed to prove polynomial and exponential stability. The stability part is also divided into perfectly conducting and Silver--M\"uller boundary conditions. As an application, we conclude with a result of  convergence to the time-harmonic regime when the Silver--M\"uller boundary data is time-harmonic.

\section{The model}
\label{sec-model}
The physical system we are interested in is a plasma or totally ionized gas, pervaded by a strong, external, static magnetic field $\boldsymbol{B}_{\text{ext}}(\boldsymbol{x})$, which makes the medium anisotropic. The sources of this field are assumed to be outside the plasma. Such a medium can be described as a collection of charged particles (electrons and various species of ions) which move in vacuum and create electromagnetic fields which, in turn, affect their motion. Electromagnetic fields are, thus, governed by the usual Maxwell's equations in vacuum:
\begin{eqnarray}
\curl \boldsymbol{\mathcal{E}} = - \frac{\partial \boldsymbol{\mathcal{B}}}{\partial t}, && c^2\, \curl \boldsymbol{\mathcal{B}} = \frac{\boldsymbol{\mathcal{J}}}{\varepsilon_0} + \frac{\partial \boldsymbol{\mathcal{E}}}{\partial t} ,
\label{maxrot}\\
\divg \boldsymbol{\mathcal{E}} = \frac{\varrho}{\varepsilon_0}, && \divg \boldsymbol{\mathcal{B}} = 0 .
\label{maxdiv}
\end{eqnarray}
Here $\boldsymbol{\mathcal{E}}$ and~$\boldsymbol{\mathcal{B}}$ denote the electric and magnetic fields; $\varrho$~and $\boldsymbol{\mathcal{J}}$ the electric charge and current densities; $\varepsilon_0$~is the electric permittivity, and $c$ the speed of light, in vacuum.

\medbreak

The electromagnetic field is the sum of a static part and a small perturbation caused by the penetration of an electromagnetic wave. To simplify the discussion, we assume
the plasma to be in mechanical and electrostatic equilibrium in the absence of the wave. Thus, the electric and magnetic fields can be written as:
\begin{equation*}
\boldsymbol{\mathcal{E}}(t,\boldsymbol{x})= \epsilon\, \boldsymbol{E}(t,\boldsymbol{x}), \quad \text{and} \quad \boldsymbol{\mathcal{B}}(t,\boldsymbol{x})=\boldsymbol{B}_{\text{ext}}(\boldsymbol{x})+\epsilon\, \boldsymbol{B}(t,\boldsymbol{x}),
\end{equation*}
where $\epsilon\ll 1$ is the perturbation parameter.  The total charge and current densities associated with the fields are those due to the perturbation
\begin{eqnarray}
\varrho(t,\boldsymbol{x})=\epsilon\, \rho(t,\boldsymbol{x}), \quad \text{and} \quad \boldsymbol{\mathcal{J}}(t,\boldsymbol{x})= \epsilon\, \boldsymbol{J}(t,\boldsymbol{x}).
\label{varr J}
\end{eqnarray}
The static parts of $\boldsymbol{\mathcal{E}}$, $\varrho$ and $\boldsymbol{\mathcal{J}}$ are zero by the equilibrium assumption. 

\medbreak

Furthermore, we assume the plasma to be cold, \ie, the thermal agitation of particles, and thus their pressure, is negligible.  We shall designate the particles species (electrons and various species of ions) with the index~$s$. We denote as $q_s$ the charge of one particle and $m_s$ its mass. The momentum conservation equation of particles of the species~$s$ writes:
\begin{equation}
m_s\,\frac{\partial \boldsymbol{\mathcal{U}}_s}{\partial t} + m_s\,(\boldsymbol{\mathcal{U}}_s \cdot \nabla)\, \boldsymbol{\mathcal{U}}_s - q_s\, (\boldsymbol{\mathcal{E}}+\,\boldsymbol{\mathcal{U}}_s \times \mathcal{B}) + m_s \,\nu_s\, \boldsymbol{\mathcal{U}}_s=0,
\label{fg}
\end{equation}
where $\boldsymbol{\mathcal{U}}_s$ denotes the fluid velocity and $\nu_s\geq0$ is  the collision frequency which only depends on the variable $\boldsymbol{x}$. The charge and current densities can be  expressed as a function of the particle densities $n_s(t,\boldsymbol{x})$ and the fluid velocities:
\begin{equation*}
\varrho =\sum_s \varrho_s = \sum_s q_s\, n_s,  \qquad \boldsymbol{\mathcal{J}}= \sum_s \boldsymbol{\mathcal{J}}_s =\sum_s q_s\,n_s\,\boldsymbol{\mathcal{U}}_s.
\end{equation*}
Now, multiplying Equation~\eqref{fg} by $\frac{n_s\,q_s}{m_s}$, we get 
\begin{equation}
\frac{\partial \boldsymbol{\mathcal{J}}_s}{\partial t} + \frac{1}{\varrho_s}\,(\boldsymbol{\mathcal{J}}_s \cdot \nabla)\, \boldsymbol{\mathcal{J}}_s - \frac{q_s}{m_s}\,(\varrho_s\,\boldsymbol{\mathcal{E}} + \boldsymbol{\mathcal{J}}_s \times \mathcal{B}) + \nu_s\, \boldsymbol{\mathcal{J}}_s = 0.
\label{J}
\end{equation}
We now linearize Equation~\eqref{J}. From the above discussion, we can assume, for each species~$s$,
\begin{eqnarray*}
\varrho_s(t,\boldsymbol{x})=q_s\,n^0_s(\boldsymbol{x})\,+\,\epsilon\, \rho_s(t,\boldsymbol{x}), \quad \text{and} \quad \boldsymbol{\mathcal{J}}_s(t,\boldsymbol{x})= \epsilon\, \boldsymbol{J}_s(t,\boldsymbol{x}),
\end{eqnarray*}
where $n^0_s$ is the equilibrium particle density, assumed to depend on~$\boldsymbol{x}$ only. On the left-hand side of~\eqref{J}, the terms of order~$0$ in~$\epsilon$ vanish. To express the terms of order~$1$, we introduce the plasma and cyclotron frequencies for the species~$s$, respectively:
\begin{equation}
\omega_{ps} := \sqrt{\frac{n^0_s\, q^2_s}{ \varepsilon_0\,m_s}},\qquad \varOmega_{cs} := \frac{q_s\,|\boldsymbol{B}_{\text{ext}}|}{m_s} \,;
\label{rrt}
\end{equation}
they only depend on the space variable~$\boldsymbol{x}$. Observe that the cyclotron frequency is signed: it has the same sign as the charge~$q_s$.  Finally, denoting $\boldsymbol{b} = \dfrac{\boldsymbol{B}_{\text{ext}}}{|\boldsymbol{B}_{\text{ext}}|}$ the unit vector aligned with the external magnetic field, we obtain the linearized equation:
\begin{equation}
\frac{\partial \boldsymbol{J}_s}{\partial t} - \varepsilon_0\,\omega^2_{ps}\,\boldsymbol{E} - \varOmega_{cs}\,\boldsymbol{J}_s\times \boldsymbol{b}+\nu_s\, \boldsymbol{J}_s=0.
\label{pp1}
\end{equation}

The perturbative electromagnetic field $(\boldsymbol{E},\boldsymbol{B})$ satisfies, at order 1 in $\epsilon$, the usual Maxwell equations derived form \eqref{maxrot} and~\eqref{maxdiv}, namely the evolution equations:
\begin{equation*}
\curl \boldsymbol{E} = -\frac{\partial \boldsymbol{B}}{\partial t}, \qquad c^2\,\curl  \boldsymbol{B} = \frac{\boldsymbol{J}}{\varepsilon_0} + \frac{\partial \boldsymbol{E}}{\partial t},\quad \text{where:}\quad \boldsymbol{J} := \sum_{s} \boldsymbol{J}_{s} \,,
\end{equation*}
and the divergence equations:
\begin{eqnarray}
\divg \boldsymbol{E} &=& \frac{\rho}{\varepsilon_0},\quad \text{where:}\quad \rho = \sum_s \rho_s \,,
\label{div E}\\
\divg\boldsymbol{B} &=& 0.
\label{div B}
\end{eqnarray}
Indeed, as its sources are outside the plasma, $\boldsymbol{B}_{\text{ext}}(\boldsymbol{x})$ is  curl- and divergence-free.

\medbreak

For the sake of simplicity, we assume that there only are two species of particles in the plasma: the electrons ($s=1$) and one kind of ions ($s=2$). Obviously, the whole discussion can be extended to an arbitrary number of species, provided they all carry an electric charge (no neutral atoms).

\medbreak

All in all, the model which will be the object of this article is the following. 
Let $\Omega$ be a domain in $\xR^3$,  \ie, a bounded, open and connected subset of~$\xR^3$ with a Lipschitz boundary $\Gamma := \partial \Omega$. The evolution equation for the hydrodynamic and electromagnetic variables are:
\begin{eqnarray} 
\frac{\partial \boldsymbol{J}_1}{\partial t} &=& \varepsilon_0\,\omega_{p1}^2\,\boldsymbol{E} + \varOmega_{c1}\,  \boldsymbol{J}_1 \times \boldsymbol{b} - \nu_1\, \boldsymbol{J}_1 \,,\quad \text{in } \Omega\times\xR_{>0};
\label{uu}
\\
\frac{\partial \boldsymbol{J}_2}{\partial t} &=& \varepsilon_0\,\omega_{p2}^2\,\boldsymbol{E} + \varOmega_{c2}\,  \boldsymbol{J}_2 \times \boldsymbol{b} - \nu_2\, \boldsymbol{J}_2  \,,\quad \text{in } \Omega\times\xR_{>0};
\label{ii}
\\
\frac{\partial \boldsymbol{E}}{\partial t} &=& c^2\, \curl\,\boldsymbol{B} - \frac1{\varepsilon_0}\, \sum_s  \boldsymbol{J}_{s} \,,\quad \text{in } \Omega\times\xR_{>0};
\label{oo}
\\
\frac{\partial \boldsymbol{B}}{\partial t} &=& - \curl\boldsymbol{E}  \label{pp2} \,,\quad \text{in } \Omega\times\xR_{>0};
\end{eqnarray}
with the initial conditions at $t=0$:
\begin{equation}
\boldsymbol{J}_1(0)=\boldsymbol{J}_{1,0};~\boldsymbol{J}_2(0)=\boldsymbol{J}_{2,0};~ \boldsymbol{E}(0)=\boldsymbol{E}_0 ;~ \boldsymbol{B}(0)=\boldsymbol{B}_0,
\quad \text{in } \Omega.
\label{init cond 0}
\end{equation}
The boundary $\Gamma$ is split into two parts $\Gamma = \overline{\Gamma_A} \cup \overline{\Gamma_P}$, with $\Gamma_A \cap \Gamma_P = \varnothing$; both $\Gamma_P$ and~$\Gamma_A$ may be empty. On~$\Gamma_P$, there holds a usual perfectly conducting boundary (metallic) condition. On~$\Gamma_A$, there holds a Silver--M\"uller boundary condition: 
\begin{eqnarray}
\boldsymbol{E} \times \boldsymbol{n} &=& 0, \quad \text{on } \Gamma_P\times\xR_{>0} \,, 
\label{lk}\\
\boldsymbol{E}\times \boldsymbol{n}+ c\,\boldsymbol{B}_\top &=& \boldsymbol{g}, \quad \text{on }  \Gamma_A\times\xR_{>0} \,, 
\label{lm}
\end{eqnarray}
where $\boldsymbol{n}$ denotes the outward unit normal vector to $\Gamma$, $\boldsymbol{B}_\top$ is the component of~$\boldsymbol{B}$ tangent to the boundary~$\Gamma$, and $\boldsymbol{g}$ is a data defined on $\Gamma_A\times\xR_{>0}$. If $\boldsymbol{g}=0$, this is an \emph{absorbing} or \emph{outgoing wave} condition, meaning that the electromagnetic energy can freely leave the domain through~$\Gamma_A$. If $\boldsymbol{g}\neq0$, this is an \emph{incoming wave} condition, modelling  the injection of an electromagnetic wave into the plasma, and $\Gamma_A$ is interpreted as an antenna (see Fig.~\Rref{fig2} for a possible configuration). 

\smallbreak

The subsets $\overline{\Gamma_A}$ and~$\overline{\Gamma_P}$ are compact Lipschitz submanifolds of~$\Gamma$. When both $\Gamma_P\neq\varnothing$ and $\Gamma_A\neq\varnothing$, we do not necessarily suppose that $\partial\Gamma_A \cap \partial\Gamma_P = \varnothing$ (\ie, we consider both truncated exterior and interior problems), but we do assume that $\Gamma_A$ is not too irregular. A sufficient condition is to assume it either smooth, or polyhedral without so-called \emph{pathological vertices} \cite[p.~204]{ACL+17}. This requirement is not very stringent; it can always be satisfied in the outgoing wave case, where $\Gamma_A$ appears as an artificial boundary, whose exact location and shape are to some extent arbitrary.

\smallbreak

Otherwise, our assumptions on the domain are minimal. We do not assume $\Omega$ to be topologically trivial (but we do assume that it does not have an \emph{infinitely} multiple topology), nor $\Gamma,\ \Gamma_A,\ \Gamma_P$ to be connected (though we do assume that they have a \emph{finite} number of connected components). The perfectly conducting boundary~$\Gamma_P$, if not empty, is just assumed to be Lipschitz.

\medbreak

The solution to the system \eqref{uu}--\eqref{pp2} with boundary conditions~\eqref{lk}--\eqref{lm} can be shown to satisfy Equation~\eqref{div B}~in $\Omega$ for all $t\geq 0$, as well the boundary condition
\begin{equation}
\boldsymbol{B}\cdot \boldsymbol{n} = 0 \quad \text{on } \Gamma_P\times\xR_{>0} ,
\label{lkm}
\end{equation}
provided they hold at $t=0$. These properties will appear crucial for the derivation of the most suitable functional framework for stabilization. Similarly, Equation~\eqref{div E} can be recovered if it holds at $t=0$ and the charge conservation equation $\partial_t \rho + \divg\boldsymbol{J}=0$ is verified; yet, the latter is an immediate consequence of the continuity equations for the various species, \emph{viz.}, $\partial_t \rho_s + \divg\boldsymbol{J}_s=0$.

\medbreak

In this article we are interested in to cases: first when $\Gamma_A=\varnothing$, \ie, we have a perfect conductor condition on the whole boundary; the second case when $\Gamma_A$~is non-empty and so Equation~\eqref{lm} holds. In the second case, conditions on and statements about~$\Gamma_P$ are of course void if~$\Gamma_P=\varnothing$.

\medbreak

As alluded to in~\S\Rref{sec-intro}, the model \eqref{uu}--\eqref{pp2} has been studied by many authors \cite{N+12, N+18, VKLS+13} when the medium is homogeneous and isotropic, \ie, $\varOmega_{cs}\equiv0$, and $\nu_s$ and $\omega_{ps}$ are constants. The dispersive medium model with perfectly conducting boundary condition on the whole boundary has been studied in~\cite{N+12,VKLS+13}, and it was proven in~\cite{N+12} that it is polynomially stable. In~\cite{N+18}, the differential equation~\eqref{pp1} is set in a subset of the full domain, and the Silver--M\"uller boundary condition is imposed on the entire exterior boundary; it was shown that the model is strongly stable. 
Similarly, some works on the boundary stabilization of the sourceless Maxwell system (\vg,~\cite{ELNi02,NiPi03}) allow for some inhomogeneity or anisotropy of the permittivity and permeability coefficients, and some works on the internal stabilization by Ohm's law (\vg,~\cite{Phung00}) allow for an inhomogeneous conductivity. But these inhomogeneous or anisotropic terms occur in Maxwell's equations themselves, not in a coupled ODE as in our model. Furthermore, most of the works which consider mixed boundary conditions actually assume $\partial\Gamma_A \cap \partial\Gamma_P = \varnothing$ (physically, a truncated exterior problem).

\smallbreak

Therefore, our goal in the present work is to investigate the stabilization of the model in an inhomogeneous and anisotropic medium with space variable coefficients $\nu_s$, $\omega_{ps}$ and~$\varOmega_{cs}$, for both types of boundary conditions and both types of truncated problems.
We will give sufficient conditions on these coefficients that guarantee first the strong stability, and then the polynomial or exponential stability of the energy.

\begin{figure}
\centering \includegraphics[height=5.0cm]{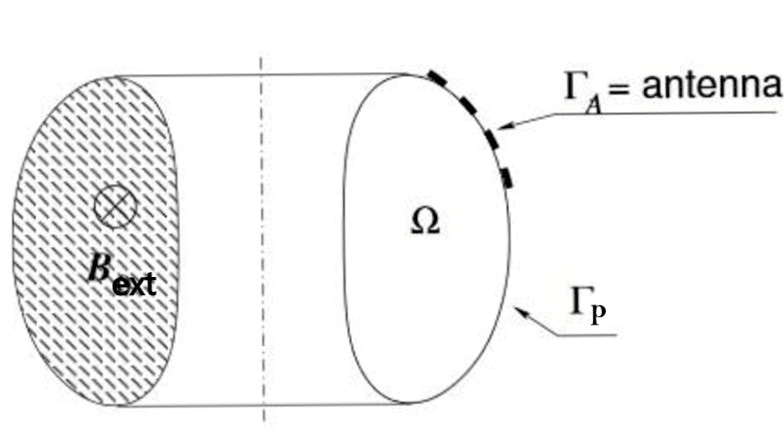}
\caption{A cross--section of an example of a domain which represents the plasma volume in a tokamak.} 
\label{fig2}
\end{figure}

\section{Preliminaries}
\label{sec-prelim}
In this section, we introduce the Hilbert spaces needed in the study of Maxwell's equations, and the relevant Green's formulas used in the sequel. 

\medbreak

The Sobolev spaces of vector fields $\mathbf{L}^2(\Omega):=(\xLtwo(\Omega))^3$, $\mathbf{H}^1(\Omega):=(\xHone(\Omega))^3$ and $\mathbf{H}^{\ell}(\Gamma):=(\xHn{\ell}(\Gamma))^3$ for $\ell\in\lbrace  \frac1{2},-\frac1{2}\rbrace$ are defined as usual. We denote $(\cdot \mid \cdot)$ the inner products of both $\xLtwo(\Omega)$ and $\mathbf{L}^2(\Omega)$, and $\|\cdot\|$ the associated norm. As usual, $\xHone_0(\Omega)$ is the subspace of $\xHone(\Omega)$ whose elements vanish on the boundary~$\Gamma$. 
The space $\widetilde{\mathbf{H}}^{\frac1{2}}(\Gamma_A)$ is the subspace of $\mathbf{H}^{\frac1{2}}(\Gamma_A)$ (the trace space of $\mathbf{H}^1(\Omega)$ on~$\Gamma_A$) made of fields defined on $\Gamma_A$ such that their extension by zero to~$\Gamma_P$ belongs to~$\mathbf{H}^{\frac1{2}}(\Gamma)$. The space $\widetilde{\mathbf{H}}^{-\frac1{2}}(\Gamma_A)$ is the dual space of $\widetilde{\mathbf{H}}^{\frac1{2}}(\Gamma_A)$. 

\smallbreak

On the other hand, for any Hilbert space $\mathcal{W}$ other than $\xLtwo(\Omega)$ or $\mathbf{L}^2(\Omega)$, its inner product will be denoted by $(\cdot,\cdot)_{\mathcal{W}}$ and its norm by~$\| \cdot \|_{\mathcal{W}}$. The duality pairing between $\mathcal{W}$ and its dual space is written as~$\langle \cdot,\cdot \rangle_{\mathcal{W}}$; the subscript designates the space to which the \emph{second} variable belongs.

\medbreak

The spaces $\mathbf{H}(\divg ; \Omega)$ and $\mathbf{H}(\curl ; \Omega)$ are the usual ones in electromagnetics; they are endowed with their canonical norm. The respective subspaces of fields with vanishing normal (resp.~tangential) trace are denoted $\mathbf{H}_0(\divg;\Omega)$ (resp.~$\mathbf{H}_0(\curl;\Omega)$).
The ranges of the tangential trace mapping $\gamma_\top:\boldsymbol{v} \mapsto \boldsymbol{v} \times \boldsymbol{n}$ and the tangential component mapping $\pi: \boldsymbol{v} \mapsto \boldsymbol{v}_\top:= \boldsymbol{n} \times(\boldsymbol{v}\times\boldsymbol{n})$ from $\mathbf{H}(\curl;\Omega)$ are denoted by 
\begin{eqnarray*}
\mathbf{TT}(\Gamma):=\lbrace \boldsymbol{\varphi} \in \mathbf{H}^{-\frac1{2}}(\Gamma)\,:\,\exists \boldsymbol{v} \in \mathbf{H}(\curl ;\Omega),~ \boldsymbol{\varphi}=\boldsymbol{v} \times \boldsymbol{n}_{|\Gamma}\rbrace ,
\end{eqnarray*}
\\[-35pt]
\begin{eqnarray*}
\mathbf{TC}(\Gamma):=\lbrace \boldsymbol{\lambda} \in \mathbf{H}^{-\frac1{2}}(\Gamma)\,:\,\exists \boldsymbol{v} \in \mathbf{H}(\curl ;\Omega),~ \boldsymbol{\lambda}=\boldsymbol{v}_{\top{|\Gamma}} \rbrace. \quad  ~
\end{eqnarray*}
These two spaces have been described in \cite{BC+1}, where they are respectively denote $\mathbf{H}_{\parallel}^{-\frac1{2}}(\divg_{\Gamma}, \Gamma)=\mathbf{TT}(\Gamma)$ and $\mathbf{H}_{\perp}^{-\frac1{2}}(\curl_{\Gamma}, \Gamma)=\mathbf{TC}(\Gamma)$. Furthermore \cite{BC+2}, they are in duality with respect to the pivot space 
\begin{eqnarray*}
\mathbf{L}^2_t(\Gamma):=\{ \boldsymbol{v} \in \mathbf{L}^2(\Gamma)\,:\, \boldsymbol{v} \cdot \boldsymbol{n}=0 \}.
\end{eqnarray*}
Therefore, one can prove the following  formula:
\begin{equation}
\forall(\boldsymbol{v},\boldsymbol{w})\in\mathbf{H}(\curl;\Omega)^2,\quad
(\boldsymbol{v}\mid\curl \boldsymbol{w})-(\curl\boldsymbol{v}\mid \boldsymbol{w})=\langle \boldsymbol{v}\times \boldsymbol{n} , \boldsymbol{w}_\top \rangle_{\mathbf{TC}(\Gamma)}.
\label{green rot rot}
\end{equation}
The spaces $\mathbf{TT}(\Gamma_A)$ and $\mathbf{TC}(\Gamma_A)$ denote respectively the ranges of $\gamma_\top$ and $\pi_\top$, restricted on the part $\Gamma_A$ of the boundary. In \cite{BC+1}, they are called $\mathbf{H}_{\parallel,00}^{-\frac1{2}}(\divg_{\Gamma_A}, \Gamma_A)$ and $\mathbf{H}_{\perp,00}^{-\frac1{2}}(\curl_{\Gamma_A}, \Gamma_A)$. The subspace of elements of $\mathbf{H}(\curl;\Omega)$ such that the tangential trace vanishes on the part $\Gamma_P$ of the boundary is denoted by
\begin{equation*}
\mathbf{H}_{0,\Gamma_P}(\curl ; \Omega)= \lbrace \boldsymbol{v}\in\mathbf{H}(\curl;\Omega)\,:\, \boldsymbol{v} \times \boldsymbol{n}_{|_{\Gamma_P}}=0  \rbrace .
\end{equation*}
Then, the range of the trace mappings on $\Gamma_A$ from $\mathbf{H}_{0,\Gamma_P}(\curl ; \Omega)$ are denoted 
\begin{eqnarray*}
\widetilde{\mathbf{TT}}(\Gamma_A)
&:=& \{\boldsymbol{\varphi} \in \mathbf{H}^{-\frac1{2}}(\Gamma_A)\,:\,\exists \boldsymbol{v} \in \mathbf{H}_{0,\Gamma_P}(\mathbf{\curl}; \Omega),~ \boldsymbol{\varphi}=\boldsymbol{v} \times \boldsymbol{n}_{|_{\Gamma_A}}\}
\\
&=& \{\boldsymbol{\varphi} \in \mathbf{TT}(\Gamma_A)\,:\,  \text{the extension of } \boldsymbol{\varphi} \text{ by $0$ to } \Gamma \text{ belongs to } \mathbf{TT}(\Gamma)\} \,;
\\
\widetilde{\mathbf{TC}}(\Gamma_A)
&:=& \{\boldsymbol{\lambda} \in \mathbf{H}^{-\frac1{2}}(\Gamma_A)\,:\,\exists \boldsymbol{v} \in \mathbf{H}_{0,\Gamma_P}(\mathbf{\curl}; \Omega),~ \boldsymbol{\lambda}={\boldsymbol{v}_\top}_{|_{\Gamma_A}}\} 
\\
&=& \{\boldsymbol{\lambda} \in \mathbf{TC}(\Gamma_A)\,:\,  \text{the extension of } \boldsymbol{\lambda} \text{ by $0$ to } \Gamma \text{ belongs to } \mathbf{TC}(\Gamma)\} \,;
\end{eqnarray*}
they are respectively called $\mathbf{H}_{\parallel}^{-\frac1{2}}(\divg_{\Gamma_A}^0, \Gamma_A)$ and $\mathbf{H}_{\perp}^{-\frac1{2}}(\curl_{\Gamma_A}^0, \Gamma_A)$ in \cite{BC+1}. 
The spaces $\widetilde{\mathbf{TT}}(\Gamma_A)$ and $\mathbf{TC}(\Gamma_A)$ are in duality with respect to the pivot space $\mathbf{L}_t^2(\Gamma_A)$, and similarly for 
$\mathbf{TT}(\Gamma_A)$ and $\widetilde{\mathbf{TC}}(\Gamma_A)$. We denote the duality product between those spaces as ${}_{\gamma^0_A}\langle \cdot,\cdot \rangle_{\pi_A}$ or ${}_{\gamma_A}\langle \cdot , \cdot \rangle_{\pi^0_A}$. 
This allows one to derive the following integration by parts formula:
\begin{eqnarray}
\forall (\boldsymbol{v},\boldsymbol{w}) \in\mathbf{H}(\curl;\Omega) \times \mathbf{H}_{0,\Gamma_P}(\curl;\Omega),\quad (\boldsymbol{v}\mid\curl \boldsymbol{w})-(\curl\boldsymbol{v}\mid \boldsymbol{w})={}_{\gamma_A}\langle \boldsymbol{v}\times \boldsymbol{n}, \boldsymbol{w}_\top \rangle_{\pi^0_A}.
\label{green rot 0 Gamma p}
\end{eqnarray}
If $\Gamma_P=\varnothing$, \ie, $\Gamma_A=\Gamma$, then $\mathbf{H}_{0,\Gamma_P}(\curl;\Omega) = \mathbf{H}(\curl;\Omega)$, while $\widetilde{\mathbf{TT}}(\Gamma_A) = \mathbf{TT}(\Gamma)$ and $\widetilde{\mathbf{TC}}(\Gamma_A) = \mathbf{TC}(\Gamma)$. 
In the case of a truncated exterior problem, \ie, $\partial\Gamma_A \cap \partial\Gamma_P = \varnothing$, it holds again that $\widetilde{\mathbf{TT}}(\Gamma_A) = \mathbf{TT}(\Gamma_A)$ and $\widetilde{\mathbf{TC}}(\Gamma_A) = \mathbf{TC}(\Gamma_A)$.

\smallbreak

A profound and useful property of these spaces is: in the absence of pathological vertices, $ \widetilde{\mathbf{TT}}(\Gamma_{A})\cap \mathbf{TC}(\Gamma_{A})$ is included in~$ \mathbf{L}^2_t(\Gamma_{A})$, see~\cite{BaHa97} for the truncated exterior problem in a smooth domain and~\cite[Remarks 5.1.5 and~5.1.8]{ACL+17} in the general case.
This is the framework of this article.

\smallbreak

We shall also use the basic integration by parts formula between $\mathbf{H}(\curl;\Omega)$ and~$\mathbf{H}^1(\Omega)$:
\begin{eqnarray}
\forall (\boldsymbol{v},\boldsymbol{w}) \in \mathbf{H}(\curl;\Omega) \times \mathbf{H}^1(\Omega),\quad
(\boldsymbol{v}\mid\curl \boldsymbol{w})-(\curl\boldsymbol{v}\mid \boldsymbol{w})=\langle \boldsymbol{v}\times \boldsymbol{n} , \boldsymbol{w} \rangle_{\mathbf{H}^{\frac1{2}}(\Gamma)}.
\label{green rot H1}
\end{eqnarray}

Finally, let us recall some useful subspaces of $\mathbf{H}(\curl;\Omega)$ and $\mathbf{H}(\divg;\Omega)$:
\begin{eqnarray*}
\mathbf{H}(\divg 0; \Omega) &=& \lbrace \boldsymbol{v}\in\mathbf{L}^2(\Omega) : \divg\boldsymbol{v}=0 \rbrace ,
\\
\mathbf{H}_0(\divg 0; \Omega) &=& \mathbf{H}(\divg 0; \Omega)\cap \mathbf{H}_0(\divg ; \Omega) ,
\\
\mathbf{H}_{0,\Gamma_P}(\divg ; \Omega) &=& \lbrace \boldsymbol{v}\in\mathbf{H}(\divg;\Omega) : \boldsymbol{v} \cdot \boldsymbol{n}_{|\Gamma_P}=0  \rbrace ,
\\
\mathbf{H}_{0,\Gamma_P}(\divg 0; \Omega) &=& \mathbf{H}(\divg 0; \Omega)\cap \mathbf{H}_{0,\Gamma_P}(\divg ; \Omega) ,
\\
\mathbf{H}(\curl 0; \Omega) &=& \lbrace \boldsymbol{v}\in\mathbf{L}^2(\Omega) : \curl\boldsymbol{v}=0 \rbrace ,
\\
\mathbf{H}_0(\curl 0; \Omega) &=& \mathbf{H}(\curl 0; \Omega)\cap \mathbf{H}_0(\curl ; \Omega) ,
\\
\quad\mathbf{H}_{0,\Gamma_P}(\curl 0; \Omega) &=& \mathbf{H}(\curl 0; \Omega)\cap \mathbf{H}_{0,\Gamma_P}(\curl ; \Omega) .
\end{eqnarray*}

\section{Well-posedness of the model}
\label{sec-well-posed}
In the whole article, we shall make the following\dots 
\begin{hpthss}
We suppose that there exists strictly positive real numbers $\nu^*$, $\varOmega^*$ and $\omega^*$ such that, for almost all $\boldsymbol{x}\in\Omega$ and for each species $s$ (ions and electrons), one has:
\begin{eqnarray}
0 \leq \nu_s(\boldsymbol{x}) \leq \nu^*,\label{2}\\
 |\varOmega_{cs}(\boldsymbol{x})|\leq \varOmega^*,\label{3}\\
0 < \omega_{ps}(\boldsymbol{x}) \leq \omega^*. \label{4}
\end{eqnarray}
\label{hyp-1}
\end{hpthss}
For $ s\in \{ 1,\ 2\}$ and $\boldsymbol{x} \in \Omega$ fixed, the mapping $\boldsymbol{v}\mapsto \varOmega_{cs}(\boldsymbol{x})\, \boldsymbol{b}(\boldsymbol{x})\times \boldsymbol{v} + \nu_s(\boldsymbol{x})\, \boldsymbol{v}$ defined from $\xR^3$ (or~$\xC^3$) to itself is linear. So, there exists a matrix $\mathbb M_s(\boldsymbol{x})\in \mathcal{M}_{3}(\xR)$ such that: 
\begin{equation}
\varOmega_{cs}(\boldsymbol{x})\, \boldsymbol{b}(\boldsymbol{x})\times \boldsymbol{v}+\nu_s(\boldsymbol{x})\, \boldsymbol{v} = \mathbb M_s(\boldsymbol{x}) \, \boldsymbol{v}, \quad \forall \boldsymbol{v} \in \xR^3 \text{ or } \xC^3.
\end{equation}
We denote by $\III \cdot \III_{\mathcal{M}}$ the operator norm on the space~$\mathcal{M}_{3}(\xC)$ induced by the Hermitian norm of $\xC^3$. 
\begin{prpstn}
There exists $\lambda > 0$ such that $\III \lambda \mathbb{M}_s(\boldsymbol{x}) \III_{\mathcal{M}} < 1$, for all $s\in\lbrace 1,\ 2 \rbrace$ and $\boldsymbol{x}\in \Omega$. Therefore, the matrix $\,\mathbb{I}+\lambda \mathbb{M}_s$ is invertible for all $s\in\lbrace 1,\ 2 \rbrace$ and $\boldsymbol{x}\in \Omega$, where $\mathbb{I}$ is the identity matrix, and its inverse is uniformly bounded on $\Omega$.
\label{lambda inverse} 
\end{prpstn}
\begin{proof}
This is an easy consequence of Hypothesis~\Rref{hyp-1}, by a simple perturbation argument. See~\cite[Propositions~3.1 and 3.2]{LZ+17}.
\end{proof} 
%
\begin{dfntn}
Let $\lambda$ be given by Proposition \Rref{lambda inverse}.  Let 
$\mathbb{D}_{\lambda}:\Omega\longrightarrow \mathcal{M}_{3}(\xR)$ be the matrix 
\begin{eqnarray}
\mathbb{D}_{\lambda}(\boldsymbol{x}) := \sum_s {\omega_{ps}^2(\boldsymbol{x})}(\mathbb{I}+ \lambda \mathbb{M}_s(\boldsymbol{x}))^{-1}, \quad \text{for}\quad \boldsymbol{x}\in\Omega.
\label{matrix D wel pos}
\end{eqnarray}
\end{dfntn}
By convention, sums on the variable $s$ run on all particle species, \ie, from $s=1$ to~$2$ in our model.
\begin{prpstn}
The matrix $\mathbb{D}_{\lambda}(\boldsymbol{x})$ is positive for all $\boldsymbol{x}$ in $\Omega$. Moreover, there exists $\xi > 0$ such that 
\begin{eqnarray*}
\sup_{\boldsymbol{x}\in\Omega} \III \mathbb{D}_{\lambda}(\boldsymbol{x}) \III_{\mathcal{M}} \leq\xi.
\end{eqnarray*}
\label{D lamda positive}
\end{prpstn}
\begin{proof}
Let $\boldsymbol{x}\in\Omega$. To show the positivity of $\mathbb{D}_{\lambda}$, it is enough to prove that the matrix $(\mathbb{I}+ \lambda \mathbb{M}_s(\boldsymbol{x}))^{-1}$ is positive for $s\in\lbrace 1,\ 2 \rbrace$. Given $\boldsymbol{v}\in\xR^3$, we have
\begin{equation*}
(\mathbb{I}+\lambda\, \mathbb{M}_s(\boldsymbol{x})) \boldsymbol{v}  \cdot\boldsymbol{v} 
= \boldsymbol{v} \cdot \boldsymbol{v} +\lambda\,  \mathbb{M}_s(\boldsymbol{x})\boldsymbol{v}\cdot \boldsymbol{v} 
= | \boldsymbol{v} |^2 + \lambda \nu_s (\boldsymbol{x})|\boldsymbol{v}|^2  \stackrel{\eqref{2}}{\geq} | \boldsymbol{v} |^2  \ge 0.
\end{equation*}
Then, $\mathbb{I}+\lambda \mathbb{M}_s(\boldsymbol{x})$ is positive. Next, given $\boldsymbol{w}\in\xR^3$,  there exists $\boldsymbol{\eta} \in \xR^3 $ such that  $ \boldsymbol{w}= (\mathbb{I}+\lambda \mathbb{M}_s(\boldsymbol{x}))\boldsymbol{\eta} $. Hence, it follows that 
\begin{eqnarray*}
(\mathbb{I}+ \lambda\, \mathbb{M}_s(\boldsymbol{x}))^{-1}\boldsymbol{w} \cdot\boldsymbol{w} 
&=&\boldsymbol{\eta} \cdot(\mathbb{I}+ \lambda \mathbb{M}_s(\boldsymbol{x}))\,\boldsymbol{\eta} \geq | \boldsymbol{\eta} |^2  \ge 0.
\end{eqnarray*}
So, the matrix $(\mathbb{I}+ \lambda \mathbb{M}_s(\boldsymbol{x}))^{-1}$ is positive. 
The uniform boundedness of $\mathbb{D}_{\lambda}$ is an easy consequence of Hypothesis~\Rref{hyp-1} and Proposition~\Rref{lambda inverse}.
%
\end{proof}

\medbreak

To prove the well-posedness, we rewrite the system \eqref{uu}--\eqref{init cond 0} with boundary condition \eqref{lk}--\eqref{lm} as the first order evolution equation
\begin{eqnarray}
\left\{\begin{array}{lll}
\partial_{t}\boldsymbol{U}+\xA\,\boldsymbol{U} 
&=& 0,\\ 
\boldsymbol{U}(0)
&=&\boldsymbol{U}{_0},
\end{array}
\right.
\label{pb-evol-abstrait}
\end{eqnarray}
where the vector $\boldsymbol{U}$ is 
$$ \boldsymbol{U} = \left( \boldsymbol{J}_1,\boldsymbol{J}_2,\boldsymbol{E},\boldsymbol{B} \right)^\top \,, $$
and $\xA$ is a linear operator formally given by the expression
\begin{eqnarray}
\xA =
\begin{pmatrix}
\mathbb M_{1} & 0 & -\varepsilon_0\,\omega_{p1}^2\ & 0 \cr
0&\mathbb M_{2} & -\varepsilon_0\,\omega_{p2}^2\ & 0 \cr
\frac1{\varepsilon_0} & \frac1{\varepsilon_0} & 0 &-c^2 \curl \cr
0 & 0 & \curl& 0
\end{pmatrix}.
\label{matrix A}
\end{eqnarray}
The existence and uniqueness of the solution to Problem~\eqref{pb-evol-abstrait} follows from the classical Lumer--Phillips theorem~\cite{P83,CH98}, as we shall se later.

\medbreak

We introduce the weighted $L^2$ spaces associated to each species index~$s$:
\begin{equation}
\mathbf{L}^2_{(s)}(\Omega) := \left\{  
\boldsymbol{w} : \Omega \to \xC \text{ measurable, s.t. } \int_{\Omega} \left| \frac{\boldsymbol{w}}{\omega_{ps}} \right|^2\, \xdif\Omega < + \infty,
\right\}
\end{equation}
\ie, $\boldsymbol{w} \in \mathbf{L}^2_{(s)}(\Omega)$ iff  $\boldsymbol{w}/\omega_{ps} \in \mathbf{L}^2(\Omega)$, endowed with their canonical norm
\begin{equation}
\| \boldsymbol{w} \|_{(s)} := \| \boldsymbol{w} \|_{\mathbf{L}^2_{(s)}(\Omega)} := \left\| \frac{\boldsymbol{w}}{\omega_{ps}} \right\|.
\end{equation}
In view of the bound~\eqref{4}, one immediately deduces a basic useful result:
\begin{lmm}
\label{lem-weight-basic}
For each~$s$:
\begin{enumerate}[label=(\roman*)]
\item The space $\mathbf{L}^2_{(s)}(\Omega)$ is continuously embedded into~$\mathbf{L}^2(\Omega)$. 
\item For any $\boldsymbol{w} \in \mathbf{L}^2(\Omega)$, it holds that $\omega_{ps}^2\, \boldsymbol{w} \in \mathbf{L}^2_{(s)}(\Omega)$.
\end{enumerate}
\end{lmm}
Next, we introduce the energy space 
\begin{equation*}
\mathbf{X}=\mathbf{L}^2_{(1)}(\Omega)\times\mathbf{L}^2_{(2)}(\Omega)\times\mathbf{L}^2(\Omega)\times\mathbf{L}^2(\Omega),
\end{equation*}
and we endow it with the inner product defined for all $\boldsymbol{U}=(\boldsymbol{U}_1,\boldsymbol{U}_2,\boldsymbol{U}_3,\boldsymbol{U}_4)^{\top}$ and $\boldsymbol{V}=(\boldsymbol{V}_1,\boldsymbol{V}_2,\boldsymbol{V}_3,\boldsymbol{V}_4)^{\top}$ by
\begin{equation}
\left(\boldsymbol{U}, \boldsymbol{V}\right)_{\mathbf{X}} := \frac1{\varepsilon_0}\sum_s \left( \frac{\boldsymbol{U}_s}{\omega_{ps}} \Biggm| \frac{\boldsymbol{V}_s}{\omega_{ps}}  \right) + \varepsilon_0\,(\boldsymbol{U}_3 \mid \boldsymbol{V}_3) + c^2 \varepsilon_0\,(\boldsymbol{U}_4 \mid \boldsymbol{V}_4),
\end{equation}
and the associated norm $\|\cdot\|_{\mathbf{X}}$.

\subsection{Perfectly conducting case}
Here, we suppose that $\Gamma_A$ is empty. The domain $\Omega$ is encased in a perfect conductor, which means: 
\begin{equation*}
\forall t>0,\quad \boldsymbol{E}(t) \times \boldsymbol{n} = 0, \quad \boldsymbol{B}(t)\cdot \boldsymbol{n} = 0, \quad \text{on } \Gamma.
\end{equation*}
Now, we define the linear unbounded operator $\xA_1 : D(\xA_1)\subset \mathbf{X}\rightarrow \mathbf{X}$ as 
\begin{eqnarray}
&&D(\xA_1):=\mathbf{L}^2_{(1)}\Omega)\times\mathbf{L}^2_{(2)}(\Omega)\times\mathbf{H}_0(\curl;\Omega)\times\mathbf{H}(\curl;\Omega) ,
\nonumber 
\\
&& \xA_1 \boldsymbol{U}:=\xA\,\boldsymbol{U}, \quad \forall \boldsymbol{U}\in D(\xA_1).
\label{A1 definition}
\end{eqnarray}
The fact that $\xim(\xA_1)\subset \mathbf{X}$ follows from Proposition~\Rref{lambda inverse} and Lemma~\Rref{lem-weight-basic}. The abstract evolution equation \eqref{pb-evol-abstrait} writes:
\begin{equation}
\partial_{t}\boldsymbol{U}(t) + \xA_1\boldsymbol{U}(t)=0,\, \text{ for }  t > 0,  \qquad  \boldsymbol{U}(0)=\boldsymbol{U}_0.
\label{evol prob cp}
\end{equation} 
\begin{prpstn}
The operator $\xA_1$ is maximal monotone. 
\label{maximonot 1}
\end{prpstn}
\begin{proof}
First, we check that  $\xA_1$ is monotone. Given $\boldsymbol{U}=\left(\boldsymbol{U}_1, \boldsymbol{U}_2, \boldsymbol{U}_3, \boldsymbol{U}_4 \right)^\top \in D(\xA_1)$, one finds, by the definition of $\xA$, 
\begin{equation*}
(\xA\, \boldsymbol{U}, \boldsymbol{U})_{\mathbf{X}}
= \frac{1}{\varepsilon_0}\sum_s \left( \frac{\mathbb{M}_s\boldsymbol{U}_s}{\omega_{ps}} \Biggm| \frac{\boldsymbol{U}_s}{\omega_{ps}} \right) - \sum_s(\boldsymbol{U}_3\mid\boldsymbol{U}_s) + \sum_s(\boldsymbol{U}_s\mid\boldsymbol{U}_3) 
 -{\varepsilon_0}({c^2\, \curl\boldsymbol{U}_4}\mid \boldsymbol{U}_3) + {\varepsilon_0}{c^2\,}({ \curl\boldsymbol{U}_3}\mid \boldsymbol{U}_4 ).
\end{equation*}
By Lemma~\Rref{lem-weight-basic}, $\boldsymbol{U}_s \in \mathbf{L}^2(\Omega)$ for $s=1,\ 2$ too.
Taking the real part of this inner product, one gets:
\begin{equation}
\Re ( \xA\, \boldsymbol{U}, \boldsymbol{U} )_{\mathbf{X}}= \frac{1}{\varepsilon_0}\,\sum_s \Re \left( \frac{\mathbb{M}_s\boldsymbol{U}_s}{\omega_{ps}} \Biggm| \frac{\boldsymbol{U}_s}{\omega_{ps}} \right)
-{\varepsilon_0}c^2\, \Re \left[ ({\curl\boldsymbol{U}_4}\mid \boldsymbol{U}_3)-({\curl\boldsymbol{U}_3}\mid \boldsymbol{U}_4) \right].
\label{real part A1u.u} 
\end{equation} 
But, for all $s=1,\ 2$, on has, according to the definition of $\mathbb{M}_s$,
\begin{equation}
\Re \left( \frac{\mathbb{M}_s\boldsymbol{U}_s}{\omega_{ps}} \Biggm| \frac{\boldsymbol{U}_s}{\omega_{ps}} \right)
= \Re \left(\frac{\nu_s\,\boldsymbol{U}_s}{\omega_{ps}} \Biggm| \frac{\boldsymbol{U}_s}{\omega_{ps}} \right) + \Re \left( \varOmega_{cs}\, \boldsymbol{b} \times \frac{\boldsymbol{U}_s}{\omega_{ps}} \Biggm| \frac{\boldsymbol{U}_s}{\omega_{ps}} \right) 
= \left( \frac{\nu_s\,\boldsymbol{U}_s}{\omega_{ps}} \Biggm| \frac{\boldsymbol{U}_s}{\omega_{ps}} \right).
\label{real part Msu.u}
\end{equation}
Thus, plugging \eqref{real part Msu.u} into~\eqref{real part A1u.u} and  using Green's formula~\eqref{green rot rot}, the boundary condition $\boldsymbol{U}_3\times\boldsymbol{n}=0$ on $\Gamma$ and the condition~\eqref{2}, one obtains
\begin{equation}
\Re \left({\xA_1}{\boldsymbol{U}}, \boldsymbol{U}\right)_{\mathbf{X}}=
\frac{1}{\varepsilon_0}\,\sum_s \left( \frac{\nu_s\,\boldsymbol{U}_s}{\omega_{ps}} \Biggm| \frac{\boldsymbol{U}_s}{\omega_{ps}} \right) \geq 0.
\label{resul monot A1}
\end{equation} 
Hence the monotonicity of $\xA_1$.

\medbreak

Let us proceed to the maximal character. Let  $\lambda > 0$ given by Proposition~\Rref{lambda inverse}. Given any $\boldsymbol{F}=\left(\boldsymbol{F}_1, \boldsymbol{F}_2, \boldsymbol{F}_3, \boldsymbol{F}_4 \right)^\top \in \mathbf{X}$, we look for $\boldsymbol{U}=\left(\boldsymbol{U}_1, \boldsymbol{U}_2, \boldsymbol{U}_3, \boldsymbol{U}_4 \right)^\top \in D(\xA_1)$ such that $(\mathbb{I}+\lambda\xA_1)\, \boldsymbol{U}=\boldsymbol{F}$.
More explicitly, this equation writes:
\begin{eqnarray}
\boldsymbol{U}_1+\lambda\, \mathbb M_{1}\boldsymbol{U}_1 - \lambda {{\varepsilon_0}\omega_{p1}^2}\, \boldsymbol{U}_3 &=&\boldsymbol{F}_1 , 
 \label{u+lAu-1}\\
\boldsymbol{U}_2+\lambda\,  \mathbb M_{2}\boldsymbol{U}_2 - \lambda {{\varepsilon_0}\omega_{p2}^2}\, \boldsymbol{U}_3 &=& \boldsymbol{F}_2 , \label{u+lAu-2}\\
 \boldsymbol{U}_3+ \frac{\lambda}{\varepsilon_0}\, \boldsymbol{U}_1 + \frac{\lambda}{\varepsilon_0}\, \boldsymbol{U}_2 - \lambda c^2\, \curl\boldsymbol{U}_4 &=& \boldsymbol{F}_3 ,
\label{u+lAu-3}\\
 \boldsymbol{U}_4+\lambda\, \curl\boldsymbol{U}_3 &=& \boldsymbol{F}_4 . \label{u+lAu-4}
\end{eqnarray}
Assuming that a solution $\boldsymbol{U}$ of \eqref{u+lAu-1}--\eqref{u+lAu-4} exists, we can eliminate the equations \eqref{u+lAu-1}, \eqref{u+lAu-2} and \eqref{u+lAu-4} respectively:
\begin{eqnarray}
\boldsymbol{U}_1
&=&(\mathbb{I}+ \lambda \mathbb M_{1})^{-1}(\boldsymbol{F}_1+\lambda \varepsilon_0\,\omega_{p1}^2\,\boldsymbol{U}_3),
\label{ua well posedness cp}\\ 
\boldsymbol{U}_2 
&=&
(\mathbb I + \lambda \mathbb M_{2})^{-1}(\boldsymbol{F}_2+\lambda \varepsilon_0\,\omega_{p2}^2\,\boldsymbol{U}_3),
\label{ub well posedness cp}\\ 
\boldsymbol{U}_4
&=& \boldsymbol{F}_4 -\lambda\, \curl\boldsymbol{U}_3.
\label{ud well posedness cp}
\end{eqnarray} 
Inserting these three expressions into \eqref{u+lAu-3}, one obtains, in function of $\boldsymbol{U}_3$,  
\begin{eqnarray}
\boldsymbol{U}_3 + \lambda^2{c^2} \curl\curl\boldsymbol{U}_3 + \lambda^2\, \mathbb{D}_\lambda\, \boldsymbol{U}_3
= \boldsymbol{F}_3 + \lambda c^2 \curl \boldsymbol{F}_4 -\frac{\lambda}{\varepsilon_0}\,\sum_s\,(\mathbb{I} + \lambda \mathbb{M}_{s})^{-1}\boldsymbol{F}_s. 
\label{u_3-E-1}
\end{eqnarray}
Multiplying this identity by a test-function $\boldsymbol{v}\in \mathbf {H}_0(\curl;\Omega)$ and applying the Green formula~\eqref{green rot rot}, one finds the following variational formulation:\\
\emph{Find $ \boldsymbol{U}_3\in\mathbf{H}_0(\curl;\Omega)$ such that}
\begin{eqnarray}
\qquad \qquad a(\boldsymbol{U}_3,\boldsymbol{v})=L(\boldsymbol{v}), \quad \forall \boldsymbol{v}\in\mathbf{H}_0(\curl;\Omega)  
\label{varia-E-1}
\end{eqnarray} 
where the forms $a$ and $L$ are defined on $\mathbf{H}(\curl;\Omega)$ as:
\begin{eqnarray}
a(\boldsymbol{w},\boldsymbol{v})&:=&\left(\boldsymbol{w}\mid \boldsymbol{v}\right)+\lambda^2{c^2}\left({\curl\boldsymbol{w}}\mid {\curl \boldsymbol{v}} \right)+{\lambda}^2\left(\mathbb{D}_{\lambda}\boldsymbol{w}\mid \boldsymbol{v}\right),
\label{a coer well posi cp}
\\
L(\boldsymbol{v})&:=&\left(\boldsymbol{F}_3\mid \boldsymbol{v} \right)+{\lambda}{c^2}\left(\boldsymbol{F}_4 \mid 
{\curl \boldsymbol{v}} \right)
-\frac{\lambda}{\varepsilon_0}\sum_s\left((\mathbb I + \lambda \mathbb M_{s})^{-1}\boldsymbol{F}_s \mid \boldsymbol{v}\right).
\label{l coer well posi cp}
\end{eqnarray}
The problem \eqref{varia-E-1} is well-posed. Indeed, thanks to Proposition \Rref{D lamda positive},  the sesquilinear form $a$ is continous and coercive on $\mathbf{H}_0(\curl;\Omega)$. The form $L$ is anti-linear, and by Proposition~\Rref{lambda inverse} and Lemma~\Rref{lem-weight-basic}, it obviously continuous on $\mathbf{H}_0(\curl;\Omega)$. Then, we conclude by the Lax--Milgram theorem the existence of a unique solution $\boldsymbol{U}_3\in\mathbf{H}_0(\curl;\Omega)$ to the formulation~\eqref{varia-E-1}.

\medbreak

Returning to the problem \eqref{u+lAu-1}--\eqref{u+lAu-4}, we define $\boldsymbol{U}_1$ and $\boldsymbol{U}_2$ by \eqref{ua well posedness cp} and~\eqref{ub well posedness cp}. Again, by Proposition~\Rref{lambda inverse} and Lemma~\Rref{lem-weight-basic}, they respectively belong to $\mathbf{L}^2_{(1)}(\Omega)$ and~$\mathbf{L}^2_{(2)}(\Omega)$. Also, we define $\boldsymbol{U}_4$ by \eqref{ud well posedness cp}; it belongs to~$\mathbf{L}^2(\Omega)$. Next, if we take  $\boldsymbol{v}\in\boldsymbol{\mathcal{D}}(\Omega)$ as a test function in the formulation \eqref{varia-E-1} and  use the Green formula~\eqref{green rot rot}, we obtain Equation~\eqref{u_3-E-1}. So, by the definition of~$\boldsymbol{U}_4$, we can write this equation as 
\begin{equation}
\boldsymbol{U}_3-\lambda {c^2}\curl\boldsymbol{U}_4+{\lambda}^2\mathbb{D}_{\lambda}\boldsymbol{U}_3
= \boldsymbol{F}_3 - \frac{\lambda}{\varepsilon_0}\,\sum_s\,(\mathbb{I} + \lambda \mathbb{M}_{s})^{-1}\boldsymbol{F}_s. 
\label{u_3-E-11...}
\end{equation}
This equation, on the one hand, implies that $\curl\boldsymbol{U}_4\in\mathbf{L}^2(\Omega)$, and on the other hand is equivalent to \eqref{u+lAu-3} (just replace  $\mathbb{D}_{\lambda}$ with its expression). Hence, the quadruple $(\boldsymbol{U}_1,\boldsymbol{U}_2,\boldsymbol{U}_3,\boldsymbol{U}_4)$ belongs to $D(\xA_1)$ and it solves the equations \eqref{u+lAu-1}--\eqref{u+lAu-4}. The proof is  completed.  
\end{proof}
\begin{thrm}
The operator $-\xA_{1}$ generates a $\xCzero$-semigroup of contractions  $\,(\mathit{T}_1(t))_{t\geq0}\,$ on the energy space $\mathbf{X}$. Thus, for all $\boldsymbol{U}_0 \in \mathbf{X}$, there exists a unique solution $\boldsymbol{U} \in \xCzero(\xR_{\ge0}; \mathbf{X})$, given by $\,\boldsymbol{U}(t)=\mathit{T}_{1}(t)\,\boldsymbol{U}_0$, which solves the problem~\eqref{evol prob cp}. 
\smallbreak
Moreover, if $\boldsymbol{U}_0 \in D(\xA_{1})$, then 
\begin{eqnarray*}
\boldsymbol{U} \in \xCone(\xR_{\ge0}; \mathbf{X})\cap \xCzero(\xR_{\ge0};D(\xA_{1})).
\end{eqnarray*}
Furthermore, we have 
$\| \boldsymbol{U} (t) \|_{\mathbf{X}}\leq \| \boldsymbol{U}_0 \|_{\mathbf{X}}$  and  $\| \partial_t\boldsymbol{U}(t) \|_{\mathbf{X}}\leq \| \xA_{1}\boldsymbol{U}_0 \|_{\mathbf{X}}.$
\label{thm well posed cp}
\end{thrm}
\begin{proof}
From the previous proposition, the operator $-\xA_1$ is maximal dissipative. The domain $D(\xA_1)$ of $-\xA_1$ is dense in $\mathbf{X}$ according to \cite[Theorem~1.4.6]{P83}. Then, we can apply the Lumer--Phillips theorem (see  \cite[Theorem~1.4.3]{P83}) to obtain the result. 
\end{proof}

\subsection{Silver--M\"uller case}
Now, we assume that $\Gamma_A$ is non-empty. A Silver--M\"uller boundary condition holds on the antenna $\Gamma_A$, and  a perfect conductor boundary condition on the rest of the boundary~$\Gamma_P$:
\begin{equation}
\left\{\begin{array}{lll}
\boldsymbol{E}(t) \times \boldsymbol{n} &=&0 \quad \text{on } \Gamma_P,\ t>0,\\
\boldsymbol{E}(t) \times \boldsymbol{n}+c\boldsymbol{B}_\top(t) &=&\boldsymbol{g}(t) \quad \text{on } \Gamma_A,\   t>0.
\end{array}
\right.
\label{cond}
\end{equation}
Our goal is to solve Problem~\eqref{pb-evol-abstrait} with the boundary condition~\eqref{cond}. First, we will start with the homogeneous (or absorbing, outgoing wave) case, $\boldsymbol{g}=0$, and next we proceed to the general (incoming wave) case, \ie, $\boldsymbol{g}\neq0$.

\subsubsection{Homogeneous (absorbing) boundary condition}
Let us define the linear unbounded operator $\xA_2:D(\xA_2)\subset \mathbf{X}\rightarrow \mathbf{X}$ as 
\begin{equation*}
D(\xA_2):=\mathbf{L}^2_{(1)}(\Omega)\times\mathbf{L}^2_{(2)}(\Omega)\times \mathcal{H},
\end{equation*}
where
\begin{equation*}
\mathcal{H}= \{(\boldsymbol{V}_3, \boldsymbol{V}_4)\in \mathbf{H}_{0,\Gamma_P}(\curl; \Omega)\times \mathbf{H}(\curl;\Omega) : \boldsymbol{V}_3 \times \boldsymbol{n} +c\,{\boldsymbol{V}_4}_\top=0 \text{ on } \Gamma_A \},
\end{equation*}
and
\begin{equation}
\xA_2\boldsymbol{U}:=\xA\, \boldsymbol{U}, \quad \forall \boldsymbol{U}\in D(\xA_2).
\label{A2 definition}
\end{equation}
Therefore, the abstract evolution equation \eqref{pb-evol-abstrait} writes:
\begin{equation}
\partial_{t}\boldsymbol{U}(t) + \xA_2\boldsymbol{U}(t)=0,\, \text{ for }  t > 0,  \qquad  \boldsymbol{U}(0)=\boldsymbol{U}_0.
\label{evol prob sm homo}
\end{equation} 
We shall need the following Hilbert space
\begin{equation}
\mathcal{V}:= \{ \boldsymbol{v} \in \mathbf{H}_{0,\Gamma_P}(\curl;\Omega):  \boldsymbol{v}_\top\in\mathbf{L}^2(\Gamma_{A})\}
\label{def-space-V}
\end{equation}
equipped with the inner product
\begin{eqnarray}
(\boldsymbol{w} , \boldsymbol{v})_{\mathcal {V}}:=(\boldsymbol{w}\mid\boldsymbol{v})+ (\curl\boldsymbol{w}\mid\curl\boldsymbol{v})+(\boldsymbol{w}_\top\mid\boldsymbol{v}_\top)_{\Gamma_A}.
\label{inner product V} 
\end{eqnarray}
Above, $(\cdot \mid \cdot)_{\Gamma_A}$ denotes the inner product in $\mathbf{L}^2(\Gamma_A)$.

\begin{prpstn}
The operator $\xA_2$ is maximal monotone. 
\label{maximonot 2}
\end{prpstn}
\begin{proof}
Let us start by proving the monotonicity of $\xA_2$. Pick any $\boldsymbol{U}=\left( \boldsymbol{U}_1, \boldsymbol{U}_2, \boldsymbol{U}_3, \boldsymbol{U}_4 \right)^\top$ in $D(\xA_2)$. 
The equality~\eqref{real part A1u.u} still holds; it only relies on the expression of~$\xA$ in~$\Omega$, not on the boundary conditions.
As $\boldsymbol{U}_3\in\mathbf{H}_{0,\Gamma_P}(\curl;\Omega)$, we get by the integration-by-parts formula~\eqref{green rot 0 Gamma p}:
\begin{eqnarray*}
\Re[({\curl\boldsymbol{U}_4}\mid \boldsymbol{U}_3)-({\curl\boldsymbol{U}_3}\mid \boldsymbol{U}_4)] = \Re({}_{\gamma^0_A}\langle \boldsymbol{U}_3 \times \boldsymbol{n}, \boldsymbol{U}_{4\top} \rangle_{\pi_A}).
\end{eqnarray*}
Now, we use the Silver--M\"uller boundary condition $\boldsymbol{U}_3 \times \boldsymbol{n} + c\, \boldsymbol{U}_{4\top}=0 $ on $ \Gamma_{A}$. We remark that both $\boldsymbol{U}_3 \times \boldsymbol{n}$ and $\boldsymbol{U}_{4\top}$  belong to $ \widetilde{\mathbf{TT}}(\Gamma_{A})\cap \mathbf{TC}(\Gamma_{A}) \subset \mathbf{L}^2_t(\Gamma_{A})$ as said in~\S\Rref{sec-prelim}.
Then, it holds that 
\begin{eqnarray*}
\Re({}_{\gamma^0_A}\langle \boldsymbol{U}_3 \times \boldsymbol{n}, \boldsymbol{U}_{4\top} \rangle_{\pi_A}) = -c\, \| \boldsymbol{U}_{4\top} \|^2_{\mathbf{L}^2(\Gamma_A)}.
\end{eqnarray*}
We thus conclude from \eqref{real part A1u.u} that 
\begin{eqnarray}
\Re ({\xA_2}{\boldsymbol{U}}, \boldsymbol{U} )_{\mathbf{X}}= \frac{1}{\varepsilon_0}\,\sum_s  \left( \frac{\nu_s\, \boldsymbol{U}_s}{\omega_{ps}} \Biggm| \frac{\boldsymbol{U}_s}{\omega_{ps}} \right) + {\varepsilon_0} c^3\, \| \boldsymbol{U}_{4\top} \|^2_{\mathbf{L}^2(\Gamma_A)} \geq 0,
\label{monotonie A2 sm homo}
\end{eqnarray}
which yields the monotonicity of the operator ~$\xA_2$.

\medbreak

Now we show the maximality of~$\xA_2$. Again, we use the same $\lambda > 0$ given by Proposition \Rref{lambda inverse}. 
Given any $\boldsymbol{F}=\left(\boldsymbol{F}_1, \boldsymbol{F}_2, \boldsymbol{F}_3, \boldsymbol{F}_4 \right)^\top \in \mathbf{X}$, we look for $\boldsymbol{U}=\left(\boldsymbol{U}_1, \boldsymbol{U}_2, \boldsymbol{U}_3, \boldsymbol{U}_4 \right)^\top \in D(\xA_2)$ such that $(\mathbb{I}+\lambda\xA_2)\,\boldsymbol{U}=\boldsymbol{F}$,
which is equivalent to the system \eqref{u+lAu-1}--\eqref{u+lAu-4} (plus boundary conditions). Following the same argument as in Proposition~\Rref{maximonot 1}, we can eliminate $\boldsymbol{U}_1$, $\boldsymbol{U}_2$ and~$\boldsymbol{U}_4$, and they are given respectively by \eqref{ua well posedness cp}, \eqref{ub well posedness cp} and \eqref{ud well posedness cp}, while $\boldsymbol{U}_3$ verifies the equation:
\begin{eqnarray}
\boldsymbol{U}_3+\lambda^2{c^2}\curl {\curl\boldsymbol{U}_3}+{\lambda}^2\mathbb{D}_{\lambda}\boldsymbol{U}_3
= \boldsymbol{F}_3 +{\lambda}{c^2}\curl \boldsymbol{F}_4 -\frac{\lambda}{\varepsilon_0}\,\sum_s\,(\mathbb{I} + \lambda \mathbb{M}_{s})^{-1}\boldsymbol{F}_s. 
\label{u_3-E-1 sm}
\end{eqnarray}
Thus,  multiplying \eqref{u_3-E-1 sm} by a test-function $\boldsymbol{v}\in\mathcal{V}$, applying Green's formula~\eqref{green rot 0 Gamma p}, and using the  Silver--M\"uller boundary condition and the expression~\eqref{ud well posedness cp}, we arrive at the variational formulation:\\ 
\emph{Find $\boldsymbol{U}_3\in\mathcal{V}$ such that}
\begin{eqnarray}
\qquad  \widetilde{a}(\boldsymbol{U}_3,\boldsymbol{v})=L(\boldsymbol{v}), \quad \forall \boldsymbol{v}\in\mathcal{V} 
\label{varia-E-1 sm}
\end{eqnarray} 
with the sesquilinear form $\widetilde{a}$ defined as:
\begin{eqnarray}
\qquad \widetilde{a}(\boldsymbol{w},\boldsymbol{v}):= a(\boldsymbol{w},\boldsymbol{v})+ \lambda c \left(\boldsymbol{w}_\top\mid \boldsymbol{v}_\top\right)_{\Gamma_{A}},
\label{a coer well posi sm}
\end{eqnarray}
and the forms $a$ and $L$ given respectively by \eqref{a coer well posi cp} and \eqref{l coer well posi cp}.

\medbreak
 
As the form $a$ is coercive on $\mathbf{H}(\curl;\Omega)$ (Proposition~\Rref{maximonot 1}), the form $\widetilde{a}$ is coercive on~$\mathcal{V}$. 
So, by Lax--Milgram theorem, Problem~\eqref{varia-E-1 sm} admits a unique solution $\boldsymbol{U}_3\in \mathcal{V}$. 
Defining $\boldsymbol{U}_1$, $\boldsymbol{U}_2$ and $\boldsymbol{U}_4$ respectively by \eqref{ua well posedness cp}, \eqref{ub well posedness cp} and \eqref{ud well posedness cp}, they respectively belong to $\mathbf{L}^2_{(1)}(\Omega)$, $\mathbf{L}^2_{(2)}(\Omega)$ and~$\mathbf{L}^2(\Omega)$; taking a test-function $\boldsymbol{v}\in \boldsymbol{\mathcal{D}}(\Omega)$ in \eqref{varia-E-1 sm}, we find the equation \eqref{u_3-E-1 sm} which is equivalent to \eqref{u+lAu-3}. Thus $\boldsymbol{U}=(\boldsymbol{U}_1,\boldsymbol{U}_2,\boldsymbol{U}_3,\boldsymbol{U}_4)$ formally satisfies $(\mathbb{I}+\lambda\xA)\,\boldsymbol{U}=\boldsymbol{F}$, and in order to prove that $D(\xA_2)$ it remains only to check the homogeneous Silver--M\"uller boundary condition on $\Gamma_A$. To this end, using the integration-by-parts formula \eqref{green rot 0 Gamma p} in~\eqref{varia-E-1 sm} and the definition of~$\boldsymbol{U}_4$, it follows from the identity~\eqref{u_3-E-1 sm} that:
\begin{equation} 
\lambda c \left(\boldsymbol{U}_{3\top}\mid \boldsymbol{v}_\top\right)_{\Gamma_{A}}-\lambda c^2 {}_{\gamma_{A}}\langle \boldsymbol{U}_4 \times \boldsymbol{n}, \boldsymbol{v}_\top \rangle_{\pi^0_{A}}=0, \quad \forall \boldsymbol{v} \in \mathcal{V}.
\label{int}
\end{equation}
Let $ \boldsymbol{\omega} \in \widetilde{\mathbf{H}}^{\frac1{2}}(\Gamma_A)$ and $\widetilde{\boldsymbol{\omega}} \in \mathbf{H}^\frac1{2}(\Gamma)$ be its extension by~$0$ to the whole boundary. By the surjectivity of the trace mapping, there exists $ \boldsymbol{v} \in \mathbf{H}^1(\Omega)$ such that $\widetilde{\boldsymbol{\omega}} =\boldsymbol{v}_{|_\Gamma}$; clearly, $\boldsymbol{v} \in \mathcal{V}$. From the integration by parts formulas~\eqref{green rot 0 Gamma p},~\eqref{green rot H1} applied to $\boldsymbol{U}_4$ and~$\boldsymbol{v}$, we get that 
\begin{equation} 
{}_{\gamma_{A}}\langle \boldsymbol{U}_4 \times \boldsymbol{n}, \boldsymbol{v}_\top \rangle_{\pi^0_{A}}=\langle \boldsymbol{U}_4 \times \boldsymbol{n}, \boldsymbol{v} \rangle_{\widetilde{\mathbf{H}}^{\frac1{2}}(\Gamma_A)}.
\label{integ by part 2 form}
\end{equation} 
Recalling  that   $ \boldsymbol{v}_{|_\Gamma}=\boldsymbol{v}_\top+ (\boldsymbol{n} \cdot\boldsymbol{v})\,\boldsymbol{n}$,  it follows that  $\left(\boldsymbol{U}_{3\top}\mid \boldsymbol{v}_\top \right)_{\Gamma_{A}}=\left(\boldsymbol{U}_{3\top}\mid \boldsymbol{v} \right)_{\Gamma_{A}} $. 
Using \eqref{integ by part 2 form} and the previous identity, Equation~\eqref{int} becomes: 
\begin{equation}
\left(\boldsymbol{U}_{3\top}\mid \boldsymbol{v} \right)_{\Gamma_{A}}- c\, \langle \boldsymbol{U}_4\times\boldsymbol{n}, \boldsymbol{v} \rangle_{\widetilde{\mathbf{H}}^{\frac1{2}}(\Gamma_A)}=0.
\label{rrh}
\end{equation}
As $\boldsymbol{v}_{|_{\Gamma_A}}={\boldsymbol{\omega}}$ and ${\boldsymbol{\omega}}$ is arbitrary in $\widetilde{\mathbf{H}}^{\frac1{2}}(\Gamma_A)$, one concludes from~\eqref{rrh} that $\boldsymbol{U}_{3\top}-  c\, \boldsymbol{U}_4 \times \boldsymbol{n}=0\,$ in $\,\widetilde{\mathbf{H}}^{-{\frac1{2}}}(\Gamma_A) $ which is equivalent to  $\boldsymbol{U}_3 \times \boldsymbol{n} + c\, \boldsymbol{U}_{4\top}=0$ in $\widetilde{\mathbf{H}}^{-{\frac1{2}}}(\Gamma_A) $, and therefore also in 
$\mathbf{L^2}(\Gamma_A)$ because $\boldsymbol{U}_{3\top}$ is in $\mathbf{L}^2(\Gamma_A)$ (plus a density argument). This proves the Silver--M\"uller boundary condition, and the proof of the Proposition is complete.
\end{proof}
\begin{thrm}
The operator $-\xA_{2}$  generates a $\xCzero$-semigroup of contractions  $(\mathit{T}_2(t))_{t\geq0}\,$ on the energy space $\mathbf{X}$. Thus, for all $\boldsymbol{U}_0 \in \mathbf{X}$, there exists a unique solution $\boldsymbol{U} \in \xCzero(\xR_{\ge0}; \mathbf{X})$, given by $\boldsymbol{U}(t)=\mathit{T}_{2}(t)\boldsymbol{U}_0$,  which solves the problem~\eqref{evol prob sm homo}. 
\smallbreak
Moreover, if $\boldsymbol{U}_0 \in D(\xA_{2})$, then 
\begin{equation*}
\boldsymbol{U} \in \xCone(\xR_{\ge0}; \mathbf{X})\cap \xCzero(\xR_{\ge0};D(\xA_{2})).
\end{equation*}
Furthermore, we have 
$\| \boldsymbol{U} (t) \|_{\mathbf{X}}\leq \| \boldsymbol{U}_0 \|_{\mathbf{X}}$  and  $\| \partial_t\boldsymbol{U}(t) \|_{\mathbf{X}}\leq \| \xA_{2}\boldsymbol{U}_0 \|_{\mathbf{X}}.$
\label{bonne posi sm homo}
\end{thrm}
\begin{proof}
Entirely similar to Theorem~\Rref{thm well posed cp}.
\end{proof}

\subsubsection{General (non-homogeneous) boundary condition}
\label{sm-nonhomo}
Here, we suppose that $\boldsymbol{g}\neq 0$ in~\eqref{cond}. We shall solve the evolution problem by using a lifting of the boundary data~$\boldsymbol{g}$. To this end, we introduce the mapping:
\begin{eqnarray*}
Z_A \,:\, \mathbf{H}_{0,\Gamma_P}(\mathbf {\curl}; \Omega)\times \mathbf{H}(\mathbf {\curl};\Omega)&\rightarrow & \widetilde{\mathbf{TT}}(\Gamma_{A})+{\mathbf{TC}}(\Gamma_{A})\\
 (\boldsymbol{v},\boldsymbol{w}) &\mapsto &  \gamma_\top(\boldsymbol{v})+ c\pi_\top(\boldsymbol{w}).
\end{eqnarray*}
It is clear that $Z_A$ is linear and continuous, and due to the surjectivity of $\gamma_\top$ and~$\pi_\top$ (see the definition of $\widetilde{\mathbf{TT}}$ and~${\mathbf{TC}}$), $Z_A$~is also surjective. Then, we deduce that $Z_A$ is bijective from $(\ker Z_A)^{\perp}$ to $\widetilde{\mathbf{TT}}(\Gamma_{A})+{\mathbf{TC}}(\Gamma_{A})$ and we denote its inverse by~$R_A$. By the Banach--Schauder theorem, $R_A$~is continuous. 

\medbreak

We assume the following regularity on the boundary data~$\boldsymbol{g}$:
\begin{eqnarray}
\boldsymbol{g} \in \xWn{{2,1}}(\xR_{>0}; \widetilde{\mathbf{TT}}(\Gamma_{A}) + \mathbf{TC}(\Gamma_{A})).
\label{regu g sm}
\end{eqnarray}
According to the previous paragraph, for any $t\ge0$ there exists $(\boldsymbol{g}_3(t),\boldsymbol{g}_4(t))\in \mathbf{H}_{0,\Gamma_P}(\curl;\Omega)\times \mathbf{H}(\curl;\Omega)$ such that  
\begin{equation}
(\boldsymbol{g}_3(t),\boldsymbol{g}_4(t)) = R_A[\boldsymbol{g}(t)], \quad \text{\ie,}\quad \boldsymbol{g}_3(t)\times \boldsymbol{n} + c\,\boldsymbol{g}_{4\top}(t)=\boldsymbol{g}(t) \text{ on } \Gamma_A, 
\label{def g3 g4}
\end{equation}
and the functions $(\boldsymbol{g}_3,\boldsymbol{g}_4)$ have the following regularity:
\begin{equation}
(\boldsymbol{g}_3,\boldsymbol{g}_4)\in \xWn{{2,1}}(\xR_{>0};\mathbf{H}_{0,\Gamma_P}(\mathbf {\curl}; \Omega) \times \mathbf{H}(\mathbf {\curl}; \Omega)).
\label{re g3 g4}
\end{equation} 
\begin{thrm}
Suppose that the initial data satisfy:
\begin{eqnarray}
\left\lbrace \begin{array}{lll}
\boldsymbol{J}_{1,0}\in \mathbf{L}^2_{(1)}(\Omega),\quad  \boldsymbol{J}_{2,0} \in \mathbf{L}^2_{(2)}(\Omega),\\
\boldsymbol{E}{_0}\in\mathbf{H}_{0,\Gamma_P}(\curl;\Omega),\ \boldsymbol{B}{_0} \in  \mathbf {H}(\curl;\Omega),\\
\boldsymbol{E}{_0} \times \boldsymbol{n} +  c\, \boldsymbol{B}{_{0\top}} = \boldsymbol{g}(0)\quad on ~ \Gamma_{A},
\end{array} \right.
\label{initial cond nonho}
\end{eqnarray}
and $\boldsymbol{g}$ is of regularity~\eqref{regu g sm}, which gives a meaning to the initial value~$\boldsymbol{g}(0)$.
Then, there exists one, and only one, solution  $\boldsymbol{U}=(\boldsymbol{J}_1,\boldsymbol{J}_2,\boldsymbol{E},\boldsymbol{B})^\top$ to the evolution problem~\eqref{pb-evol-abstrait}, supplemented with the boundary conditions~\eqref{cond}, such that its components have the following regularity:
\begin{eqnarray*}
(\boldsymbol{J}_1, \boldsymbol{J}_2) &\in& \xCone(\xR_{\ge0};\mathbf{L}^2_{(1)}(\Omega) \times \mathbf{L}^2_{(2)}(\Omega)),\\
\boldsymbol{E} &\in & \xCone(\xR_{\ge0}; \mathbf{L}^2(\Omega))\cap \xCzero(\xR_{\ge0}; \mathbf {H}_{0,\Gamma_P}(\curl;\Omega)),\\
\boldsymbol{B} & \in & \xCone(\xR_{\ge0}; \mathbf{L}^2(\Omega))\cap \xCzero(\xR_{\ge0}; \mathbf {H}(\curl;\Omega)).
\end{eqnarray*}
\label{non-homo}
\end{thrm}
\begin{proof}
Define $(\boldsymbol{g}_3,\boldsymbol{g}_4)$ by~\eqref{def g3 g4}, and introduce the auxiliary unknown $\boldsymbol{U}^\star=(\boldsymbol{J}^\star_1,\boldsymbol{J}^\star_2,\boldsymbol{E}^\star,\boldsymbol{B}^\star)^\top$: 
\begin{equation*}
\boldsymbol{J}^\star_1 = \boldsymbol{J}_1,\quad \boldsymbol{J}^\star_2 = \boldsymbol{J}_2,\quad \boldsymbol{E}^\star = \boldsymbol{E}-\boldsymbol{g}_3,\quad \boldsymbol{B}^\star
= \boldsymbol{B}-\boldsymbol{g}_4.
\end{equation*}
If a solution $(\boldsymbol{J}_1,\boldsymbol{J}_2,\boldsymbol{E},\boldsymbol{B})^\top$ as above exists then, by construction, the field $\boldsymbol{U}^\star(t)$ belongs to~$D(\xA_2)$ for any~$t\geq 0$. 
Moreover it is governed by the evolution equations
\begin{eqnarray}
\partial_{t} \boldsymbol{U}^\star  + \xA\, \boldsymbol{U}^\star &=& \boldsymbol{F},\quad t>0, 
\label{UU*}\\
\boldsymbol{U}^\star(0) &=& \boldsymbol{U}^\star_0,
\label{UU*0}
\end{eqnarray}
with data 
\begin{eqnarray*}
\boldsymbol{F}=\begin{pmatrix}
\varepsilon_0\,\omega_{p1}^2\,\boldsymbol{g}_3
\cr
\varepsilon_0\,\omega_{p2}^2\,\boldsymbol{g}_3
\cr
-\partial_t\boldsymbol{g}_3+ c^2 \curl \boldsymbol{g}_4
\cr
-\partial_t \boldsymbol{g}_4 - \curl \boldsymbol{g}_3
\end{pmatrix},\qquad 
\boldsymbol{U}^\star_0=\begin{pmatrix}  
\boldsymbol{J}_{1,0}
\cr
\boldsymbol{J}_{2,0}
\cr
\boldsymbol{E}_{0}-\boldsymbol{g}_3(0)\cr
\boldsymbol{B}_{0}-\boldsymbol{g}_4(0)
\end{pmatrix}.
\end{eqnarray*}
Thanks to~\eqref{re g3 g4} and Lemma~\Rref{lem-weight-basic}, one sees that $\boldsymbol{F} \in \xWn{{1,1}}(\xR_{>0};\mathbf{X})$; and, obviously, $\boldsymbol{U}^\star_0 \in D(\xA_2)$.
Thus, Problem~\eqref{UU*}--\eqref{UU*0} admits a unique strong solution \cite[Proposition~4.1.6]{CH98}, with regularity $\boldsymbol{U}^\star \in \xCone(\xR_{\ge0}; \mathbf{X})\cap \xCzero(\xR_{\ge0}; D(\xA_2))$, which depends continuously on the data $\boldsymbol{F}$ and~$\boldsymbol{U}^\star_0$. Hence, we conclude the existence of 
$$ \boldsymbol{U}=(\boldsymbol{J}_1,\boldsymbol{J}_2,\boldsymbol{E},\boldsymbol{B})^\top 
= (\boldsymbol{J}^\star_1, \boldsymbol{J}^\star_2, \boldsymbol{E}^\star+\boldsymbol{g}_3, \boldsymbol{B}^\star+\boldsymbol{g}_4)^\top$$ 
solution to~\eqref{pb-evol-abstrait} with boundary condition~\eqref{cond}, and depending continuously on the data $\boldsymbol{g}$ and~$\boldsymbol{U}_0$. To get uniqueness, we notice that the difference of two solutions solves the homogeneous problem~\eqref{evol prob sm homo} with zero initial data, so it vanishes.
\end{proof}

\subsection{On the constraint equations}
Following the usual pattern in electromagnetics, the constraints on the fields: divergence equations \eqref{div E},~\eqref{div B}, magnetic boundary condition~\eqref{lkm}, are preserved by the evolution semigroup, provided the sources $(\rho,\boldsymbol{J}) := \sum_s (\rho_s,\boldsymbol{J}_s)$ satisfy the charge conservation equation.
We omit the proofs, as they are extremely classical~\cite[Remark~5.1.2, Thms 5.2.3 and~5.2.12]{ACL+17}. The details can be found in~\cite[\S4]{LZ+17}.
As a matter of fact, once the existence and uniqueness of the solution to the coupled model is obtained, the electromagnetic variables $(\boldsymbol{E},\boldsymbol{B})$ naturally appear as the solution to the Maxwell equations with data $(\rho,\boldsymbol{J})$.

\begin{thrm}
Assume that 
\begin{equation*}
\divg \boldsymbol{E}_0=\frac{\rho(0)}{\varepsilon_0}, \quad\text{and}\quad \divg\boldsymbol{B}_0=0\quad \text{in } \Omega,\quad
\boldsymbol{B}_0 \cdot \boldsymbol{n}=0 \quad\text{on } \Gamma_P,
\end{equation*}
and that  the charge conservation equation 
\begin{equation*}
\frac{\partial \rho(t)}{\partial t}+ \divg \boldsymbol{J}(t)=0 \quad \text{holds in~$\Omega$ for a.e. }  t> 0.
\end{equation*}
Then, for all $t>0$, the electric field $\boldsymbol{E}$ satisfies \eqref{div E} and the magnetic field $\boldsymbol{B}$ satisfies \eqref{div B} and \eqref{lkm}. 
\label{divergence B}
\end{thrm}

\begin{rmrk} 
For $\ell=1,\ 2$, we define $\mathbf{X}_{\ell} := \mathbf{L}^2_{(1)}(\Omega) \times \mathbf{L}^2_{(2)}(\Omega) \times \mathbf{L}^2(\Omega) \times \mathbf{H}_{0,\Gamma_P}(\divg 0;\Omega)$ where the case $\ell=1$ corresponds to $\Gamma_A=\varnothing $, \ie, $\Gamma_P=\Gamma$ and $\ell=2$ is for the case $\Gamma_A\neq\varnothing$. Then, $\xim(\xA_\ell) \subset \mathbf{X}_{\ell}$ and, according to Theorem~\Rref{divergence B}, we conclude that the space $\mathbf{X}_{\ell}$ is stable by the semigroup $\mathit{T}_{\ell}$ generated by the operator~$-\xA_{\ell}$, \ie,
\begin{itemize}
\item For all $\boldsymbol{U}_0\in D(\xA_{1})\cap \mathbf{X}_{1}$, there exists a unique $\boldsymbol{U}\in \xCone(\xR_{\ge0}; \mathbf{X}_1)\cap \xCzero(\xR_{\ge0};D(\xA_{1})\cap \mathbf{X}_1)$ solution to the system of equations  \eqref{uu}--\eqref{init cond 0} and \eqref{div B} with boundary conditions \eqref{lk} and~\eqref{lkm}.
\item For all $\boldsymbol{U}_0\in D(\xA_{2})\cap \mathbf{X}_{2}$, there exists a unique $\boldsymbol{U}\in \xCone(\xR_{\ge0}; \mathbf{X}_2)\cap \xCzero(\xR_{\ge0};D(\xA_{2})\cap \mathbf{X}_2)$ solution to the system of equations  \eqref{uu}--\eqref{init cond 0} and \eqref{div B} with boundary conditions \eqref{lk}--\eqref{lkm}.
\end{itemize}
Also, for all $\ell$, if we take $\boldsymbol{U}_0\in\mathbf{X}_{\ell}$, the two problems above have a weak solution $\boldsymbol{U}\in \xCzero(\xR_{\ge0}; \mathbf{X}_{\ell})$.
\label{stable space}
\end{rmrk}

\section{Some results of functional analysis}
\label{sec-prelim-2}
\subsection{The geometry}
As said in~\S\Rref{sec-model}, the domain~$\Omega$ can be topologically non-trivial, and the boundary~$\Gamma$ can be connected, or not; see Figure~\ref{fig1} for an example. We now introduce some notations associated with this geometry; we use the notations from~\cite{ABD98,ACL+17,FG97}.

\medbreak

We denote by $\Gamma_k$, $0\leq k \leq K$ the connected  components of $\Gamma$, $\Gamma_0$ being the boundary of the unbounded component of $\xR^3\setminus  \overline{\Omega}$. When the boundary is connected, $\Gamma_0=\Gamma$. Let us introduce a subspace of $\xHone(\Omega)$:
\begin{equation*}
\xHone_{\partial\Omega}(\Omega):=\lbrace q\in\xHone(\Omega)\,:\, q_{|\Gamma_0}=0,\, q_{|\Gamma_k}=C_k,\, 1\leq k \leq K \rbrace.
\end{equation*}
Above, $C_k$ means a constant, and for $\ell\neq k$, $C_{\ell}$ and $C_k$ may be different. This space can be endowed with the norm $\| \cdot \|_{\xHone_{\partial\Omega}(\Omega)}=\| \grad \cdot \|$ (see \cite[Proposition~2.1.66]{ACL+17}). 

\medbreak

We assume that there exist $J$ connected open surfaces $\Sigma_j$,  $j=1,\ \ldots,\ J$, called ``cuts'', contained in $\Omega$, such that:
\begin{enumerate}[label=\roman*)]
\item each surface $\Sigma_j$ is a smooth, orientable two-dimensional manifold;
\item the boundary of $\Sigma_j$ is contained in $\partial \Omega$; 
\item the intersection $ \overline{\Sigma_j}\cap\overline{\Sigma_i}$ is empty for $i\neq  j$;
\item (if $\Gamma_A\ne\emptyset$:) $\Gamma_A \setminus \partial\Sigma$, where $\Sigma=\bigcup_{j=1}^{J}\Sigma_j$,  has a finite number of connected components, denoted $\Gamma_{A,i}$, $i=1,\ \ldots,\ N$, whose closures are compact Lipschitz submanifolds of~$\Gamma$;
\item the open set $\dot{\Omega}:=\Omega\setminus \Sigma$ is pseudo-Lipschitz \cite[Definition~3.1]{ABD98} and topologically trivial (\ie, any vector field with vanishing curl is the gradient of a scalar field  on~$\dot{\Omega}$).
\end{enumerate}
If $\Omega$ is topologically trivial, $J=0$ and $\dot{\Omega}=\Omega$.
The extension operator from $\mathbf{L}^2(\dot{\Omega})$ to $\mathbf{L}^2(\Omega)$ is denoted $\widetilde{\cdot}$, whereas  $[\cdot]_{\Sigma_j}$ denotes the jump across the surface $\Sigma_j$, $j=1,\ \ldots,\ J$. Being orientable, each cut is assumed to have a ``plus'' and a ``minus'' side, so $[w]_{\Sigma_j} = w_{|_{\Sigma_j^+}} - w_{|_{\Sigma_j^-}}$. 
For all $j$, we denote  $\langle \cdot,\cdot \rangle_{\Sigma_j}$  the duality pairing between $\xHn{\frac1{2}}(\Sigma_j)$ and its dual $\xHn{{-\frac1{2}}}(\Sigma_j)$.  

\medbreak 

\begin{figure}
\centering \includegraphics[height=5.0cm]{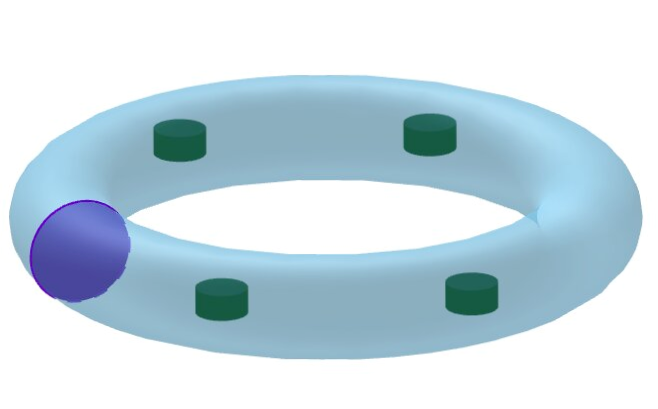}
\caption{Exemple of non-topologically trivial domain with a non-connected boundary. The domain $\Omega$ is made of the interior of the torus minus the green cylinders. The boundary $\partial\Omega$ is the union of the boundary of the torus and the cylinders which are all disjoint. The purple surface is a cut $\Sigma_1$, and the set $\dot\Omega = \Omega \setminus \Sigma_1$ is simply connected.} 
\label{fig1}
\end{figure}

\subsection{Hodge decompositions and topology-related spaces}
\label{sec-hodge}
Let $\mathbb{B}:\Omega\longrightarrow \mathcal{M}_{3}(\xC)$  be a matrix-valued function. We make  the following assumption:
\begin{eqnarray}
\exists\,\eta,\,\zeta>0, \quad\eta(\boldsymbol{v}^{*}\boldsymbol{v})\geq\|  \boldsymbol{v}^{*}\mathbb{B}(\boldsymbol{x})\boldsymbol{v}\|\geq \zeta(\boldsymbol{v}^{*}\boldsymbol{v}),\quad \forall\boldsymbol{v}\in\mathbb{C}^{3},\quad \forall\boldsymbol{x} \in\Omega.
\label{minor}
\end{eqnarray}
Define the Hilbert space 
\begin{eqnarray*}
\mathbf{H}(\divg\mathbb{B};\Omega) := \lbrace \boldsymbol{v}\in\mathbf{L}^2(\Omega)\,:\, \divg\mathbb{B}\boldsymbol{v}\in \xLtwo(\Omega)  \rbrace.
\end{eqnarray*}
This space is equipped with the canonical norm $\boldsymbol{v}\mapsto(\|\boldsymbol{v}\|^2+\|\divg\mathbb{B}\boldsymbol{v}\|^2)^{\frac1{2}}$. The subspace $\mathbf{H}(\divg\mathbb{B}0;\Omega):=\lbrace \boldsymbol{v}\in\mathbf{L}^2(\Omega)\,:\, \divg\mathbb{B}\boldsymbol{v} = 0 \rbrace$ is obviously a closed subspace of both $\mathbf{H}(\divg\mathbb{B};\Omega)$ and~$\mathbf{L}^2(\Omega)$.
If $\boldsymbol{v}\in\mathbf{H}(\divg\mathbb{B};\Omega)$, the normal trace of $\boldsymbol{v}$ is a well-defined element of $\xHn{\frac1{2}}(\Gamma)$ and the integration by parts formula holds
\begin{equation}
\forall\boldsymbol{v}\in\mathbf{H}(\divg\mathbb{B};\Omega),\,\forall q\in\xHone(\Omega),\quad
(\mathbb{B}\boldsymbol{v}\mid\grad q)+(\divg\mathbb{B}\boldsymbol{v}\mid q)=\langle \mathbb{B}\boldsymbol{v}\cdot \boldsymbol{n} , q \rangle_{\xHn{\frac1{2}}(\Gamma)}.
\label{green div grad}
\end{equation}
If $q\in\xHone_0(\Omega)$, the above formula can be extended to $\boldsymbol{v}\in \mathbf{L}^2(\Omega)$. Then, $\divg\mathbb{B}\boldsymbol{v}\in\xHn{{-1}}(\Omega)$ and one gets 
\begin{equation}
(\mathbb{B}\boldsymbol{v}\mid\grad q)+\langle \divg\mathbb{B}\boldsymbol{v}, q\rangle_{\xHone_0(\Omega)}=0.
\label{green div grad h 1 0}
\end{equation}

\medbreak

In this paragraph, we introduce some other spaces and notations associated to a matrix 
$\mathbb{B}$ satisfying~\eqref{minor}, and we prove some useful results. We start with a result on elliptic problems, whose proof is straightforward and left to the reader.
\begin{lmm}
For any $f\in\mathrm{H^{-1}}(\Omega)$, the elliptic problem:  
\begin{equation}
\left\lbrace \begin{array}{lll}
\text{Find } q\in\xHone_0(\Omega) \text{ such that  } \\
-\Delta_{\mathbb{B}}q: = -\divg(\mathbb{B}\grad q)=f
\end{array} \right.
\label{r'esol}
\end{equation}
admits a unique solution.  Furthermore, there exists a constant $\,C>0$  such that
\begin{eqnarray}
\| q \|_{\xHone(\Omega)}\leq C\, \|f\|_{\xHn{{-1}}(\Omega)}.
\label{'enig1}
\end{eqnarray}
\label{lem1}
\end{lmm}
\begin{lmm}
For any $\boldsymbol{v}\in\mathbf{L}^2(\Omega)$, there exists a unique pair $(q,\boldsymbol{v}^T)\in\mathrm{H^1_0}(\Omega)\times \mathbf{L}^2(\Omega)$ satisfying the conditions
\begin{equation}
\boldsymbol{v}=\grad q+\boldsymbol{v}^T,\quad \divg(\mathbb{B}\boldsymbol{v}^T)=0.
\label{decomp}
\end{equation}
Moreover, there exists $\,C>0$ independent constant of $\boldsymbol{v}$ such that
\begin{equation}
\| \grad q\|\leq C \,\| \boldsymbol{v} \|,\quad \| \boldsymbol{v}^T \|\leq C\, \| \boldsymbol{v} \|.
\label{in'ega}
\end{equation}
\label{lem0}
\end{lmm} 
\begin{proof}
Let  $\boldsymbol{v}\in\mathbf{L}^2(\Omega)$. As $\mathbb{B}$ is bounded  and the application $\divg$ is continuous from $ \mathbf{L}^2(\Omega)$ to~$\xHn{{-1}}(\Omega)$, we have $\divg \mathbb{B}\boldsymbol{v} \in \xHn{{-1}}(\Omega)$ and $\|\divg \mathbb{B}\boldsymbol{v} \|_{\xHn{{-1}}(\Omega)} \leq C_1\, \| \boldsymbol{v} \|$. According to Lemma~\Rref{lem1}, there exists one, and only one, $q\in\mathrm{H^1_0}(\Omega)$ that solves the problem~\eqref{r'esol} with data $f = -\divg\mathbb{B}\boldsymbol{v}$  and satisfies
$ \| \grad q\|\leq C\, \| \divg\mathbb{B}\boldsymbol{v} \|_{\xHn{{-1}}(\Omega)}$. 
Finally, let $\boldsymbol{v}^T=\boldsymbol{v}-\grad q$. By construction, we have $\boldsymbol{v}^T\in\mathbf{L}^2(\Omega)$ and $\divg\mathbb{B}\boldsymbol{v}^T=0$. The estimates of~\eqref{in'ega} are also established.
\end{proof}

\noindent We now characterize the following space:
\begin{eqnarray*}
\mathbf{Z}_{N}(\Omega; \mathbb{B}):=\mathbf{H}_0(\curl0;\Omega)\cap\mathbf{H}(\divg\mathbb{B}0;\Omega).
\end{eqnarray*}

\begin{prpstn}
The dimension of the vector space $\mathbf{Z}_{N}(\Omega; \mathbb{B})$ is equal to~$K$, the number of connected components of the boundary, minus one. Furthermore, a basis of $\mathbf{Z}_{N}(\Omega; \mathbb{B})$ is the set of the functions $(\grad{{q}_k})_{1\leq k \leq K}$, where each $q_k$ is the unique solution in $\xHone(\Omega)$ to the problem
\begin{eqnarray}
\left\lbrace \begin{array}{lll}
\Delta_{\mathbb{B}} q_k = \divg\mathbb{B}\grad q_k = 0, \quad\text{in } \Omega,\\[3pt]
{q_k}_{|\Gamma_0} = 0  \quad\text{and}\quad {q_k}_{|\Gamma_i}
=cst_{i},\quad 1\leq i\leq K,\\[3pt]
\langle \mathbb{B}\grad q_k\cdot \boldsymbol{n} , 1 \rangle_{\xHn{\frac1{2}}(\Gamma_0)}
= -1 ~\text{ and }~ \langle \mathbb{B}\grad q_k\cdot \boldsymbol{n} , 1 \rangle_{\xHn{\frac1{2}}({\Gamma_i})}
= \delta_{ki}, ~1\leq i \leq K.
\end{array} \right.
\label{div B grad}
\end{eqnarray} 
\label{dim ZNB}
\end{prpstn}
\begin{proof}
Entirely similar to~\cite[Proposition~3.18]{ABD98}, using the integration-by-parts formulas  \eqref{green div grad h 1 0} and~\eqref{green div grad}, and the well-posedness of elliptic problems involving the operator~$\Delta_{\mathbb{B}}$, as in Lemma~\Rref{lem1}.
\end{proof}
\begin{rmrk}
All norms being equivalent on the finite-dimensional space~$\mathbf{Z}_{N}(\Omega;\mathbb{B})$, we may use any norm to measure elements of this space; for example 
$$\boldsymbol{v}\mapsto \| \boldsymbol{v} \|, \quad\text{or}\quad 
\boldsymbol{v}\mapsto |(\langle \mathbb{B}\boldsymbol{v}\cdot\boldsymbol{n}, 1 \rangle_{\xHn{\frac1{2}}({\Gamma_{k}})})_{1\leq 1 \leq K}|_p,$$
with $1 \leq p \leq \infty$. Also, we easily check that $\mathbf{Z}_N(\Omega; \mathbb{B})=\grad [{Q}_N(\Omega;\mathbb{B})]$, where 
$${Q}_N(\Omega;\mathbb{B}):=\lbrace q\in \xHone_{\partial\Omega}(\Omega)~:~ \divg\mathbb{B}\grad q=0 \text{ in } \Omega\rbrace, $$
so, $\boldsymbol{v}=\grad q \mapsto |(q_{|\Gamma_k})_{1 \leq k \leq K}|_p$, with $1 \leq p \leq \infty$, is another norm. 
Hereafter, we denote $|\cdot|_{\mathbf{Z}_N(\Omega;\mathbb{B})}$ the chosen norm. 
\end{rmrk}

\subsection{A compactness result}
We introduce the function space  
\begin{equation*}
\mathbf {X}_{N,\Gamma}(\Omega; \mathbb{B}):= \{ \boldsymbol{w} \in \mathbf{H}(\curl;\Omega): \divg\mathbb{B}\boldsymbol{w}\in \xLtwo(\Omega) \text{ and } \boldsymbol{w}\times\boldsymbol{n}_{|_{\Gamma}} \in \mathbf{L}^2(\Gamma)\} \,;
\end{equation*}
obviously, $\boldsymbol{w}\times\boldsymbol{n}_{|_{\Gamma}}$ can be replaced with $\boldsymbol{w}_{\top|_{\Gamma}}$ in the above definition. It is endowed with its canonical norm
\begin{equation}
 \| \boldsymbol{w} \|_{\mathbf {X}_{N,\Gamma}(\Omega; \mathbb{B})}^2=\| \boldsymbol{w} \|^2+ \| \curl \boldsymbol{w} \|^2+\| \divg\mathbb{B}\boldsymbol{w}\|^2+\| \boldsymbol{w}_\top \|_{\mathbf{L}^2(\Gamma)}^2.
\label{natural norm XNA}
\end{equation}
Below, we derive some useful properties of the space $\mathbf {X}_{N,\Gamma}(\Omega; \mathbb{B})$: it is compactly embedded into~$\mathbf{L}^2(\Omega)$, which yields an a inequality in $\mathbf{L}^2$~norm for elements of $\mathbf {X}_{N,\Gamma}(\Omega; \mathbb{B})$, and finally a new norm equivalent to~\eqref{natural norm XNA}. These results parallel and generalize those of~\cite[Thm~3.22]{H+14} (for the boundary condition $\boldsymbol{w}\times\boldsymbol{n}=0$ and $\Gamma$~connected) and~\cite[Thm~8.1.3]{ACL+17} (where $\mathbb{B}$ is assumed to be real and symmetric, and $\boldsymbol{w}\times\boldsymbol{n}=0$ on~$\Gamma_P$), both of which grounded in the pioneering work~\cite[Theorem~2.2]{Webe80}.
However, we choose to present the proof, as the two simultaneous negative features 
(non-Hermitianness of~$\mathbb{B}$ and non-connectedness of~$\Gamma$) call for a careful demonstration. 

\subsubsection{Compact embedding of $\mathbf {X}_{N,\Gamma}(\Omega; \mathbb{B})$ into $\mathbf{L}^2(\Omega)$}
We denote
\begin{equation*}
\mathbf{Z}_N(\Omega):=\mathbf{H}_0(\curl0;\Omega)\cap\mathbf{H}(\divg0;\Omega) =\mathbf{Z}_N(\Omega;\mathbb{I}). 
\end{equation*}
As the identity matrix $\mathbb{I}$ obviously satisfies the condition~\eqref{minor}, $\mathbf{Z}_N(\Omega)$ is of dimension $K$ and a basis is given by~\eqref{div B grad}.
%
Next, we introduce the (closed) subspace of $\mathbf{H}(\divg 0;\Omega)$:
\begin{equation*}
\mathbf{H}^{\Gamma}(\divg 0;\Omega):=\lbrace \boldsymbol{v}\in \mathbf{H}(\divg 0;\Omega): \langle \boldsymbol{v} \cdot \boldsymbol{n}, 1 \rangle_{\xHn{\frac1{2}}({\Gamma_{k}})}=0,~ 1\leq k \leq K\rbrace.
\end{equation*}
As an immediate consequence of Proposition~\Rref{dim ZNB}, we have\dots
\begin{prpstn}
Let  $\Omega$ a domain. The following orthogonal decomposition  of  the space $\mathbf{H}(\divg0;\Omega)$ holds:
\begin{equation*}
\mathbf{H}(\divg0;\Omega)=\mathbf{Z}_N(\Omega)\stackrel{\perp}{\oplus}\mathbf{H}^{\Gamma}(\divg 0;\Omega).
\end{equation*}
\label{decomp div0}
\end{prpstn}
Now, we can prove the following compactness result.
\begin{thrm}
Let $\Omega$ be a domain. The embedding of the space $\mathbf {X}_{N,\Gamma}(\Omega; \mathbb{B})$ into $\mathbf{L}^2(\Omega)$ is compact.
\label{compact}
\end{thrm}
\begin{proof}
Let $(\boldsymbol{v}_n)_{n\in\mathbb{N}}$ be a bounded sequence of $\mathbf {X}_{N,\Gamma}(\Omega; \mathbb{B})$. According to Lemma~\Rref{lem0}, there  exists two sequences  $(q_n)_{n\in\mathbb{N}}$ and $(\boldsymbol{v}^T_n)_{n\in\mathbb{N}}$ of elements, respectively, of
 $\mathrm{H^1_0}(\Omega)$ and  $\mathbf{L}^2(\Omega)$, such that   $\boldsymbol{v}_n=\grad q_n+\boldsymbol{v}^T_n$ for all $n$. Our aim, using this decomposition, is to prove that a subsequence of $ (\boldsymbol{v}_n)_n$ converges strongly in $\mathbf{L}^2(\Omega)$. This is done in two steps.
 
\medbreak 
 
\noindent \textsc{Step 1.}\quad 
According to~\eqref{in'ega}, the sequence $(q_n)_n$ satisfies, for all~$n$: $\| \grad q_n\|\leq C \| \boldsymbol{v}_n \|$, with $C$  independent of $ \boldsymbol{v}_n$. So, $ (q_n)_n$ is bounded in  $\mathrm{H^1_0}(\Omega)$, and 
since $\mathrm{H^1_0}(\Omega)$ is compactly embedded into~$\xLtwo(\Omega)$, there exists a subsequence, still denoted by $(q_n)_n$, that converges strongly in $\xLtwo(\Omega)$. Now, let us show that the subsequence $(\grad q_n)_n$ converges in~$\mathbf{L}^2(\Omega)$. Denote $ \boldsymbol{v}_{nm}:= \boldsymbol{v}_n-\boldsymbol{v}_m$ and $ q_{nm} := q_n-q_m$. By construction, the sequence $ (q_n)_n$  verifies  $ \divg\mathbb{B}\boldsymbol{v}_n=\divg\mathbb{B}\grad q_n,$ for all $n\in\mathbb{N}$. This leads to the inequality 
\begin{eqnarray*}
\left| \left( \divg(\mathbb{B}\grad q_{nm})\mid q_{nm}\right) \right| 
&=& \left| \left( \divg(\mathbb{B}\boldsymbol{v}_{nm})\mid q_{nm}\right) \right| \\
& \leq & \| \divg(\mathbb{B}\boldsymbol{v}_{nm})\| \cdot \| q_{nm} \| \\
&\leq &2\sup_n\| \boldsymbol{v}_n\|_{ \mathbf {X}_{N,\Gamma}(\Omega; \mathbb{B})}\,\| q_{nm} \|
\leq  C'\,\| q_{nm} \|.
\end{eqnarray*}
On the other hand, from~\eqref{minor} and the integration-by-parts formula~\eqref{green div grad}, we deduce:
\begin{equation*}
\left| \left( \divg(\mathbb{B}\grad q_{nm})\mid q_{nm}\right) \right| 
= \left| \left( \mathbb{B}\grad q_{nm}\mid \grad q_{nm}\right) \right| 
\geq \zeta\,\|\grad q_{nm}\|.
\end{equation*}
Combining the above, we conclude:
\begin{equation*}
\| \grad q_n-\grad q_m\|
\leq \frac{C'}{\zeta}\,\| q_{nm} \|.
\end{equation*}
So, $(\grad q_n)_n$ is a Cauchy sequence in~$\mathbf{L}^2(\Omega)$, and therefore it converges in this space.

\medbreak

\noindent \textsc{Step 2.}\quad 
Recall that the sequence $(\boldsymbol{v}^T_n)_n$  verifies $ \divg\mathbb{B}\boldsymbol{v}^T_n=0$, $ \curl\boldsymbol{v}^T_n=\curl\boldsymbol{v}_n$ and $ \boldsymbol{v}^T_n\times\boldsymbol{n}_{|\Gamma}=\boldsymbol{v}_n\times\boldsymbol{n}_{|\Gamma}$. By  Proposition~\Rref{decomp div0}, there exists a sequence $(\boldsymbol{y}_n)_n$ of elements of  $\mathbf{H}^{\Gamma}(\divg0;\Omega)$ and a sequence $(\boldsymbol{z}_n)_n$ on~$\mathbf{Z}_N(\Omega)$ such that  $\mathbb{B}\boldsymbol{v}^T_n=\boldsymbol{z}_n+\boldsymbol{y}_n$ for all~$n$. 
The sequence $(\boldsymbol{z}_n)_n$ is bounded in the finite-dimensional space $\mathbf{Z}_N(\Omega)$, so there exists a subsequence, still denoted by $(\boldsymbol{z}_n)_n$, which converges in any norm, \vg, that of~$\mathbf{L}^2(\Omega)$. Then, according to \cite[Theorem~3.4.1]{ACL+17}, there exists a sequence $ (\boldsymbol{w}_n)_n$ of elements of  $\mathbf{H}^1(\Omega)$ such that  $\boldsymbol{y}_n=\curl\boldsymbol{w}_n$ for all~$n$, and it satisfies:   
\begin{equation*}
\| \boldsymbol{w}_n \|_{\mathbf{H^1}(\Omega)}\leq \xi\, \| \boldsymbol{y}_n \|
\end{equation*}
for some $\xi>0$. As $(\boldsymbol{y}_n)_n$ is bounded in~$\mathbf{L}^2(\Omega)$, it follows that $(\boldsymbol{w}_n)_n$ is bounded in~$\mathbf{H}^1(\Omega)$. As the trace mapping is continuous from $\mathbf{H}^1(\Omega)$ to~$\mathbf{H}^{\frac{1}{2}}(\Gamma)$, it follows that $({\boldsymbol{w}_n}_{|\Gamma})_n$ is bounded in  $\mathbf{H}^{\frac{1}{2}}(\Gamma)$. Therefore, by Sobolev's compact embedding theorem, we can extract a subsequence, still denoted by $(\boldsymbol{w}_n)_n$, that converges in $\mathbf{L}^2(\Omega)$ and such that $({\boldsymbol{w}_n}_{|\Gamma})_n$  converges in $\mathbf{L}^2(\Gamma)$. Denote $\boldsymbol{v}^T_{nm}:=\boldsymbol{v}^T_n-\boldsymbol{v}^T_m$, $ \boldsymbol{w}_{nm}:=\boldsymbol{w}_n-\boldsymbol{w}_m$ and $\boldsymbol{z}_{nm}:=\boldsymbol{z}_n-\boldsymbol{z}_m$. According to the condition \eqref{minor}, $\mathbb{B}$ is invertible, and we find
\begin{eqnarray*}
\left| \left(\mathbb{B}^{-1}(\boldsymbol{z}_{nm}+\curl\boldsymbol{w}_{nm})\mid \boldsymbol{z}_{nm}+\curl\boldsymbol{w}_{nm} \right) \right|
&=&\left| \left( \boldsymbol{v}^T_{nm}\mid\mathbb{B}\boldsymbol{v}^T_{nm}\right) \right| \\
&\geq & \zeta\, \|\boldsymbol{v}^T_{nm}\|.
\end{eqnarray*}
Next, by integration by parts~\eqref{green rot rot}, we obtain
\begin{eqnarray*}
&&\left| \left( \mathbb{B}^{-1}(\boldsymbol{z}_{nm}+\curl\boldsymbol{w}_{nm}) \mid \boldsymbol{z}_{nm}+\curl\boldsymbol{w}_{nm} \right) \right|
\\
&=& \left| \left( \boldsymbol{v}^T_{nm} \mid \boldsymbol{z}_{nm}+\curl\boldsymbol{w}_{nm} \right) \right| \\
&=& \left| \left(\boldsymbol{v}^T_{nm}\mid\boldsymbol{z}_{nm}\right) + \left( \curl\boldsymbol{v}^T_{nm}\mid\boldsymbol{w}_{nm}\right)+ \left(\boldsymbol{v}^T_{nm}\times\boldsymbol{n}\mid (\boldsymbol{w}_{nm})_\top\right)_{\mathbf{L}^2(\Gamma)} \right| \\
&=& \left| \left(\boldsymbol{v}^T_{nm}\mid\boldsymbol{z}_{nm}\right) + \left( \curl\boldsymbol{v}_{nm}\mid\boldsymbol{w}_{nm}\right)+ \left(\boldsymbol{v}_{nm}\times\boldsymbol{n}\mid (\boldsymbol{w}_{nm})_\top\right)_{\mathbf{L}^2(\Gamma)} \right| \\
&\leq& 2\sup_n\| \boldsymbol{v}_{nm}\|_{ \mathbf {X}_{N,\Gamma}(\Omega; \mathbb{B})} (\|\boldsymbol{z}_{nm}\|+\|\boldsymbol{w}_{nm}\|+\|\boldsymbol{w}_{nm}\|_{\mathbf{L}^2(\Gamma)}).\\
& \leq & C'\, (\|\boldsymbol{z}_{nm}\|+\|\boldsymbol{w}_{nm}\|+\|\boldsymbol{w}_{nm}\|_{\mathbf{L}^2(\Gamma)}).
\end{eqnarray*}
Combining the above, we find 
\begin{eqnarray*}
\|\boldsymbol{v}^T_n-\boldsymbol{v}^T_m\|_{\mathbf{L}^2(\Omega)}\leq \frac{C'}{\zeta}\,(\|\boldsymbol{z}_{nm}\|+\|\boldsymbol{w}_{nm}\|+\|\boldsymbol{w}_{nm}\|_{\mathbf{L}^2(\Gamma)}).
\end{eqnarray*}
So, $(\boldsymbol{v}^T_n)_n$ is a Cauchy sequence in $\mathbf{L}^2(\Omega)$, and it converges in this space. Finally, the subsequence $(\boldsymbol{v}_n)_n$,  defined by $\boldsymbol{v}_n:=\grad q_n+\boldsymbol{v}^T_n$,  converges in $\mathbf{L}^2(\Omega)$.
\end{proof}

\subsubsection{Equivalent norms on $\mathbf{X}_{N,\Gamma}(\Omega; \mathbb{B})$}
As a consequence of Theorem \Rref{compact}, there holds a basic inequality. The proof follows the lines of \cite[Thm~3.4.3]{ACL+17} and~\cite[Proposition~7.4]{FG97}.
\begin{prpstn}
There exists a constant $C>0$ such that 
\begin{equation}
\forall \boldsymbol{v}\in \mathbf{X}_{N,\Gamma}(\Omega; \mathbb{B}),\quad
\| \boldsymbol{v} \|\leq C\,\left\lbrace \| \curl \boldsymbol{v} \|+\| \divg\mathbb{B}\boldsymbol{v} \|+\| \boldsymbol{v}_\top \|_{\mathbf{L}^2(\Gamma)}+
 | P_{\mathbf{Z}_{N}(\Omega;\mathbb{B})}\boldsymbol{v} |_{\mathbf{Z}_N(\Omega;\mathbb{B})} \right\rbrace.
\label{seminorme}
\end{equation}
\end{prpstn}
\begin{crllr}
The semi-norm
\begin{eqnarray*}
\quad|v|_{\mathbf{X}_{N,\Gamma}(\Omega; \mathbb{B})}=
\left(\| \curl \boldsymbol{v} \|^2+\| \divg\mathbb{B}\boldsymbol{v} \|^2+\| \boldsymbol{v}_\top \|^2_{\mathbf{L}^2(\Gamma)}%
+ | P_{\mathbf{Z}_{N}(\Omega;\mathbb{B})}\boldsymbol{v} |_{\mathbf{Z}_N(\Omega;\mathbb{B})}\right)^{\frac{1}{2}}.
\end{eqnarray*} 
is a norm in $\mathbf{X}_{N,\Gamma}(\Omega; \mathbb{B})$, equivalent to the natural norm.
\label{coro semi norm}
\end{crllr}

\section{Spectral properties of some useful matrices}
\label{sec-spectral}
We still denote $\III \cdot \III_{\mathcal{M}}$ the operator norm of $\mathcal{M}_{3}(\xC)$ induced by the Hermitian norm of~$\xC^3$. In the rest of the paper, in addition to Hypothesis~\Rref{hyp-1}, we shall make the following\dots

\begin{hpthss}
For each species $s$, the real functions $\nu_s$ and $\omega_{ps}$ are bounded below by a strictly positive constant on $\Omega$, \ie, there exist $\nu_{*}>0$ and $\omega_*>0$ such that:
\begin{equation}
\omega_{ps}(\boldsymbol{x}) \geq \omega_*,\quad 
\nu_s(\boldsymbol{x}) \geq \nu_{*},\quad \forall s\in\lbrace 1,\ 2 \rbrace, \quad \forall \boldsymbol{x}\in\Omega \text{ a.e.}
\label{nus bounded below}
\end{equation} 
\label{hyp-2}
\end{hpthss}
\begin{prpstn}
Let $s\in\lbrace 1,\ 2 \rbrace$ and $\alpha\in\xR$. Then, the matrix $\,\ii\alpha\,\mathbb{I}+\mathbb{M}_s$ is invertible for all $\boldsymbol{x}\in\Omega$. Moreover, its inverse is uniformly bounded on $\Omega$.
\label{i alpha inverse stability}
\end{prpstn}
\begin{proof}
First, we determine the matrix~$\mathbb{M}_s$. At each point $\boldsymbol{x}\in\Omega$, we consider an orthonormal Stix frame~\cite{S92} $(\boldsymbol{e}_1(\boldsymbol{x}), \boldsymbol{e}_2(\boldsymbol{x}), \boldsymbol{e}_3(\boldsymbol{x})=\boldsymbol{b}(\boldsymbol{x}))$. 
In this frame, the expression of~$\mathbb{M}_s$ writes:
\begin{equation*}
\mathbb{M}_s=
\begin{pmatrix}
\nu_s & -\varOmega_{cs} & 0\cr
\varOmega_{cs} & \nu_s & 0 \cr
0 & 0 & \nu_s \cr
\end{pmatrix}.
\end{equation*}
Hence we deduce the expression 
\begin{equation*}
\ii\alpha\mathbb{I}+\mathbb{M}_s=\begin{pmatrix}
\ii\alpha+\nu_s &-\varOmega_{cs}&0\\
\varOmega_{cs} &\ii\alpha+\nu_s&0\\
0&0&\ii\alpha+\nu_s\
\end{pmatrix}.
\end{equation*}
The determinant of this matrix is: 
\begin{equation*}
\det(\ii\alpha\mathbb{I}+\mathbb{M}_s)
= (\ii\alpha+\nu_s)\left[ (\ii\alpha+\nu_s)^2+\varOmega_{cs}^2\right]
= (\ii\alpha+\nu_s)\left[(\varOmega_{cs}^2+\nu_s^2 -\alpha^2)+2\ii\alpha\nu_s\right]:=d_s. 
\end{equation*}
By Hypothesis~\Rref{hyp-2}, for $\alpha\in\xR$ fixed it holds that $|d_s(\boldsymbol{x})|\geq d_\alpha > 0$ a.e. on~$\Omega$. Thus, the matrix $\ii\alpha\mathbb{I}+\mathbb{M}_s$ is invertible, and the usual inversion formula gives:
\begin{eqnarray*}
(\ii\alpha\mathbb{I}+\mathbb{M}_s)^{-1}=\frac1{d_s}\begin{pmatrix}
(\ii\alpha+\nu_s)^2 & \varOmega_{cs}(\ii\alpha+\nu_s)& 0\\
-\varOmega_{cs}(\ii\alpha+\nu_s) &(\ii\alpha+\nu_s)^2 & 0\\
0 & 0 &(\ii\alpha+\nu_s)^2+\varOmega_{cs}^2\
\end{pmatrix}.
\end{eqnarray*}
We can check that the above matrix is normal (\ie, it commutes with its conjugate transpose). By~\cite[Theorem~1.4-2]{C88}, we deduce that the $\III\cdot\III_{\mathcal{M}}$~norm of $(\ii\alpha\mathbb{I}+\mathbb{M}_s)^{-1}$ is equal to its spectral radius. Therefore, to prove that $(\ii\alpha\mathbb{I}+\mathbb{M}_s)^{-1}$ is uniformly bounded it suffices to bound its spectral radius on $\Omega$. Its eigenvalues are:
\begin{equation*}
\frac{(\ii\alpha+\nu_s)^2 \pm \ii \varOmega_{cs}(\ii\alpha+\nu_s)}{d_s} \quad\text{and}\quad \frac{(\ii\alpha+\nu_s)^2+\varOmega_{cs}^2}{d_s} \,.
\end{equation*}
According to Hypothesis \Rref{hyp-1} and the above, these eigenvalues are bounded on~$\Omega$. 
\end{proof}
From Hypotheses~\Rref{hyp-1} and~\Rref{hyp-2} there follows:
\begin{prpstn}
Let $\alpha\in\xR$. Let $\mathbb{D}_{\alpha}:\Omega\longrightarrow \mathcal{M}_{3}(\xC)$ be the matrix 
\begin{equation}
\mathbb{D}_{\alpha}(\boldsymbol{x}):=\sum_s {\omega_{ps}^2(\boldsymbol{x})}(\ii\alpha\mathbb{I}+ \mathbb{M}_s(\boldsymbol{x}))^{-1}, \quad \text{for }~ \boldsymbol{x}\in\Omega.
\label{matrix D i alpha stability} 
\end{equation}
Then, $\mathbb{D}_{\alpha}$ is uniformly bounded on $\Omega$. 
\label{matrix D alpha bounded}
\end{prpstn}
Let $\alpha\in\xR$. We now introduce another matrix which will play an important role in the proofs of stability.
\begin{equation}
\mathbb{B}_{\alpha}:=\ii\alpha\mathbb{I}\,+\,\mathbb{D}_{\alpha}:=\begin{pmatrix}
P & Q & 0\\
-Q & P & 0\\
0 & 0 & R \\
\end{pmatrix}
\label{B alpha matrix}
\end{equation}
where the functions $P$, $Q$ and $R$ are given by 
\begin{eqnarray}
P(\boldsymbol{x})&:=& \ii\alpha \,+\, \sum_s \frac{\omega_{ps}^2(\boldsymbol{x})(\ii\alpha+\nu_s(\boldsymbol{x}))}{(\ii\alpha+\nu_s(\boldsymbol{x}))^2+\varOmega_{cs}^2(\boldsymbol{x})} \,,
\label{P}\\
Q(\boldsymbol{x})&:=& \sum_s \frac{\omega_{ps}^2(\boldsymbol{x})\varOmega_{cs}(\boldsymbol{x})}{(\ii\alpha+\nu_s(\boldsymbol{x}))^2+\varOmega_{cs}^2(\boldsymbol{x})} \,,
\label{Q}\\
R(\boldsymbol{x})&:=& \ii\alpha \,+\, \sum_s \frac{\omega_{ps}^2(\boldsymbol{x})}{\ii\alpha+\nu_s(\boldsymbol{x})} \,.
\label{R}
\end{eqnarray}
The matrix $\mathbb{B}_{\alpha}$ is normal $(\mathbb{B}_{\alpha}\mathbb{B}_{\alpha}^{*}=\mathbb{B}_{\alpha}^{*}\mathbb{B}_{\alpha})$, and its eigenvalues are
\begin{equation*}
\lambda_{\alpha,1}=P+\ii Q,\quad \lambda_{\alpha,2}=P-\ii Q,\quad \lambda_{\alpha,3}= R.
\end{equation*} 
According to Proposition~\Rref{matrix D alpha bounded}, we deduce that $\mathbb{B}_{\alpha}$ is uniformly bounded on $\Omega$, $i.e$ there exists a constant $\eta_{\alpha}>0$ depending on~$\alpha$ such that 
\begin{equation*}
\sup_{\boldsymbol{x}\in\Omega} \III \mathbb{B}_{\alpha}(\boldsymbol{x}) \III_{\mathcal{M}} \leq\eta_{\alpha}.
\end{equation*} 
\begin{prpstn}
Let $\alpha\in\xR$. Then, the real parts of~$(\lambda_{\alpha,j})_{j=1,\ 2,\ 3}$ are  uniformly bounded below on $\Omega$. We then define $\zeta_{\alpha}$ to be 
\begin{equation}
\zeta_{\alpha}:=\min_{j=1,\ 2,\ 3}\,\inf_{\boldsymbol{x}\in\Omega}\,\Re({\lambda}_{\alpha,j}(\boldsymbol{x})) > 0.
\label{zeta}
\end{equation}
\end{prpstn}
\begin{proof}
From \eqref{P}--\eqref{R}, one obtains the expression of the real parts of the eigenvalues of $\mathbb{B}_{\alpha}$:
\begin{eqnarray*}
\Re(\lambda_{\alpha,1}(\boldsymbol{x}))&=& 
\sum_s \frac{\omega_{ps}^2(\boldsymbol{x})\nu_s(\boldsymbol{x})}{(\varOmega_{cs}^2(\boldsymbol{x})+\nu_{s}^2(\boldsymbol{x})-\alpha^2)^2+4\alpha^2\nu_s^2(\boldsymbol{x})}\,[(\varOmega_{cs}(\boldsymbol{x})+\alpha)^2+\nu_s^2(\boldsymbol{x})],
\\
\Re(\lambda_{\alpha,2}(\boldsymbol{x}))&=& 
\sum_s \frac{\omega_{ps}^2(\boldsymbol{x})\nu_s(\boldsymbol{x})}{(\varOmega_{cs}^2(\boldsymbol{x})+\nu_s^2(\boldsymbol{x})-\alpha^2)^2+4\alpha^2\nu_s^2(\boldsymbol{x})}\,[(\varOmega_{cs}(\boldsymbol{x})-\alpha)^2+\nu_s(\boldsymbol{x})^2],
\\
\Re(\lambda_{\alpha,3}(\boldsymbol{x}))&=& 
\sum_s \frac{\omega_{ps}^2(\boldsymbol{x})\nu_s(\boldsymbol{x})}{\nu_s^2(\boldsymbol{x})+\alpha^2}.
\end{eqnarray*}
Due to Hypothesis~\Rref{hyp-2} and assumption \eqref{4}, one deduces that these real parts are strictly positive. The rest of the proof follows by Hypotheses \Rref{hyp-1} and~\Rref{hyp-2}.
\end{proof}
\begin{lmm}
Given $\alpha\in\xR$, it holds that  
\begin{eqnarray}
\eta_{\alpha}(\boldsymbol{v}^{*}\boldsymbol{v})\geq | \boldsymbol{v}^{*}\mathbb{B}_{\alpha}(\boldsymbol{x})\boldsymbol{v} | \geq \Re[\boldsymbol{v}^{*}\mathbb{B}_{\alpha}(\boldsymbol{x})\boldsymbol{v}] \geq \zeta_{\alpha}(\boldsymbol{v}^{*}\boldsymbol{v}),\quad \forall\boldsymbol{v}\in\mathbb{C}^{3},\quad \forall\boldsymbol{x} \in\Omega.
\label{minor B i alpha}
\end{eqnarray}
\label{lem minor B i alpha}
\end{lmm}
\begin{rmrk}
According to Lemma \Rref{lem minor B i alpha}, we can apply all the results of Subsection~\Rref{sec-hodge} to the matrix $\mathbb{B}_{\alpha}$, for $\alpha\in\xR$. 
\end{rmrk}

\section{Strong stability}
\label{sec-strong-stab}
We define the energy of our model as $\mathcal{E}(t):=\frac1{2}\| \boldsymbol{U} \|_{\mathbf{X}}^2$. With the definition~\eqref{rrt} of the plasma pulsation, the term given by the $\boldsymbol{J}_s$ variables is interpreted as the kinetic energy of the particles: $\left| \dfrac{\boldsymbol{J}_s}{\omega_{ps}} \right|^2 \propto  n_s\, |\boldsymbol{\mathcal{U}}_s|^2$ at dominant order; while the $(\boldsymbol{E},\boldsymbol{B})$ part is the electromagnetic energy of the wave.
If $\boldsymbol{U}_0$ satisfies the condition~\eqref{initial cond nonho}, then, using the Green formula~\eqref{green rot rot}, one easily finds that: 
\begin{equation}
\frac{\xdif}{\xdif t} \mathcal{E}(t)
 = - \frac1{\varepsilon_0} \sum_{s} \left\| \frac{\sqrt{\nu_{s}}\, \boldsymbol{J}_{s}}{\omega_{p s}} \right\|^2  - \varepsilon_0 c^2\int_{\Gamma_A} (c\, |\boldsymbol{B}_\top|^2 - \boldsymbol{g}\cdot\boldsymbol{B}_\top) \, d\Gamma.
\label{derive energie}
\end{equation}
Here, $\Gamma_A$ is arbitrary. The above equation shows that the energy is non-increasing if $\Gamma_A=\varnothing$ or $\boldsymbol{g}=0$. (On the other hand, if $\Gamma_A=\varnothing$ and $\nu_s=0$, the derivative vanishes and $\mathcal{E}(t)=\mathcal{E}(0)$ for all $t>0$; this justifies Hypothesis~\Rref{hyp-2}.)

\smallbreak

Therefore, we will study the decay of the energy in both cases: perfectly conducting ($\Gamma_A=\varnothing$) and homogeneous Silver--M\"uller ($\Gamma_A\neq\varnothing$ and $\boldsymbol{g}=0$). Notice that, as a consequence of Hypotheses \Rref{hyp-1} and~\Rref{hyp-2}, the spaces $\mathbf{L}^2_{(s)}(\Omega)$ are equal to~$\mathbf{L}^2(\Omega)$, and the norms $\|\cdot\|_{(s)}$ and~$\|\cdot\|$ are equivalent. Similarly, the norm~$\|\cdot\|_{\mathbf{X}}$ is equivalent to the canonical norm of~$\mathbf{L}^2(\Omega)^4$.

\medbreak

For $\ell=1,\ 2$, the domain $D(\xA_\ell)$ is not compactly embedded into~$\mathbf{X}$; thus, the resolvent of~$\xA_\ell$ is not compact, as said in the Introduction. This fact precludes the use of many operator-theoretical results. To show the strong stability we shall use the general criterion of Arendt--Batty and Lyubich--Vu~\cite{AB88,LV88}.
\begin{thrm}[Arendt--Batty / Lyubich--Vu]
Let $\mathit{X}$ be a reflexive Banach space and  $(\mathit{T}(t))_{t\geq0}$ be a $C_{0}$-semigroup on $\mathit{X}$ generated by $\mathcal{L}$. Assume that $(\mathit{T}(t))_{t\geq0}$ is bounded and no eigenvalue of $\mathcal{L}$ lies on the imaginary axis. If $\sigma(\mathcal{L})\cap \ii\xR$ is countable, then  $(\mathit{T}(t))_{t\geq0}$ is strongly stable. 
\label{A B, L V}
\end{thrm}

\subsection{Perfectly conducting case}
\begin{prpstn}
For all $\alpha\in\xR \setminus \{0\}$, the operator 
$\ii\alpha\mathbb{I}+\xA_{1}$ is injective, \ie,
\begin{equation*}
\ker(\ii\alpha\mathbb{I}+\xA_{1})=\lbrace0\rbrace.
\end{equation*}
Furthermore, $0$ is an eigenvalue of $\xA_{1}$ and the corresponding set of eigenvectors is:
\begin{equation*}
\ker \xA_{1} = \lbrace (0,0,0,\boldsymbol{V})\, :\, \boldsymbol{V} \in \mathbf{H}(\curl 0;\Omega)\rbrace.
\end{equation*}
\label{ker A1}
\end{prpstn}
\begin{proof}
Let  $\alpha\in\xR$ and 
$\boldsymbol{U}=(\boldsymbol{U}_1,\boldsymbol{U}_2,\boldsymbol{U}_3,\boldsymbol{U}_4)^{\top}\in D(\xA_1)$ such that  
\begin{equation}
(\ii\alpha\,\mathbb{I}+\xA_1)\, \boldsymbol{U}=0 .
\label{00injec}
\end{equation}
This is equivalent to the system
\begin{eqnarray}
\ii\alpha\,\boldsymbol{U}_1+\mathbb{M}_1\boldsymbol{U}_1 - \varepsilon_0\omega_{p1}^2\, \boldsymbol{U}_3 &=&0,
\label{1injec}\\
\ii\alpha\,\boldsymbol{U}_2+\mathbb{M}_2\boldsymbol{U}_2 - \varepsilon_0\omega_{p2}^2\, \boldsymbol{U}_3 &=&0,
\label{2injec}\\
\ii\alpha\,\boldsymbol{U}_3 + \frac{1}{\varepsilon_0}\,\boldsymbol{U}_1 + \frac{1}{\varepsilon_0}\, \boldsymbol{U}_2- c^2 \curl\boldsymbol{U}_4 &=&0,
\label{3injec}\\
\ii\alpha\,\boldsymbol{U}_4 + \curl\boldsymbol{U}_3 &=&0.
\label{4injec}
\end{eqnarray}
Taking the real part of the inner product of~\eqref{00injec} with $\boldsymbol{U}$ in $\mathbf{X}$, one gets:
\begin{equation*}
\Re(\ii\alpha\,\|\boldsymbol{U}\|_{\mathbf{X}}^2)=\Re\left(\xA_1\boldsymbol{U},\boldsymbol{U}\right)_{\mathbf{X}}=0.
\end{equation*}
By the monotonicity of $\xA_1$, see Equation~\eqref{resul monot A1}, one obtains:
\begin{equation}
\left( \frac{\nu_s\,\boldsymbol{U}_s }{\varepsilon_0\omega_{ps}} \Biggm| \frac{\boldsymbol{U}_s}{\omega_{ps}} \right)=0, \quad  s=1,\ 2.
\label{5injec}
\end{equation}
By Hypothesis~\Rref{hyp-2}:
\begin{equation}
\left( \frac{\nu_s\,\boldsymbol{U}_s }{\varepsilon_0\omega_{ps}} \Biggm| \frac{\boldsymbol{U}_s}{\omega_{ps}} \right) \geq \frac{\nu_{*}}{\varepsilon_0}\, \|\boldsymbol{U}_s\|_{(s)}^2,\quad s=1,\ 2.
\label{6injec}
\end{equation}
From Equations \eqref{5injec} and~\eqref{6injec}, we deduce that 
\begin{equation}
\boldsymbol{U}_1=0 \quad\text{and}\quad \boldsymbol{U}_2=0 \quad \text{in } \Omega.
\label{ker u1 u2}
\end{equation}
This, together with Equation~\eqref{1injec}, implies that 
\begin{equation}
\boldsymbol{U}_3=0 \quad ~\, \text{in}~\,\Omega.
\label{ker u3}
\end{equation}
If $\alpha\neq 0$, $\boldsymbol{U}_4=0$ follows from \eqref{ker u3} and~\eqref{4injec}. And if $\alpha=0$, we deduce from~\eqref{3injec}, \eqref{ker u1 u2} and~\eqref{ker u3} that $ \curl\boldsymbol{U}_4=0$.   
The proof of the proposition is complete.
\end{proof}

Therefore, $\ker \xA_{1} $ consists of the space of stationary solutions to Problem~\eqref{uu}--\eqref{init cond 0} with boundary condition~\eqref{lk}, and it is of infinite dimension. From Remark~\Rref{stable space}, we can define the operator ${\xA_1}_{|\mathbf{X}_1}: D(\xA_1)\cap \mathbf{X}_1\rightarrow \mathbf{X}_1$ the restriction of $\xA_1$ on the space~$\mathbf{X}_1$. In this case, the set of stationary solutions of the problem formed by Equations~\eqref{uu}--\eqref{init cond 0} and~\eqref{div B}, with the boundary conditions \eqref{lk} and~\eqref{lkm}, is equal to 
\begin{equation*}
\ker({\xA_{1}}_{|\mathbf{X}_1})=\lbrace 0 \rbrace^3 \times \mathbf{Z}_{T}(\Omega), 
\end{equation*}
where the kernel 
\begin{equation*}
\mathbf{Z}_{T}(\Omega) := \mathbf{H}(\curl 0;\Omega)\cap \mathbf{H}_0(\divg 0; \Omega). 
\end{equation*}
We recall that the space $\mathbf{Z}_{T}(\Omega)$ is of dimension $J$, the number of cuts~\cite{ABD98,ACL+17} 
(if $\Omega$ is topologically trivial then $\mathbf{Z}_{T}(\Omega)=\lbrace 0 \rbrace$). Consider $(\widetilde{\grad \dot{q}_j})_{1\leq j \leq J}$ a basis of $\mathbf{Z}_{T}(\Omega)$ given by~\cite[Proposition~3.14]{ABD98} where $\dot{q}_j\in\xHone(\dot{\Omega})$ is a function such that (among other conditions) $\langle \widetilde{\grad \dot{q}_j} \cdot \boldsymbol{n}, 1 \rangle_{\Sigma_i}=\delta_{ji}$ for $i=1,\ \ldots,\ J$.  
From this basis we deduce the following orthogonal decomposition in $\mathbf{H}_0(\divg 0;\Omega)$:
\begin{equation}
\mathbf{H}_0(\divg 0;\Omega)={\mathbf{Z}_{T}(\Omega)}\stackrel{\perp}{\oplus}{\mathbf{H}_0^{\Sigma}(\divg 0;\Omega)},
\label{orth decom ZT}
\end{equation}
where 
\begin{equation*}
\mathbf{H}_0^{\Sigma}(\divg 0;\Omega):=\lbrace \boldsymbol{v}\in \mathbf{H}_0(\divg 0;\Omega): \langle \boldsymbol{v} \cdot \boldsymbol{n}, 1 \rangle_{\Sigma_j}=0,~ 1\leq j \leq J\rbrace. 
\end{equation*}
Then, according to~\eqref{orth decom ZT} and by \cite[Propositions~3.7.3 and 3.7.4]{ACL+17}, we have the following decomposition 
\begin{equation} 
\mathbf{L}^2(\Omega) = \mathbf{H}(\curl 0;\Omega) \stackrel{\perp}{\oplus} \mathbf{H}_0^{\Sigma}(\divg 0;\Omega).
\label{orth decom Hrot}
\end{equation}
\begin{prpstn}
For all $\alpha\in\xR\setminus\lbrace0\rbrace$, the operator $ \ii\alpha\,\mathbb{I}+\xA_1$ is surjective, \ie,
\begin{equation*}
\xim(\ii\alpha\,\mathbb{I}+\xA_1)=\mathbf{X}.
\end{equation*}
\label{surjj*}
\end{prpstn}
\begin{proof}
We take any  $\alpha\in\xR\setminus\lbrace0\rbrace$ and any $\boldsymbol{F}= ( \boldsymbol{F}_1, \boldsymbol{F}_2, \boldsymbol{F}_3, \boldsymbol{F}_4
)^\top \in \mathbf{X}$, and we look for $\boldsymbol{U}= ( \boldsymbol{U}_1, \boldsymbol{U}_2, \boldsymbol{U}_3, \boldsymbol{U}_4, )^\top \in D(\xA_1)$, which solves
\begin{equation}
(\ii\alpha\,\mathbb{I}+\xA_1)\,\boldsymbol{U}=\boldsymbol{F}.
\label{000tild}
\end{equation}
Equivalently, according to~\eqref{matrix A}, we consider the following system 
\begin{eqnarray}
\ii\alpha\,\boldsymbol{U}_1 + \mathbb{M}_1\boldsymbol{U}_1 - \varepsilon_0\omega_{p1}^2\, \boldsymbol{U}_3 &=&\boldsymbol{F}_1,
\label{11tild}\\
\ii\alpha\,\boldsymbol{U}_2 + \mathbb{M}_2\boldsymbol{U}_2 - \varepsilon_0\omega_{p2}^2\, \boldsymbol{U}_3 &=&\boldsymbol{F}_2,
\label{22tild}\\
\ii\alpha\,\boldsymbol{U}_3 + \frac{1}{\varepsilon_0}\, \boldsymbol{U}_1 + \frac{1}{\varepsilon_0}\,\boldsymbol{U}_2 - c^2\curl\boldsymbol{U}_4 &=&\boldsymbol{F}_3,
\label{33tild}\\
\ii\alpha\,\boldsymbol{U}_4+\curl \boldsymbol{U}_3 &=&\boldsymbol{F}_4.
\label{44tild}
\end{eqnarray}
Using \eqref{11tild}, \eqref{22tild} and~\eqref{44tild}, we keep $\boldsymbol{U}_3$ as the main unknown and eliminate the others:
\begin{eqnarray}
\boldsymbol{U}_1 &=& (\ii\alpha\mathbb{I}+\mathbb M_{1})^{-1}(\boldsymbol{F}_1 + \varepsilon_0\,\omega_{p1}^2\,\boldsymbol{U}_3),\label{441tild}\\ 
\boldsymbol{U}_2 &=& (\ii\alpha\mathbb I + \mathbb M_{2})^{-1}(\boldsymbol{F}_2 +\varepsilon_0\,\omega_{p2}^2\,\boldsymbol{U}_3),\label{442tild}\\  
\boldsymbol{U}_4 &=& (\ii\alpha)^{-1}(\boldsymbol{F}_4 -{\curl\boldsymbol{U}_3})
\label{443tild}.
\end{eqnarray} 
Inserting these expressions into~\eqref{33tild}, we obtain an equation in~$\boldsymbol{U}_3$: 
\begin{equation}
\ii\alpha\, \boldsymbol{U}_3 + \frac{c^2}{\ii\alpha}\, \curl\curl\boldsymbol{U}_3 + \mathbb{D}_{\alpha}\boldsymbol{U}_3
= \boldsymbol{F}_3 + \frac{c^2}{\ii\alpha}\, \curl \boldsymbol{F}_4 -\frac{1}{\varepsilon_0}\sum_s (\ii\alpha\mathbb{I}+\mathbb M_{s})^{-1}{\boldsymbol{F}_s}.
\label{55tild}
\end{equation}
Here, we cannot apply the Lax--Milgram theorem as in Proposition~\Rref{maximonot 1}: the operator on the left-hand side (even suitably rescaled) is not positive. So, we shall solve this problem with a suitable version of the Fredholm alternative for constrained problems, as in~\cite[\S4.5.1]{ACL+17}. Taking account of the constraints is necessary in order to give some compactness properties (by Theorem~\Rref{compact}), which are not furnished by the definition of our evolution operator.

\medbreak

Thus, we introduce the following mixed formulation for~\eqref{55tild}:\\
\emph{Find $(\boldsymbol{U}_3,p)\in\mathbf{H}_0(\curl;\Omega) \times \xHone_{\partial\Omega}(\Omega)$ such that}
\begin{eqnarray}
a_{\alpha}(\boldsymbol{U}_3,\boldsymbol{v})+c_\alpha(\boldsymbol{U}_3,\boldsymbol{v})+\overline{b_\alpha(\boldsymbol{v},p)}
 &=& L_\alpha(\boldsymbol{v}),\quad\forall\boldsymbol{v}\in\mathbf{H}_0(\curl; \Omega),\label{555tild}\\
b_\alpha(\boldsymbol{U}_3,q)
 &=& \left(\boldsymbol{G}  \mid \grad q\right),\label{666tild}
\end{eqnarray}
where the sesquilinear forms $a_{\alpha}$, $c_\alpha$ and $b_\alpha$ are respectively defined on $\mathbf{H}_0(\curl;\Omega) \times \mathbf{H}_0(\curl;\Omega)$, $\mathbf{L}^2(\Omega) \times \mathbf{H}_0(\curl;\Omega)$ and $\mathbf{H}_0(\curl;\Omega) \times \xHone_{\partial\Omega}(\Omega)$ as:
\begin{eqnarray}
a_{\alpha}(\boldsymbol{w},\boldsymbol{v}) &:=&( \ii\alpha)^{-1} c^2 ( \curl\boldsymbol{w}\mid\curl\boldsymbol{v} ),
\label{a cond}
\\
c_\alpha(\boldsymbol{w},\boldsymbol{v}) &:=& (\mathbb{B}_{\alpha}\boldsymbol{w}\mid \boldsymbol{v} ), \quad\text{ where }\quad \mathbb{B}_{\alpha} = \ii\alpha\mathbb{I}+\mathbb{D}_{\alpha},
\label{c cond}
\\
b_\alpha(\boldsymbol{v},q) &:=& (\mathbb{B}_{\alpha}\boldsymbol{v}  \mid \grad q ).
\label{b cond}
\end{eqnarray}
The anti-linear form $L_\alpha$ on~$\mathbf{H}_0(\curl;\Omega)$ is given by:
\begin{equation}
L_\alpha(\boldsymbol{v}):=\left(\boldsymbol{F}_3\mid\boldsymbol{v}\right)+ \frac{c^2}{\ii\alpha}\, (\boldsymbol{F}_4\mid\curl\boldsymbol{v}) - \frac{1}{\varepsilon_0}\sum_s \left( (\ii\alpha\mathbb{I}+\mathbb M_{s})^{-1}{\boldsymbol{F}_s}\mid\boldsymbol{v}\right),
\label{l cond}  
\end{equation}
and $\boldsymbol{G}$ is an element of $\mathbf{L}^2(\Omega)$ which will be chosen later.

\medbreak

To show the well-posedness of the variational formulation \eqref{555tild}--\eqref{666tild}, we first verify that the assumptions of~\cite[Theorem 4.5.9]{ACL+17} on Helmholtz-like problems with constraints are fulfilled, and we conclude by the Fredholm alternative.
\\[3pt]
\textit{\texttt{i) Continuity:}} it is clear that the sesquilinear forms $a_\alpha$, $b_\alpha$ and $c_\alpha$, and the anti-linear from $L_\alpha$, are continuous on their respective spaces.
\\[3pt]
\textit{\texttt{ii) Coercivity on the kernel:}} the kernel of  $b_\alpha$ is defined by
\begin{equation*}
\mathrm{K}
=\lbrace\boldsymbol{v}\in \mathbf {H}_0(\mathbf  {\curl};\Omega)~\text{ : }~ b_\alpha(\boldsymbol{v},q)=0,\quad  \forall q\in\xHone_{\partial\Omega}(\Omega)\rbrace
\end{equation*}
which, by Green's formulas \eqref{green div grad h 1 0} and~\eqref{green div grad}, can be written as
\begin{eqnarray*}
\mathrm{K}
=\lbrace\boldsymbol{v}\in \mathbf {H}_0(\mathbf {\curl};\Omega)\cap\mathbf{H}(\divg\mathbb{B}_{\alpha}0;\Omega)~\text{ : }~ \langle \mathbb{B}_{\alpha}\boldsymbol{v}\cdot \boldsymbol{n} , 1 \rangle_{\xHn{\frac1{2}}({\Gamma_k})}=0,~ \forall 1 \leq k \leq K\rbrace \\
=\lbrace\boldsymbol{v}\in \mathbf {H}_0(\mathbf {\curl};\Omega)\cap\mathbf{H}(\divg\mathbb{B}_{\alpha}0;\Omega)~\text{ : }~ P_{\mathbf{Z}_{N}(\Omega;\mathbb{B_\alpha})}\boldsymbol{v}=0 \rbrace.\qquad\qquad\qquad\qquad\quad \,
\end{eqnarray*}
According to Corollary~\Rref{coro semi norm} applied to the closed subspace $\mathrm{K}$ of~$\mathbf{X}_{N,\Gamma}(\Omega; \mathbb{B})$, the sesquilinear form  $a_\alpha$ is coercive on $\mathrm{K}\times\mathrm{K}$.
Furthermore, the embedding $\mathrm{K} \subset \mathbf{L}^2(\Omega)$ is compact by Theorem~\ref{compact}. 
\\[3pt]
\textit{\texttt{iii) Inf-sup condition:}} let $q\in\xHone_{\partial\Omega}(\Omega)$ and set $\boldsymbol{v}=\grad q\in\mathbf {H}_0(\curl0;\Omega)$, thus $\boldsymbol{v}$ checks 
$\|\boldsymbol{v}\|_{\mathbf{H}(\curl;\Omega)}
=(\|\boldsymbol{v}\|^2+\|\curl\boldsymbol{v}\|^2)^{\frac{1}{2}}=\|\boldsymbol{v}\|$.
On the other hand, according to Lemma \Rref{lem minor B i alpha}, one has
\begin{eqnarray}
|b_\alpha(\boldsymbol{v},q)|
&=&|\left(\mathbb{B}_{\alpha}\boldsymbol{v}\mid \grad q\right)|=|\left(\mathbb{B}_{\alpha}\boldsymbol{v}\mid \boldsymbol{v}\right)|
\nonumber\\
&\geq &\zeta_{\alpha}\|\boldsymbol{v}\|^2=\zeta_{\alpha}\|\boldsymbol{v}\| \,\|\grad q\|.
\label{b in'eg}
\end{eqnarray} 
Combining the above, it follows that  
\begin{equation}
\frac{|b_\alpha(\boldsymbol{v},q)|}{\|\boldsymbol{v}\|_{\mathbf{H}(\curl;\Omega)}}\geq\zeta_{\alpha}\|\grad q\|=\zeta_{\alpha}\| q\|_{\xHone_{\partial\Omega}(\Omega)}.
\label{bv in'eg}
\end{equation}
Consequently, there exists  $C_b=\zeta_{\alpha}>0$ such that  
\begin{equation*}
\forall q\in\xHone_{\partial\Omega}(\Omega),\quad \sup_{\boldsymbol{v}\in\mathbf{H}_0(\curl;\Omega)}\frac{|b_\alpha(\boldsymbol{v},q)|}{\|\boldsymbol{v}\|_{\mathbf{H}(\curl;\Omega)}}\geq C_b\,\| q\|_{\xHone_{\partial\Omega}(\Omega)}.
\end{equation*}
Hence, the assumptions of  \cite[Theorem~4.5.9]{ACL+17} are satisfied: we can apply the usual Fredholm alternative \cite[Theorem~4.5.7]{ACL+17}. So, we show that the variational formulation \eqref{555tild} is injective on the kernel, \ie, its solution is unique. Let $\boldsymbol{Z}_3$ be a solution to 
\begin{equation}
\forall\boldsymbol{v}\in\mathrm{K},\quad a_\alpha(\boldsymbol{Z}_3,\boldsymbol{v}) + c_\alpha(\boldsymbol{Z}_3,\boldsymbol{v}) = 0.
\label{initialtildd}
\end{equation} 
Since $\boldsymbol{Z}_3$ belongs to  $\mathrm{K}$, one has $\boldsymbol{Z}_3\in\mathbf{H}(\curl;\Omega)$ with 
$\divg\mathbb{B}_{\alpha}\boldsymbol{Z}_3=0$ in $\Omega$ and 
$\boldsymbol{Z}_3\times\boldsymbol{n}_{|\Gamma}=0$. 
Next, consider $\boldsymbol{y}\in\boldsymbol{\mathcal{D}}(\Omega)$. Introduce the scalar field  $\varphi\in\xHone_{\partial\Omega}(\Omega)$ that solves the variational formulation: for all  $\psi\in\xHone_{\partial\Omega}(\Omega)$,  $\left(\mathbb{B}_{\alpha}\grad \varphi\mid \grad \psi\right)=\left(\mathbb{B}_{\alpha}\boldsymbol{y}\mid \grad \psi\right)$. By construction, $\boldsymbol{v}:=\boldsymbol{y}-\grad \varphi$ belongs to $\mathrm{K}$ with $\curl\boldsymbol{v}=\curl\boldsymbol{y}$. Using it as a test function in~\eqref{initialtildd} yields 
\begin{eqnarray*}
\langle(\ii\alpha)^{-1}c^2\curl\curl\boldsymbol{Z}_3+\mathbb{B}_{\alpha}\boldsymbol{Z}_3 ,\boldsymbol{y} \rangle 
&=&(\ii\alpha)^{-1} c^2\left( \curl\boldsymbol{Z}_3\mid\curl\boldsymbol{y}\right)+\left(\mathbb{B}_{\alpha}\boldsymbol{Z}_3\mid \boldsymbol{y}\right)\\
&=&(\ii\alpha)^{-1} c^2\left( \curl\boldsymbol{Z}_3\mid\curl\boldsymbol{v}\right)+\left(\mathbb{B}_{\alpha}\boldsymbol{Z}_3\mid \boldsymbol{v}+\grad \varphi\right)\\
&=&(\ii\alpha)^{-1} c^2\left( \curl\boldsymbol{Z}_3\mid\curl\boldsymbol{v}\right)+\left(\mathbb{B}_{\alpha}\boldsymbol{Z}_3\mid \boldsymbol{v}\right)=0.
\end{eqnarray*}
The last line is obtained by integration by parts, using the facts that $\divg(\mathbb{B}_{\alpha}\boldsymbol{Z}_3)=0$ in $\Omega$ and $\langle \mathbb{B}_{\alpha}\boldsymbol{Z}_3\cdot \boldsymbol{n} , 1 \rangle_{{\mathbf{H}}^{\frac1{2}}({\Gamma_k})}=0$, for all $1 \leq k \leq K$.  Recall that  $ \mathbb{B}_{\alpha}=\ii\alpha\mathbb{I}+\mathbb{D}_{\alpha}$; as $\boldsymbol{y}$ is arbitrary, it follows that:
\begin{equation}
\ii\alpha\boldsymbol{Z}_3+(\ii\alpha)^{-1}{c^2}\curl {\curl\boldsymbol{Z}_3}+\mathbb{D}_{\alpha}\boldsymbol{Z}_3
= 0,\quad \mbox{in}\quad \boldsymbol{\mathcal{D}}'(\Omega).
\label{55tildd}
\end{equation}
Let now $\boldsymbol{Z}_1$, $\boldsymbol{Z}_2$ and $\boldsymbol{Z}_4$ defined as
\begin{equation}
(\ii\alpha\mathbb{I}+\mathbb M_{s})\boldsymbol{Z}_s
= \varepsilon_0\,\omega_{ps}^2\,\boldsymbol{Z}_3,~ s=1,\ 2 ~;\quad\quad\quad \boldsymbol{Z}_4=-(\ii\alpha)^{-1}\curl \boldsymbol{Z}_3.
\label{s4}
\end{equation}
Clearly, these fields belong to~$\mathbf{L}_{(s)}^2(\Omega)=\mathbf{L}^2(\Omega)$.  Replacing the matrix~$\mathbb{D}_{\alpha}$ with its expression in~\eqref{55tildd} and using the definitions above, we get
\begin{equation}
\ii\alpha\boldsymbol{Z}_3+\frac{1}{\varepsilon_0}\boldsymbol{Z}_1+\frac{1}{\varepsilon_0}\boldsymbol{Z}_2- c^2\curl\boldsymbol{Z}_4=0,
\label{s3}
\end{equation}
which implies that $\curl\boldsymbol{Z}_4\in\mathbf{L}^2(\Omega)$. The equations \eqref{s4}--\eqref{s3} are equivalent to $(\ii\alpha\mathbb{I}+\xA_1)\boldsymbol{Z} =0$, with $\boldsymbol{Z} =(\boldsymbol{Z}_1,\boldsymbol{Z}_2,\boldsymbol{Z}_3,\boldsymbol{Z}_4)^\top \in D(\xA_1)$. 
Therefore, according to Proposition~\Rref{ker A1}, one finds $\boldsymbol{Z} =0$, and so $\boldsymbol{Z}_3=0$ in~$\Omega$. Thus, the formulation~\eqref{initialtildd} is injective. 
Finally, according to Theorem~4.5.9 and Proposition~4.5.8 of~\cite{ACL+17}, the problem \eqref{555tild}--\eqref{666tild} has a unique solution  $(\boldsymbol{U}_3,p)\in\mathbf{H}_0(\curl;\Omega) \times\xHone_{\partial\Omega}(\Omega)$.

\medbreak
 
To show the equivalence between~\eqref{555tild} and the strong formulation~\eqref{55tild}, we have to check that the Lagrange multiplier $p$ vanishes. Taking $\boldsymbol{v}=\grad p \in \mathbf {H}_0(\mathbf{\curl};\Omega)$ as a test function in~\eqref{555tild}, we obtain 
\begin{eqnarray}
\left(\mathbb{B}_{\alpha}\boldsymbol{U}_3\mid \grad p \right)+\left(\mathbb{B}_{\alpha}\grad p \mid \grad p \right)=\left(\boldsymbol{F}_3\mid\grad p\right)-\frac{1}{\varepsilon_0}\sum_s \left( (\ii\alpha\mathbb{I}+\mathbb M_{s})^{-1}{\boldsymbol{F}_s}\mid\grad p\right).
\label{inter1tild cond par}
\end{eqnarray}
The first term above is the left-hand side of the constraint equation~\eqref{666tild}. Thus, choosing 
\begin{equation}
\boldsymbol{G}:=\boldsymbol{F}_3-\frac{1}{\varepsilon_0}\sum_s (\ii\alpha\mathbb{I}+\mathbb M_{s})^{-1}\boldsymbol{F}_s
\in \mathbf{L}^2(\Omega),
\label{choix de g}
\end{equation} 
we get, according to  \eqref{inter1tild  cond par} and~\eqref{666tild},
\begin{equation*}
\left(\mathbb{B}_{\alpha}\grad p \mid \grad p \right)=0.
\end{equation*}
Thanks to Lemma \Rref{lem minor B i alpha}, we deduce that $\grad p=0$ in $\Omega$. As $p$ belongs to $\xHone_{\partial\Omega}(\Omega)$,  we find  $p=0$. 

\medbreak

Returning to Problem~\eqref{000tild}, we define $\boldsymbol{U}_1\in \mathbf{L}^2_{(1)}(\Omega)$, $\boldsymbol{U}_2\in \mathbf{L}^2_{(2)}(\Omega)$ respectively by \eqref{441tild} and~\eqref{442tild}. Also, we define $\boldsymbol{U}_4 \in \mathbf{L}^2(\Omega)$ by~\eqref{443tild}. Taking $\boldsymbol{v}\in\boldsymbol{\mathcal{D}}(\Omega)$ as a test function in~\eqref{555tild}, replacing $\mathbb{B}_{\alpha}$ with its expression and using Green's formula~\eqref{green rot rot}, we obtain Equation~\eqref{55tild}, and by the definition~\eqref{443tild} of $\boldsymbol{U}_4$ we find 
\begin{equation}
\ii\alpha \boldsymbol{U}_3 - {c^2}\curl \boldsymbol{U}_4 + \mathbb{D}_{\alpha}\boldsymbol{U}_3
= \boldsymbol{F}_3  -\frac{1}{\varepsilon_0}\sum_s (\ii\alpha\mathbb{I} +\mathbb M_{s})^{-1}\boldsymbol{F}_s \quad \text{in } \boldsymbol{\mathcal{D}}'(\Omega).
\label{inter2tild}
\end{equation}
This implies that $\curl\boldsymbol{U}_4\in\mathbf{L}^2(\Omega)$. To finish the proof, it remains to check that Equation~\eqref{33tild} is satisfied: to this end, it is enough to replace in~\eqref{inter2tild}  the matrix $\mathbb{D}_{\alpha}$ with its definition and to use  \eqref{441tild} and~\eqref{442tild}. 
\end{proof}

\medbreak

Let us introduce a closed subspace of~$\mathbf{X}$:
\begin{equation} 
\mathbf{\widetilde{X}}_1:=\mathbf{L}{^2}(\Omega)\times \mathbf{L}{^2}(\Omega) \times \mathbf{L}{^2}(\Omega) \times\mathbf{H}_0^{\Sigma}(\divg 0;\Omega).
\label{eq-def-Xtild1}
\end{equation}
Of course, $\mathbf{\widetilde{X}}_1$ is a Hilbert space when endowed with the inherited inner product.
\begin{prpstn}
The range $\xim(\xA_{1})$ of 
$\xA_{1}$ is  included in~$\mathbf{\widetilde{X}}_{1}$.
\label{range-Atild1}
\end{prpstn}
\begin{proof}
Consider  $\boldsymbol{U}=(\boldsymbol{U}_1,\boldsymbol{U}_2,\boldsymbol{U}_3,\boldsymbol{U}_4)^\top$ an element of $D(\xA_1)$. Then, by the definition~\eqref{A1 definition} of~$\xA_1$,  $\xA_1\boldsymbol{U}$ belongs to $\mathbf{\widetilde{X}}_1$ if, and only if, $ \curl\boldsymbol{U}_3$ belongs to $\mathbf{H}_0^{\Sigma}(\mathbf{\divg}0;\Omega)$. But $ \boldsymbol{U}_3\in\mathbf {H}_0(\mathbf{\curl};\Omega)$, and it is well-known (see, \vg, \cite[Remark~3.5.2]{ACL+17}) that $\boldsymbol{v}\in\mathbf {H}_0(\mathbf{\curl};\Omega)$ implies $\curl\boldsymbol{v} \in \mathbf{H}_0^{\Sigma}(\mathbf{\divg}0;\Omega)$.
\end{proof}

\medbreak

The spectral analysis of the operator $\xA_{1}$ shows that no stabilization can  take place in the whole space~$\mathbf{X}$: an initial data $\boldsymbol{U}_0 \in \ker \xA_1$ generates a constant-in-time solution.
The above results lead us to introduce the unbounded operator 
${\widetilde{\xA}}_{1}: D({\widetilde{\xA}}_{1}) \rightarrow \widetilde{\mathbf{X}}_{1}$  defined by
\begin{equation}
D({\widetilde{\xA}}_{1})=D(\xA_{1})\cap\widetilde{\mathbf{X}}_{1}
\quad\text{and}\quad
{\widetilde{\xA}}_{1}\boldsymbol{U}=\xA_{1}\boldsymbol{U}, ~\forall \boldsymbol{U}\in D({\widetilde{\xA}}_{1}).
\label{eq-def-Atild1}
\end{equation}
The spectral properties of ${\widetilde{\xA}}_{1}$ are easily deduced from Propositions \Rref{ker A1} and~\Rref{surjj*}.
\begin{prpstn}
For all $\alpha\in\xR$, the operator $\ii\alpha\mathbb{I} + {\widetilde{\xA}}_{1}$ is injective. For $\alpha\in\xR \setminus \{0\}$, it is surjective.
\label{spec-Atild1}
\end{prpstn}
\begin{proof}
The injectivity for $\alpha\neq0$ directly follows from Proposition~\Rref{ker A1}. For $\alpha=0$, ${\widetilde{\xA}}_{1} \boldsymbol{U} = 0$ means $\boldsymbol{U} \in \widetilde{\mathbf{X}}_{1}$ and $\xA_{1}\boldsymbol{U}=0$, hence $\boldsymbol{U}_1=\boldsymbol{U}_2=\boldsymbol{U}_3=0$, $\boldsymbol{U}_4 \in \mathbf{H}(\curl 0;\Omega)$ and $\boldsymbol{U}_4 \in \mathbf{H}_0^{\Sigma}(\divg 0;\Omega)$. According to~\eqref{orth decom Hrot}, this implies $\boldsymbol{U}_4=0$. 

\medbreak

Taking account of Proposition~\Rref{surjj*}, the surjectivity property means that, if $\boldsymbol{F} \in \widetilde{\mathbf{X}}_{1}$, the unique solution $\boldsymbol{U}$ to $\ii\alpha\,\boldsymbol{U} + \xA_1\ \boldsymbol{U} = \boldsymbol{F}$ belongs to~$\widetilde{\mathbf{X}}_{1}$. This, in turn, is an obvious consequence of Proposition~\Rref{range-Atild1}.
\end{proof}

We notice that $\widetilde{\mathbf{X}}_{1}$ is an invariant space for the problem~\eqref{evol prob cp}, see Lemma~\Rref{lemma project P1B}. We then define  $\check{\mathit{T}_{1}}:={\mathit{T}_{1}}_{|D(\xA_{1})\cap\widetilde{\mathbf{X}}_{1}}$.
\begin{thrm}
The semigroup of contractions $(\check{\mathit{T}_{1}}(t))_{t\geq0}$ with generator 
$-{\widetilde{\xA}}_{1}$  is strongly stable on the energy space 
${\mathbf{\widetilde{X}}_{1}}$, \ie,
\begin{equation}
\lim\limits_{t \to +\infty}\| \check{\mathit{T}_{1}}(t)\widetilde{\boldsymbol{U}}_0\|_{\widetilde{\mathbf{X}}_1}=0,\quad \forall \widetilde{\boldsymbol{U}}_0 \in \widetilde{\mathbf{X}}_1.
\end{equation}
\label{stabi forte cp X1 til}
\end{thrm}
\begin{proof}
According to Proposition~\Rref{spec-Atild1}, we conclude that 
\begin{equation*}
\sigma(-\widetilde{\xA}_1)\cap \ii\xR=\varnothing \quad \text{or} \quad \lbrace 0 \rbrace, 
\end{equation*}
which is countable in both cases, and that $0$~is not an eigenvalue.
On the other hand, $\widetilde{\xA}_1$~is monotone in~$\widetilde{\mathbf{X}}_1$, so $-\widetilde{\xA}_1$ is dissipative in~$\widetilde{\mathbf{X}}_1$. The rest of the proof  follows from Theorem~\Rref{A B, L V}. 
\end{proof}
\begin{rmrk}
As we shall see in Section~\Rref{sec-expoly-stab}, $0$ actually does not belong to~$\sigma(-\widetilde{\xA}_1)$.
\end{rmrk}

\medbreak

We denote by $P_1$ the orthogonal  projection in $\mathbf{L}^2(\Omega)$ onto $\mathbf{Z}_{T}(\Omega)$. 
\begin{lmm}
Let $\boldsymbol{U}_0\in\mathbf{X}_1$ and $\boldsymbol{U}$ be the solution of problem \eqref{evol prob cp}. It holds  that 
\begin{equation}
P_1(\boldsymbol{B}(t))=P_1(\boldsymbol{B}_0),\quad \forall\, t>0.
\label{projec P1B}
\end{equation}
\label{lemma project P1B}
\end{lmm}
\begin{proof}
Just multiply Equation~\eqref{pp2} by a element of~$\mathbf{Z}_{T}(\Omega)$ and integrate by parts on~$\Omega$. 
\end{proof}
Lemma~\Rref{lemma project P1B}  shows that the projection of the solution $\boldsymbol{U}$ onto $\ker({\xA_{1}}_{|\mathbf{X}_1})$ does not depend on the time. Then we conclude\dots  
\begin{crllr}
It holds that 
\begin{equation*}
\lim_{t \to +\infty} \left\| {\mathit{T}_{1}(t)\,\boldsymbol{U}_0}- \sum_{1\leq j \leq J} \xi_j \, (0,0,0,\widetilde{\grad \dot{q}_j})^\top \right\|_{\mathbf{X}_1}=0,\quad\forall \boldsymbol{U}_{0} \in \mathbf{{X}}_1, 
\end{equation*}
where $\xi_j=\langle \boldsymbol{B}_0 \cdot \boldsymbol{n}, 1 \rangle_{\Sigma_j},\quad \text{for } j=1,\ \ldots,\  J$.
\label{strong stab X1}
\end{crllr}
\begin{proof}
Let $\boldsymbol{U}_0 \in \mathbf{X}_1$. From the orthogonal decomposition~\eqref{orth decom ZT} and Lemma~\Rref{lemma project P1B}, we deduce that the solution $\boldsymbol{U}$ to the system  \eqref{uu}--\eqref{init cond 0} and~\eqref{div B}, with boundary conditions \eqref{lk} and~\eqref{lkm}, can be written as:
\begin{equation*}
\boldsymbol{U}(t)= \widetilde{\boldsymbol{U}}(t) + (0,0,0,P_1\boldsymbol{B}_0),
\end{equation*}
and  $\widetilde{\boldsymbol{U}}(t)\in\widetilde{\mathbf{X}}_1$ is the solution of problem
\begin{equation*}
\partial_{t}\widetilde{\boldsymbol{U}}(t) + \widetilde{\xA}_1\widetilde{\boldsymbol{U}}(t)=0,\, \text{ for }  t > 0,  \qquad  \widetilde{\boldsymbol{U}}(0)=\widetilde{\boldsymbol{U}}_0,
\end{equation*} 
where the initial condition $\widetilde{\boldsymbol{U}}_0=\boldsymbol{U}_0-(0,0,0,P_1\boldsymbol{B}_0)$  belongs to~$\widetilde{\mathbf{X}}_1$. Next, let $\xi_j$, $j=1,\ \ldots,\ J$ be a constants such that $P_1\boldsymbol{B}_0=\sum_{1\leq j \leq J}\xi_j\,\widetilde{\grad \dot{q}_j}$. Therefore, $\boldsymbol{B}_0-P_1\boldsymbol{B}_0$ belongs to $\mathbf{H}_0^{\Sigma}(\divg 0;\Omega)$, which yields:
\begin{equation*}
\langle \boldsymbol{B}_0 \cdot \boldsymbol{n}, 1 \rangle_{\Sigma_i}
= \sum_{j}\xi_j \langle \widetilde{\grad \dot{q}_j} \cdot \boldsymbol{n}, 1 \rangle_{\Sigma_i}
= \xi_{i}.
\end{equation*}
Finally, from Theorem \Rref{stabi forte cp X1 til}, $\widetilde{\boldsymbol{U}}$ satisfies $\lim\limits_{t \to +\infty}\| \widetilde{\boldsymbol{U}}(t)\|_{\widetilde{\mathbf{X}}_1}=0$, hence the result.
\end{proof}

\subsection{Homogeneous Silver--M\"uller case}
\label{strong-sm}
\begin{prpstn}
For all $\alpha\in\xR \setminus \{0\}$, the operator 
$\ii\alpha\mathbb{I}+\xA_{2}$ is injective, \ie,
\begin{equation*}
\ker(\ii\alpha\mathbb{I}+\xA_{2})=\lbrace0\rbrace.
\end{equation*}
Furthermore, $0$ is an eigenvalue of $\xA_{2}$ and  the set of its eigenvectors is 
\begin{equation*}
\ker(\xA_{2})=\lbrace (0,0,0,\boldsymbol{V})\, :\, \boldsymbol{V} \in \mathbf{H}_{0,\Gamma_A}(\curl 0;\Omega)\rbrace.
\end{equation*}
\label{inject sm stab}
\end{prpstn}
\begin{proof}
Let $\alpha\in\xR$ and $\boldsymbol{U} = (\boldsymbol{U}_1,\boldsymbol{U}_2,\boldsymbol{U}_3,\boldsymbol{U}_4)^\top \in D(\xA_2)$ be such that 
\begin{equation}
(\ii\alpha\,\mathbb{I}+\xA_2)\,\boldsymbol{U}=0 ,
\label{0}
\end{equation}
which is equivalent, in~$\Omega$, to the system~\eqref{1injec}--\eqref{4injec}. Taking the inner product of \eqref{0} with~$\boldsymbol{U}$, one gets: 
\begin{equation*}
\Re(\ii\alpha\|\boldsymbol{U}\|_{\mathbf{X}}^2) = \Re\left(\xA_2 \boldsymbol{U}\mid\boldsymbol{U}\right)_{\mathbf{X}}=0.
\end{equation*}
By the monotonicity of $\xA_2$, see~\eqref{monotonie A2 sm homo}, one obtains
\begin{equation*}
\left( \frac{\nu_s\, \boldsymbol{U}_s}{\varepsilon_0\omega_{ps}} \Biggm| \frac{\boldsymbol{U}_s}{\omega_{ps}} \right) = 0, \quad s=1,\ 2,
\quad \text{and}\quad \| \boldsymbol{U}_{4\top} \|^2_{\mathbf{L}^2(\Gamma_A)}=0.
\end{equation*}
The rest of the proof follows the same arguments as Proposition~\Rref{ker A1}.
\end{proof}

\medbreak

The above Proposition states that $\ker \xA_2 $ coincides with the set of stationary solutions to the problem~\eqref{uu}--\eqref{init cond 0} with boundary condition \eqref{lk} and~\eqref{lm} (with $\boldsymbol{g}=0$). Similarly to the operator $\xA_1$, if we define  the operator ${\xA_2}_{|\mathbf{X}_2}: D(\xA_2)\cap \mathbf{X}_2\rightarrow \mathbf{X}_2$ as the restriction of $\xA_2$ to the space~$\mathbf{X}_2$, then we obtain
\begin{eqnarray*}
\ker({\xA_{2}}_{|\mathbf{X}_2})=\lbrace 0 \rbrace^3 \times \mathbf{Z}(\Omega;\Gamma_A), 
\end{eqnarray*}
where the kernel 
\begin{eqnarray*}
\mathbf{Z}(\Omega;\Gamma_A):=\mathbf{H}_{0,\Gamma_A}(\curl 0;\Omega)\cap \mathbf{H}_{0,\Gamma_P}(\divg 0;\Omega).
\end{eqnarray*}
Note that the set of stationary solution to Equations~\eqref{uu}--\eqref{init cond 0} and~\eqref{div B}, with boundary conditions \eqref{lk}--\eqref{lkm} and $\boldsymbol{g}=0$, is equal to $\ker({\xA_{2}}_{|\mathbf{X}_2})$.

\medbreak

The space $\mathbf{Z}(\Omega;\Gamma_A)$ has been studied by Fernandes and Gilardi in~\cite{FG97}. It is of finite dimension and from  \cite[Corollarly~5.2]{FG97} one has $\dim{\mathbf{Z}(\Omega;\Gamma_A)}\leq N+J-1$. 
(We recall that $N$ is the number of connected components of~$\Gamma_A \setminus \partial\Sigma$.)

\smallbreak

We now recall some orthogonal decompositions from~\cite{FG97}; we mostly keep the same notations. Picking a vector $\textbf{a}\in\xR^N$ such that $\sum_{i=1}^{N} a_i \neq 0$, we define the space 
\begin{eqnarray*}
\xHn{\frac1{2}}_{\text{const}\,\Gamma_A,\Sigma}(\partial\dot{\Omega};\textbf{a}):=\lbrace p\in\xHn{\frac1{2}}(\partial\dot{\Omega}):\, \exists\textbf{c}\in\xR^N, \, \exists\textbf{c}'\in \xR^J \,:\,  \textbf{c}\cdot\textbf{a}=0,\\
p_{|\Gamma_{A,i}}=c_i\, \text{ for } i=1,\ \ldots,\ N\,; \quad [p]_{\Sigma_j}=c'_j\, \text{ for }\, j=1,\ \ldots,\ J \rbrace.
\end{eqnarray*}
Moreover we introduce 
\begin{eqnarray*}
\mathbf{H}_{0,\Gamma_P;\text{flux},\Gamma_A,\Sigma}(\divg 0;\Omega) := \lbrace \boldsymbol{v}\in \mathbf{H}_{0,\Gamma_P}(\divg 0;\Omega)\,:\,\qquad\qquad \nonumber\\
~\mbox{ }\langle \boldsymbol{v} \cdot \boldsymbol{n}, p \rangle_{\xHn{\frac1{2}}(\Gamma_A\cup\Sigma)}=0\quad \forall\, p\in \xHn{\frac1{2}}_{\text{const}\,\Gamma_A,\Sigma}(\partial\dot{\Omega};\textbf{a}) \rbrace .
\end{eqnarray*}
For the proof, we refer the reader to \cite[Proposition~3.3 and Remark~2.1]{FG97}.
Then, we have the orthogonal decompositions in $\mathbf{L}^2(\Omega)$ which are proven in \cite[Propositions~6.3 and 6.4]{FG97}:
\begin{eqnarray}
\mathbf{H}_{0,\Gamma_P}(\divg 0;\Omega)&=&\qquad ~{\mathbf{Z}(\Omega;\Gamma_A)}~~\,\stackrel{\perp}{\oplus}{\mathbf{H}_{0,\Gamma_P;\text{flux},\Gamma_A,\Sigma}(\divg 0;\Omega)},
\label{orth decom Z GammaA}\\
\mathbf{L}^2(\Omega)&=&{\mathbf{H}_{0,\Gamma_A}(\curl 0;\Omega)}\stackrel{\perp}{\oplus}{\mathbf{H}_{0,\Gamma_P;\text{flux},\Gamma_A,\Sigma}(\divg 0;\Omega)}.
\label{orth decom Hrot GammaA}
\end{eqnarray}

\begin{prpstn}
For all $\alpha\in\xR\setminus\lbrace0\rbrace$, the operator $\ii\alpha\mathbb{I}+\xA_2$ is surjective, $i.e$
\begin{equation*}
\xim(\ii\alpha\mathbb{I}+\xA_2)=\mathbf{X}.
\end{equation*}
\label{surj I+A2}
\end{prpstn}
\begin{proof}
We follow the lines of the proof of Proposition~\Rref{surjj*}.
Let $\alpha\in\xR\setminus\lbrace0\rbrace$ and $\boldsymbol{F} = (\boldsymbol{F}_1,\boldsymbol{F}_2,\boldsymbol{F}_3,\boldsymbol{F}_4)^\top \in \mathbf{X}$; we look for $\boldsymbol{U} = (\boldsymbol{U}_1,\boldsymbol{U}_2,\boldsymbol{U}_3,\boldsymbol{U}_4)^\top \in D(\xA_2)$ which solves:
\begin{equation}
(\ii\alpha\,\mathbb{I}+\xA_2)\,\boldsymbol{U} = \boldsymbol{F},
\label{00}
\end{equation}
which is equivalent to the system~\eqref{11tild}--\eqref{44tild}, with different boundary conditions. Again, we eliminate $\boldsymbol{U}_1$, $\boldsymbol{U}_2$ and $\boldsymbol{U}_4$ by \eqref{441tild}, \eqref{442tild} and~\eqref{443tild} respectively, while $\boldsymbol{U}_3$ verifies the equation \eqref{55tild} in~$\Omega$. 
Given the Silver--M\"uller boundary condition, the mixed formulation of~\eqref{55tild} writes --- recall the space $\mathcal{V}$ from~\eqref{def-space-V}:\\[3pt]
\emph{Find $ (\boldsymbol{U}_3,p)\in\mathcal{V} \times\xHone_{\partial\Omega}(\Omega)$ such that}
\begin{eqnarray}
\widetilde{a}_\alpha(\boldsymbol{U}_3,\boldsymbol{v}) + c_\alpha(\boldsymbol{U}_3,\boldsymbol{v}) + b_\alpha(\boldsymbol{v},p)
&=& L_\alpha(\boldsymbol{v}),\quad\forall\boldsymbol{v}\in\mathcal{V},\label{555}\\
b_\alpha(\boldsymbol{U}_3,q) &=& \left(\boldsymbol{G} \mid \grad q \right),\quad \forall q\in\xHone_{\partial\Omega}(\Omega),\label{666}
\end{eqnarray}
where the sesquilinear form $\widetilde{a}_\alpha$ is defined on~$\mathcal{V}\times\mathcal{V}$ as:
\begin{equation}
\widetilde{a}_\alpha(\boldsymbol{w},\boldsymbol{v}):= a_{\alpha}(\boldsymbol{w},\boldsymbol{v})+  c\, \left(\boldsymbol{w}_\top\mid \boldsymbol{v}_\top\right)_{\Gamma_{A}},
\label{a coer surj stab sm}
\end{equation}
the form $a_{\alpha}$ being defined in~\eqref{a cond}; on the other hand, $b_\alpha,\ c_\alpha,\ L_\alpha$ are as in~\eqref{c cond}--\eqref{l cond}, except that the variable $\boldsymbol{v}$ now belongs to~$\mathcal{V}$. Again, $\boldsymbol{G}$ is an element of $\mathbf{L}^2(\Omega)$ which will be chosen later.

\medbreak

Checking the hypotheses of~\cite[Theorem 4.5.9]{ACL+17} proceeds as in Proposition~\Rref{surjj*}. 
\\[3pt]
\textit{\texttt{i) Continuity:}} obvious.\\[3pt]
\textit{\texttt{ii) Coercivity on the kernel:}} the kernel of  $b_\alpha(.,.)$ is defined by 
\begin{equation*}
\mathrm{K} = \lbrace\boldsymbol{v}\in\mathcal{V}~\text{ : }~ b_\alpha(\boldsymbol{v},q)=0,\quad  \forall q\in\xHone_{\partial\Omega}(\Omega)\rbrace
\end{equation*}
which, by the Green formulas \eqref{green div grad h 1 0} and~\eqref{green div grad}, can be written as:
\begin{eqnarray*}
\mathrm{K}
=\lbrace\boldsymbol{v}\in\mathcal{V} \cap \mathbf {H}(\divg\mathbb{B}_{\alpha}0;\Omega)~\text{ : }~ \langle \mathbb{B}_{\alpha}\boldsymbol{v}\cdot \boldsymbol{n} , 1 \rangle_{\xHn{\frac1{2}}({\Gamma_k})}=0,~ \forall 1 \leq k \leq K\rbrace \\
=\lbrace\boldsymbol{v}\in\mathcal{V} \cap \mathbf {H}(\divg\mathbb{B}_{\alpha}0;\Omega)~\text{ : }~ P_{\mathbf{Z}_{N}(\Omega;\mathbb{B}_\alpha)}\boldsymbol{v}=0 \rbrace.\qquad\qquad\qquad\quad\quad ~~\,\,\,
\end{eqnarray*}
This kernel is compactly embedded into~$\mathbf{L}^2(\Omega)$ by Theorem~\Rref{compact}. Furthermore, according to Corollary~\Rref{coro semi norm}, the sesquiliear form $\widetilde{a}_\alpha$ is coercive on  $\mathrm{K}\times\mathrm{K}$. Indeed, taking $\boldsymbol{v}\in \mathrm{K}$, we find
\begin{equation*}
|\widetilde{a}_\alpha(\boldsymbol{v},\boldsymbol{v})| = \left|(\ii\alpha)^{-1} c^2 \|\curl\boldsymbol{v}\|^2+ c \,\|\boldsymbol{v}_\top\|_{\mathbf{L}^2(\Gamma_A)}^2 \right| 
= \left( (|\alpha |^{-1} c^2 \|\curl\boldsymbol{v}\|^2)^2+( c\, \|\boldsymbol{v}_\top\|_{\mathbf{L}^2(\Gamma_A)}^2)^2\right)^{\frac{1}{2}}.
\end{equation*}
But, we have the inequality
\begin{eqnarray*}
(z^2+y^2)^{\frac{1}{2}}\geq \frac{1}{\sqrt{2}}|z+y|,\quad \forall (z,y)\in\xR^2.
\end{eqnarray*}
Consequently, 
\begin{eqnarray*}
|\widetilde{a}_{\alpha}(\boldsymbol{v},\boldsymbol{v})|
&\geq &\frac{1}{\sqrt{2}}\, \left( |\alpha|^{-1} c^2 \|\curl\boldsymbol{v}\|^2+c\, \|\boldsymbol{v}_\top\|_{\mathbf{L}^2(\Gamma_A)}^2\right)\\
&\geq &\frac{1}{\sqrt{2}}\,\min\lbrace |\alpha|^{-1} c^2,c \rbrace\left( \|\curl\boldsymbol{v}\|^2+\|\boldsymbol{v}_\top\|_{\mathbf{L}^2(\Gamma_A)}^2\right)
= C\,|\boldsymbol{v}|_{\mathbf{X}_{N,\Gamma}(\Omega;\mathbb{B}_{\alpha})}^2 \,.
\end{eqnarray*}
\textit{\texttt{iii) Inf-sup condition:}}  take any $q\in\xHone_{\partial\Omega}(\Omega)$ and set $\boldsymbol{v}=\grad q$. Then, we have $\curl\boldsymbol{v}=0\in\mathbf{L}^2(\Omega)$ and $\boldsymbol{v}_\top=0\in\mathbf{L}^2(\Gamma)$,  thus $\boldsymbol{v}\in\mathcal{V}$ and verifies $\|\boldsymbol{v}\|_{\mathcal{V}}=\|\boldsymbol{v}\|$. The conclusion follows from the inequalities \eqref{b in'eg} and~\eqref{bv in'eg}.

\medbreak

As in the perfect conductor case, we can apply the Fredholm alternative. So, we show that the variational formulation~\eqref{555} is injective on the kernel. Let $\boldsymbol{Z}_3$ be a solution to the variational formulation
\begin{equation}
\forall\boldsymbol{v}\in\mathrm{K},\quad \widetilde{a}_\alpha(\boldsymbol{Z}_3,\boldsymbol{v}) + c_\alpha(\boldsymbol{Z}_3,\boldsymbol{v}) = 0.
\label{initial silv}
\end{equation}
Since $\boldsymbol{Z}_3$ belongs to  $\mathrm{K}$, one has $\boldsymbol{Z}_3\in\mathbf{H}(\curl;\Omega)$  with $\divg\mathbb{B}_\alpha \boldsymbol{Z}_3=0$ in $\Omega$ and $\boldsymbol{Z}_3\times\boldsymbol{n}_{|{\Gamma_P}}=0$. As in Proposition~\Rref{surjj*}, we  obtain the existence of $\boldsymbol{Z}_1,\,\boldsymbol{Z}_2\in \mathbf{L}^2(\Omega)$ and $\boldsymbol{Z}_4\in \mathbf{H}(\curl;\Omega)$ such that $(\ii\alpha\mathbb{I}+\xA)\boldsymbol{Z} = 0$, with $\boldsymbol{Z} =(\boldsymbol{Z}_1,\boldsymbol{Z}_2,\boldsymbol{Z}_3,\boldsymbol{Z}_4)^\top$. 
To apply Proposition~\Rref{inject sm stab}, we must check that $\boldsymbol{Z} \in D(\xA_2)$, \ie, the Silver--M\"uller condition is satisfied. For $\boldsymbol{v}\in \mathrm{K}$, using the integration-by-parts formula~\eqref{green rot 0 Gamma p} in~\eqref{initial silv} and Equation~\eqref{55tildd}, we get 
\begin{equation}
\left({\boldsymbol{Z}_3}_\top\mid \boldsymbol{v}_\top\right)_{\Gamma_{A}}- c\, 
{}_{\gamma_{A}}\langle \boldsymbol{Z}_4 \times \boldsymbol{n}, \boldsymbol{v}_\top \rangle_{\pi^0_{A}}=0, \quad \forall \boldsymbol{v} \in \mathrm{K}.
\label{silver Muller kernel}
\end{equation}
Now, consider any $\boldsymbol{y}\in\mathcal{V}$.  Let $\varphi$ be the unique element of $\xHone_{\partial\Omega}(\Omega)$   such that 
\begin{equation*}
\left(\mathbb{B}_{\alpha}\grad \varphi\mid \grad \psi\right)=\left(\mathbb{B}_{\alpha}\boldsymbol{y}\mid \grad \psi\right), \quad \forall \psi\in\xHone_{\partial\Omega}(\Omega).  
\end{equation*} 
So, $\boldsymbol{v}:=\boldsymbol{y}-\grad \varphi$ belongs to $\mathrm{K}$ with $\boldsymbol{v}_\top=\boldsymbol{y}_\top$ on $\Gamma_A$. Using it as a test function in~\eqref{silver Muller kernel}, we find
\begin{equation*}
\left({\boldsymbol{Z}_3}_\top\mid \boldsymbol{y}_\top\right)_{\Gamma_{A}}- c\, 
{}_{\gamma_{A}}\langle \boldsymbol{Z}_4 \times \boldsymbol{n}, \boldsymbol{y}_\top \rangle_{\pi^0_{A}}=0, \quad \forall \boldsymbol{y} \in \mathcal{V}.
\end{equation*}
The above equation is the same as~\eqref{int}, thus we obtain the Silver--M\"uller boundary condition as in the proof of  Proposition~\Rref{maximonot 2}.   
Consequently, $\boldsymbol{Z}$ belongs to $D(\xA_2)$,  and from  Proposition \Rref{inject sm stab} we infer that  $\boldsymbol{Z}=0$,  so $\boldsymbol{Z}_3=0$: the formulation \eqref{initial silv} is injective.

\medbreak

We deduce by Theorem~4.5.9 and Proposition~4.5.8 of \cite{ACL+17} that the problem~\eqref{555}--\eqref{666} admits a unique solution $(\boldsymbol{U}_3,p)\in\mathcal{V}\times\xHone_{\partial\Omega}(\Omega)$. Choosing $\boldsymbol{G}$ as in~\eqref{choix de g}, we get once again $p=0$. Thus, $\boldsymbol{U}_3$ satisfies~\eqref{55tild}, or equivalently
\begin{equation}
\frac{c^2}{\ii\alpha}\, \curl(\curl\boldsymbol{U}_3- \boldsymbol{F}_4)+\mathbb{B}_{\alpha}\boldsymbol{U}_3
= \boldsymbol{F}_3  -\frac{1}{\varepsilon_0}\sum_s (\ii\alpha\mathbb{I} +\mathbb M_{s})^{-1}\boldsymbol{F}_s
\label{inter2}
\end{equation}
in the sense of distributions.
Defining $\boldsymbol{U}_1,\ \boldsymbol{U}_2,\ \boldsymbol{U}_4$ respectively by \eqref{441tild}, \eqref{442tild}, and~\eqref{443tild}, these fields clearly belong to~$\mathbf{L}^2(\Omega)$. Combining \eqref{inter2} and~\eqref{443tild} with the  definition of~$\mathbb{B}_{\alpha}$, one sees that $\boldsymbol{U}_4\in\mathbf{H}(\curl;\Omega)$. Thus, the quadruple $(\boldsymbol{U}_1,\boldsymbol{U}_2,\boldsymbol{U}_3,\boldsymbol{U}_4)$ verifies the system~\eqref{11tild}--\eqref{44tild}. For this quadruple to belong to $D(\xA_2)$, we have to check that the Silver--M\"uller condition holds. To this end, we use the Green formula~\eqref{green rot 0 Gamma p} in~\eqref{555}, and find that 
\begin{equation}
c \left(\boldsymbol{U}_{3\top}\mid \boldsymbol{v}_\top\right)_{\Gamma_{A}}- c^2 {}_{\gamma_{A}}\langle \boldsymbol{U}_4 \times \boldsymbol{n}, \boldsymbol{v}_\top \rangle_{\pi^0_{A}}=0, \quad \forall \boldsymbol{v} \in \mathcal{V}.
\label{u44}
\end{equation}
Following the same argument as in the proof of Proposition~\Rref{maximonot 2}, we deduce that  Equation~\eqref{u44} implies 
that $\boldsymbol{U}_{3\top}-c\, \boldsymbol{U}_4 \times \boldsymbol{n}=0$ in $\widetilde{\mathbf{H}}^{-\frac1{2}}(\Gamma_A)$ which is equivalent to $\boldsymbol{U}_3\times\boldsymbol{n}+c\,\boldsymbol{U}_{4\top}=0$ in 
$\widetilde{\mathbf{H}}^{-\frac1{2}}(\Gamma_A)$ and thus in $\mathbf{L}^2(\Gamma_A)$ because $\boldsymbol{U}_{3\top}$ belongs to $\mathbf{L}^2(\Gamma_A)$. The proof is complete.
\end{proof}

\medbreak

Let us introduce yet another closed subspace of~$\mathbf{X}$:
\begin{equation} 
\mathbf{\widetilde{X}}_2 := \mathbf{L}{^2}(\Omega)\times \mathbf{L}{^2}(\Omega) \times \mathbf{L}{^2}(\Omega) \times\mathbf{H}_{0,\Gamma_P;\text{flux},\Gamma_A,\Sigma}({\divg}0;\Omega).
\label{eq-def-Xtild2}
\end{equation}
It is a Hilbert space when endowed with the inner product of~$\mathbf{X}$.
\begin{prpstn}
The range $\xim(\xA_{2})$ of~$\xA_{2}$ is included in~$\mathbf{\widetilde{X}}_{2}$.
\label{range-Atild2}
\end{prpstn}
\begin{proof}
Let $\boldsymbol{U}=(\boldsymbol{U}_1,\boldsymbol{U}_2,\boldsymbol{U}_3,\boldsymbol{U}_4)^\top$ be an element of $D(\xA_2)$. Then, by the definition \eqref{A2 definition} of~$\xA_2$,  $\xA_2\boldsymbol{U}$ belongs to $\mathbf{\widetilde{X}}_2$ if, and only if, $ \curl\boldsymbol{U}_3$ belongs to $\mathbf{H}_{0,\Gamma_P;\text{flux},\Gamma_A,\Sigma}({\divg}0;\Omega)$. Recall that $\boldsymbol{U}_3$ belongs to 
$\mathbf{H}_{0,\Gamma_P}(\curl;\Omega)$, therefore one can conclude by~\cite[Lemma~7.7]{FG97}. 
\end{proof}

The results of the spectral analysis of the operator $\xA_{2}$  lead us to introduce  the unbounded operator 
$(D({\widetilde{\xA}}_{2}), {\widetilde{\xA}}_{2})$ on 
${\mathbf{\widetilde{X}}}_{2}$ defined by
\begin{equation}
D({\widetilde{\xA}}_{2})=D(\xA_{2})\cap{\mathbf{\widetilde{X}}_{2}}
\quad\text{and}\quad
{\widetilde{\xA}}_{2}\boldsymbol{U}=\xA_{2}\boldsymbol{U},\quad \forall \boldsymbol{U}\in D({\widetilde{\xA}}_{2}).
\label{eq-def-Atild2}
\end{equation}
\begin{prpstn}
For all $\alpha\in\xR$, the operator $\ii\alpha\mathbb{I} + {\widetilde{\xA}}_{2}$ is surjective. For $\alpha\in\xR \setminus \{0\}$, it is injective.
\label{spec-Atild2}
\end{prpstn}
\begin{proof}
Similar to Proposition~\Rref{spec-Atild1}, using Propositions \Rref{inject sm stab}, \Rref{surj I+A2}, \Rref{range-Atild2}, and the orthogonal decomposition~\eqref{orth decom Hrot GammaA}.
\end{proof}

Observe that $\widetilde{\mathbf{X}}_{2}$ is an invariant space for Problem~\eqref{evol prob sm homo}, see Lemma~\Rref{lemma project P2B}. We define  $\check{\mathit{T}_{2}}:={\mathit{T}_{2}}_{|D(\xA_{2})\cap\widetilde{\mathbf{X}}_{2}}$.
\begin{thrm}
The semigroup of contractions $(\check{\mathit{T}_{2}}(t))_{t\geq0}$ with generator 
$-{\widetilde{\xA}}_{2}$  is strongly stable on the energy space  $\widetilde{\mathbf{X}}_{2}$ in the sense that
\begin{equation*}
\lim\limits_{t \to +\infty}\| \check{\mathit{T}_{2}}(t)\widetilde{\boldsymbol{U}}_0\|_{\mathbf{X}}=0,\quad \forall \widetilde{\boldsymbol{U}}_0 \in \widetilde{\mathbf{X}}_2.
\end{equation*}
\label{strong stab X2til}
\end{thrm}
\begin{proof}
It is sufficient to repeat the proof of Theorem \Rref{stabi forte cp X1 til}.
\end{proof}
We denote by $P_2$ the orthogonal  projection in $\mathbf{L}^2(\Omega)$ onto $\mathbf{Z}(\Omega;\Gamma_A)$. 
\begin{lmm}
Let $\boldsymbol{U}_0\in\mathbf{X}_2$ and $\boldsymbol{U}$ is the solution of problem \eqref{evol prob sm homo}. It holds  that 
\begin{equation*}
P_2(\boldsymbol{B}(t))=P_2(\boldsymbol{B}_0),\quad \forall\, t>0.
\end{equation*}
\label{lemma project P2B}
\end{lmm}
Combining this result with Theorem \Rref{strong stab X2til}, we conclude
\begin{crllr}
It holds that 
\begin{equation*}
\lim\limits_{t \to +\infty} \| {\mathit{T}_{2}(t)\boldsymbol{U}_0}-  (0,0,0,P_2\boldsymbol{B}_0)^\top \|_{\mathbf{X}_2}=0,\quad\forall \boldsymbol{U}_{0} \in \mathbf{X}_2. 
\end{equation*}
\end{crllr}

\section{Stronger stability}
\label{sec-expoly-stab}
We now establish explicit decay rates (polynomial or exponential) for the energy. Our results are based on theorems relating the decay of the resolvent of an operator with respect to frequency and the the decay of the generated semigroup with respect to time.

Namely, exponential decay will be derived from the following Theorem~\cite{P83,H85}:
\begin{thrm}[Pr\"uss / Huang]
A $\xCzero$-semigroup $(\mathit{T}(t))_{t\geq0}$ of contractions on a Hilbert space 
$\mathit{X}$ generated by $\mathcal{L}$ is exponentially stable, \ie, it satisfies 
\begin{equation*}
\forall t\geq 0,\quad \forall u_0\in \mathit{X},\quad \| \mathit{T}(t) u_0\| \leq C\, \ee^{-\gamma t}\| u_0\|_{\mathit{X}}, 
\end{equation*}
for some positive constants $C$ and $\gamma$ if, and only if,
\begin{equation}
\ii\xR=\lbrace \ii\beta : \beta \in \xR\rbrace\subset \rho(\mathcal{L}),
\label{imaginer axe} 
\end{equation}
the resolvent set of the operator $\mathcal{L}$, and
\begin{equation}
\sup_{\beta\in\xR} \III (\ii\beta\, \mathbb{I}-\mathcal{L})^{-1} \III < +\infty.
\label{born'e resolvent}
\end{equation} 
\label{stab exp theor}
\end{thrm}
On the other hand, polynomial decay will follow from this other one~\cite[Theorem~2.4]{BoTo10}:
\begin{thrm}
\label{stab poly theor}
A $\xCzero$-semigroup $(\mathit{T}(t))_{t\geq0}$ of contractions on a Hilbert space 
$\mathit{X}$ generated by $\mathcal{L}$ satisfies
\begin{equation*}
\forall t>1,\quad \forall u_0\in \mathit{D}(\mathcal{L}),\quad \| \mathit{T}(t) u_0\| \leq C\, t^{-\frac1{\ell}}\| u_0\|_{\mathit{D}(\mathcal{L})}, 
\end{equation*}
as well as
\begin{equation*}
\forall t>1,\quad \forall u_0\in \mathit{D}(\mathcal{L}^{\ell}),\quad \| \mathit{T}(t) u_0\| \leq C\, t^{-1}\| u_0\|_{\mathit{D}(\mathcal{L}^{\ell})}, 
\end{equation*}
for some constant $C>0$ and for some positive integer $\ell$ if~\eqref{imaginer axe} holds and if 
\begin{equation}
\limsup_{|\beta|\rightarrow \infty}\frac1{\beta^{\ell}}\III (\ii\beta\, \mathbb{I}-\mathcal{L})^{-1} \III < +\infty.
\label{borne resolvent ply}
\end{equation} 
\end{thrm}

\subsection{Polynomial stability, perfectly conducting case}
\begin{prpstn}
Let $\rho(-\widetilde{{\xA}}_1)$ denote the resolvent set of $-\widetilde{{\xA}}_1$. Then,  $0\in\rho(-\widetilde{{\xA}}_1)$.
\label{0 resolvent cp}
\end{prpstn}
\begin{proof}
By Proposition~\Rref{spec-Atild1}, we know that $0\in\ii\xR$ is not an eigenvalue, so in order to prove that $0\in\rho(-\widetilde{{\xA}}_1)$, we need to check that $\widetilde{{\xA}}_1$ is surjective and has a bounded inverse. Both properties follow from the fact that the resolvent of~$-\widetilde{{\xA}}_1$ is uniformly bounded in the neighborhood of~$0$, which we shall now prove by a contradiction argument. 

\medbreak

Suppose the above condition is false, then there exists a sequence $(\beta_n)_{n\in\xN}$ on~$\xR\setminus\lbrace 0 \rbrace$ with $\beta_n\rightarrow 0$ as $n\rightarrow +\infty$, and a sequence of vector fields $(\boldsymbol{U}^n)_{n\in\xN}= \left( (\boldsymbol{U}_1^n,\boldsymbol{U}_2^n,\boldsymbol{U}_3^n,\boldsymbol{U}_4^n)^\top \right)_n$ on~$D(\widetilde{{\xA}}_1)$, with
\begin{eqnarray}
\| \boldsymbol{U}^n\|_{\mathbf{X}}=1,\quad \forall n,
\label{unitaire}
\end{eqnarray}
such that 
\begin{eqnarray}
\| (\ii\beta_n\mathbb{I} + \xA_1)\, \boldsymbol{U}^n \|_{\mathbf{X}} \rightarrow  0 ~\text{ as } n \to +\infty,
\label{resol uni}
\end{eqnarray}
which is equivalent to
\begin{eqnarray}
\ii\beta_n\,\boldsymbol{U}_1^n + \mathbb{M}_1\boldsymbol{U}_1^n - \varepsilon_0\omega_{p1}^2\, \boldsymbol{U}_3^n
&\rightarrow & 0 \, \text{ in } \, \mathbf{L}^2(\Omega),
\label{1exp}\\
\ii\beta_n\,\boldsymbol{U}_2^n + \mathbb{M}_2\boldsymbol{U}_2^n - \varepsilon_0\omega_{p2}^2\, \boldsymbol{U}_3^n
&\rightarrow &0 \, \text{ in }\, \mathbf{L}^2(\Omega),
\label{2exp}\\
\ii\beta_n\varepsilon_0\,\boldsymbol{U}_3^n + \boldsymbol{U}_1^n + \boldsymbol{U}_2^n - \varepsilon_0 c^2\curl\boldsymbol{U}_4^n
&\rightarrow &0 \, \text{ in }\,  \mathbf{L}^2(\Omega),
\label{3exp}\\
\ii\varepsilon_0 c^2\beta_n\,\boldsymbol{U}_4^n + \varepsilon_0 c^2 \curl\boldsymbol{U}_3^n
&\rightarrow &0 \, \text{ in }\,  \mathbf{L}^2(\Omega).
\label{4exp}
\end{eqnarray}

\medbreak

Since, by \eqref{resul monot A1} and~\eqref{unitaire},
\begin{equation}
\sum_s \left( \frac{\nu_s\, \boldsymbol{U}_s^n}{{\varepsilon_0}\omega_{ps}} \Biggm| \frac{\boldsymbol{U}_s^n}{\omega_{ps}} \right) = \Re((\ii\beta_n\mathbb{I} +\xA_1)\,\boldsymbol{U}^n,\boldsymbol{U}^n)_{\mathbf{X}}\leq \| (\ii\beta_n\mathbb{I} +\xA_1)\,\boldsymbol{U}^n \|_{\mathbf{X}},
\end{equation}
we obtain from~\eqref{resol uni} that
\begin{equation}
\left( \frac{\nu_s\, \boldsymbol{U}_s^n}{{\varepsilon_0}\omega_{ps}} \Biggm| \frac{\boldsymbol{U}_s^n}{\omega_{ps}} \right) \rightarrow 0, \text{ as } n\to +\infty,\quad s=1,\ 2
\end{equation}
which leads by \eqref{6injec} to
\begin{equation}
\|\boldsymbol{U}_s^n\|_{(s)} \rightarrow 0 , \text{ as } n\to +\infty,
\quad s=1,\ 2.
\label{snnn}
\end{equation}
The matrix $\mathbb{M}_s$ is bounded on~$\Omega$, so it follows from to \eqref{snnn} and~\eqref{1exp} that
\begin{equation}
\boldsymbol{U}_3^n \rightarrow  0 ~\text{ in } \,
\mathbf{L}^2(\Omega),  \text{ as } n\to +\infty
\label{converg uc 0 cp}
\end{equation}
and then we deduce from~\eqref{3exp} that
\begin{equation}
\curl \boldsymbol{U}_4^n \rightarrow  0 ~\text{ in }\, 
\mathbf{L}^2(\Omega), \text{ as } n\to +\infty.
\label{converg ud 0 cp}
\end{equation}
This shows that $(\curl\boldsymbol{U}_4^n)_n$ is bounded in~$\mathbf{L}^2(\Omega)$. Taking account of~\eqref{unitaire}, the sequence $(\boldsymbol{U}_4^n)_n$ is bounded in~$\mathbf{H}(\curl;\Omega)$, and more specifically in the closed subspace $\mathcal{J}_1(\Omega):=\mathbf{H}(\curl;\Omega)\cap\mathbf{H}^{\Sigma}_0(\divg0;\Omega)$ to which all its terms belong given the definition of~$D(\widetilde{{\xA}}_1)$, see \eqref{eq-def-Atild1} and~\eqref{eq-def-Xtild1}. But $\mathcal{J}_1(\Omega)$~is also a closed subspace of $\mathbf{H}(\curl;\Omega)\cap\mathbf{H}_0(\divg;\Omega)$, which is compactly embedded into~$\mathbf{L}^2(\Omega)$  \cite[Theorem~3.5.4]{ACL+17}; thus, we can extract a subsequence still denoted by $(\boldsymbol{U}_4^n)_n$ which converges strongly in~$\mathbf{L}^2(\Omega)$ to some $\boldsymbol{U}_4\in \mathcal{J}_1(\Omega)$. As a consequence, $(\curl\boldsymbol{U}_4^n)_n$ converges in the sense of distributions to~$\curl\boldsymbol{U}_4$; this combined with~\eqref{converg ud 0 cp} implies that
\begin{equation*}
\curl \boldsymbol{U}_4^n \rightarrow  \curl\boldsymbol{U}_4=0 ~\text{ in }\, \mathbf{L}^2(\Omega), \text{ as } n\to +\infty.
\end{equation*}
So, $\boldsymbol{U}_4\in\mathbf{H}(\curl0;\Omega)$. Together with $\boldsymbol{U}_4\in \mathcal{J}_1(\Omega)$, this means  that $\boldsymbol{U}_4$ belongs to $\mathbf{Z}_{T}(\Omega)=\mathbf{H}(\curl\,0;\Omega)\cap\mathbf{H}_0(\divg0;\Omega)$ and its orthogonal $\mathbf{H}^{\Sigma}_0(\divg0;\Omega)$ (\cf~\eqref{orth decom ZT}), whence $\boldsymbol{U}_4=0$.
\smallbreak

On the other hand, \eqref{unitaire}, \eqref{snnn} and~\eqref{converg uc 0 cp} imply that  
\begin{equation*}
1=\lim_{n \to +\infty} \| \boldsymbol{U}^n \|^2_{\mathbf{X}}=\lim_{n \to +\infty} {\varepsilon_0\,} c^2\| \boldsymbol{U}_4^n \|^2 ={\varepsilon_0\,} c^2\| \boldsymbol{U}_4 \|^2,
\end{equation*}
in particular, $\boldsymbol{U}_4\neq0$, and the above conclusion is contradicted. Hence the resolvent is uniformly bounded in the neighborhood of~$0$:
\begin{equation}
\exists C>0,\quad \forall \beta \in [-1,1] \setminus \{0\},\quad \III ( \ii\beta\, \mathbb{I} + \widetilde{{\xA}}_1 )^{-1} \III \leq C.
\end{equation}

\medbreak

The surjectivity of~$-\widetilde{{\xA}}_1$ and the boundedness of its inverse then follow from a standard argument. Pick any $\boldsymbol{F} \in \widetilde{\mathbf{X}}_1$. By Proposition~\Rref{spec-Atild1}, for any  $k\in\xN \setminus \{0\}$ there exists a unique $\boldsymbol{U}^k \in D(\widetilde{{\xA}}_1)$ such that $(\ii k^{-1} + \xA)\, \boldsymbol{U}^k = -\boldsymbol{F}$, and $\| \boldsymbol{U}^k \|_{\mathbf{X}} \leq C\, \| \boldsymbol{F} \|_{\mathbf{X}}$. 

\smallbreak

Being bounded, the sequence $(\boldsymbol{U}^k)_k$ admits a  subsequence (still denoted~$(\boldsymbol{U}^k)_k$) that converges weakly toward $\boldsymbol{U} \in \widetilde{\mathbf{X}}_1$, as the latter is a closed subspace of~$\mathbf{X}$, which still satisfies $\| \boldsymbol{U} \|_{\mathbf{X}} \leq C\, \| \boldsymbol{F} \|_{\mathbf{X}}$. 
Moreover, $-\xA\,\boldsymbol{U}^k \rightharpoonup -\xA\,\boldsymbol{U}$ in the sense of distributions. But, on the other hand 
\begin{equation*}
-\xA\,\boldsymbol{U}^k = \boldsymbol{F} + \ii k^{-1}\, \boldsymbol{U}^k \to \boldsymbol{F} \quad \text{in } \mathbf{X}.
\end{equation*}
Hence, $-\xA\,\boldsymbol{U} = \boldsymbol{F}$, \ie, $\boldsymbol{U} \in D(\widetilde{{\xA}}_1)$. As $\boldsymbol{F}$ is arbitrary, this proves that $-\widetilde{{\xA}}_1$ is surjective, hence bijective, between $D(\widetilde{{\xA}}_1)$ and~$\widetilde{\mathbf{X}}_1$, and its inverse is bounded: $\III (-\widetilde{{\xA}}_1)^{-1} \III \le C$ and $\III (-\widetilde{{\xA}}_1)^{-1} \III_{\widetilde{\mathbf{X}}_1 \to D(\widetilde{{\xA}}_1)} \le C+1$.
\end{proof}
\begin{rmrk}
The surjectivity of~$-\widetilde{{\xA}}_1$ could have been easily obtained by a direct argument. We gave this proof because it provides a pattern for subsequent ones.
\end{rmrk}

\begin{prpstn}
The resolvent of the operator $-\widetilde{{\xA}}_1$ satisfies the condition \eqref{borne resolvent ply} with $\ell=2$.
\label{betta resolvant borne cp}
\end{prpstn}
\begin{proof}
We again use a contradiction argument, \ie, we assume that~\eqref{borne resolvent ply} is false for some $\ell\in\xN$, which will be specified later. Then, there exists a sequence $(\beta_n)_{n\in\xN}$ on~$\xR$ with $|\beta_n|\rightarrow +\infty$ as $n\rightarrow +\infty$, and a sequence $({\boldsymbol{U}}^n)_{n\in\xN}= \left( ({\boldsymbol{U}}_1^n,{\boldsymbol{U}}_2^n,{\boldsymbol{U}}_3^n,{\boldsymbol{U}}_4^n)^\top \right)_n$ of elements of~$D(\widetilde{{\xA}}_1)$, such that
\begin{eqnarray}
\| {\boldsymbol{U}}^n\|_{\mathbf{X}}=1,\quad \forall n,
\label{unitaire poly}
\end{eqnarray}
and
\begin{eqnarray}
\beta_n^\ell\, \| (\ii\beta_n\mathbb{I} + \xA_1)\, {\boldsymbol{U}}^n \|_{\mathbf{X}} \rightarrow  0 ~\text{ as } n \to +\infty,
\label{resol uni poly}
\end{eqnarray}
which is equivalent to
\begin{eqnarray}
\beta^{\ell}_n\, (\ii\beta_n\,{\boldsymbol{U}}_1^n + \mathbb{M}_1{\boldsymbol{U}}_1^n - \varepsilon_0\omega_{p1}^2\, {\boldsymbol{U}}_3^n)
&\rightarrow & 0 \, \text{ in } \, \mathbf{L}^2(\Omega),
\label{1poly}\\
\beta^{\ell}_n\, (\ii\beta_n\,{\boldsymbol{U}}_2^n + \mathbb{M}_2{\boldsymbol{U}}_2^n - \varepsilon_0\omega_{p2}^2\, {\boldsymbol{U}}_3^n)
&\rightarrow &0 \, \text{ in }\, \mathbf{L}^2(\Omega),
\label{2poly}\\
\beta^{\ell}_n\, (\ii\beta_n\varepsilon_0\,{\boldsymbol{U}}_3^n + {\boldsymbol{U}}_1^n + {\boldsymbol{U}}_2^n - \varepsilon_0 c^2\curl{\boldsymbol{U}}_4^n)
&\rightarrow &0 \, \text{ in }\,  \mathbf{L}^2(\Omega),
\label{3poly}\\
\beta^{\ell}_n\, (\ii\varepsilon_0 c^2\beta_n\,{\boldsymbol{U}}_4^n + \varepsilon_0 c^2 \curl{\boldsymbol{U}}_3^n)
&\rightarrow &0 \, \text{ in }\,  \mathbf{L}^2(\Omega).
\label{4poly}
\end{eqnarray}

\medbreak

As, according to \eqref{resul monot A1} and~\eqref{unitaire poly},
\begin{equation}
|\beta_n^{\ell}| \sum_s \left( \frac{\nu_s\, {\boldsymbol{U}}_s^n}{{\varepsilon_0}\omega_{ps}} \Biggm| \frac{{\boldsymbol{U}}_s^n}{\omega_{ps}} \right) = |\beta_n^{\ell}|\, \Re((\ii\beta_n\mathbb{I} +\xA_1)\,{\boldsymbol{U}}^n,{\boldsymbol{U}}^n)_{\mathbf{X}}\leq |\beta_n^{\ell}|\; \| (\ii\beta_n\mathbb{I} +\xA_1)\,{\boldsymbol{U}}^n \|_{\mathbf{X}},
\end{equation}
we infer from~\eqref{resol uni poly} that
\begin{equation*}
|\beta_n^{\ell}|\, \left( \frac{\nu_s\, {\boldsymbol{U}}_s^n}{{\varepsilon_0}\omega_{ps}} \Biggm| \frac{{\boldsymbol{U}}_s^n}{\omega_{ps}} \right) \rightarrow 0\quad \text{as } n\to +\infty,\quad s=1,\ 2,
\end{equation*}
which implies, by Hypotheses~\ref{hyp-1} and~\ref{hyp-2},
\begin{equation}
\beta^{\frac{\ell}{2}}_n\, {\boldsymbol{U}}_s^n \rightarrow 0\quad \text{in } \mathbf{L}^2(\Omega), \text{ as } n\to +\infty,\quad s=1,\ 2.
\label{Us poly}
\end{equation}
(For the sake of simplicity, one may assume $\ell$ even, so that $\beta^{\frac{\ell}{2}}_n$ is unambiguously defined; otherwise, one may choose a principal determination for the square root of a negative real number. This is of little importance, as all the limits we consider are zero.)

\smallbreak

Multiplying \eqref{1poly} by $\beta^{-\frac{\ell}{2}-1}_n$, we get:
\begin{equation*}
\ii\beta^{\frac{\ell}{2}}_n\,{\boldsymbol{U}}_1^n + \beta_n^{\frac{\ell}{2}-1}\mathbb{M}_1{\boldsymbol{U}}_1^n - \varepsilon_0\omega_{p1}^2\, \beta_n^{\frac{\ell}{2}-1}{\boldsymbol{U}}_3^n \rightarrow 0\quad \text{in } \mathbf{L}^2(\Omega),
\end{equation*}
together with~\eqref{Us poly}, this yields
\begin{equation}
\beta^{\frac{\ell}{2}-1}_n\, {\boldsymbol{U}}_3^n \rightarrow 0, \text{ in } \mathbf{L}^2(\Omega)\quad \text{as } n\to +\infty.
\label{U3 poly}
\end{equation}
Similarly, multiplying \eqref{3poly} by $\beta^{-\frac{\ell}{2}-2}_n$ yields:
\begin{equation*}
\ii\beta^{\frac{\ell}{2}-1}_n\, \varepsilon_0\, {\boldsymbol{U}}_3^n + \beta^{\frac{\ell}{2}-2}_n\, {\boldsymbol{U}}_1^n + \beta^{\frac{\ell}{2}-2}_n\, {\boldsymbol{U}}_2^n - \varepsilon_0 c^2\beta^{\frac{\ell}{2}-2}_n\, \curl{\boldsymbol{U}}_4^n \rightarrow 0\quad \text{in } \mathbf{L}^2(\Omega),
\end{equation*}
taking \eqref{Us poly} and~\eqref{U3 poly} into account, we arrive at:
\begin{equation}
\beta^{\frac{\ell}{2}-2}_n\, \curl{\boldsymbol{U}}_4^n \rightarrow 0\quad \text{in } \mathbf{L}^2(\Omega),\text{ as } n\to +\infty.
\label{rotU4 poly}
\end{equation}
Now, let us multiply \eqref{4poly} by $\beta^{-\frac{\ell}{2}-2}_n$:
\begin{equation*}
\ii\varepsilon_0 c^2\beta^{\frac{\ell}{2}-1}_n\,{\boldsymbol{U}}_4^n + \varepsilon_0 c^2 \beta^{\frac{\ell}{2}-2}_n\, \curl{\boldsymbol{U}}_3^n \rightarrow 0 \quad \text{in } \mathbf{L}^2(\Omega),
\end{equation*}
then we take the inner product (on the right side) of this equation by~${\boldsymbol{U}}_4^n$ and find
\begin{equation}
\ii\varepsilon_0 c^2\beta^{\frac{\ell}{2}-1}_n\, \|{\boldsymbol{U}}_4^n\|^2 + \varepsilon_0 c^2 \beta^{\frac{\ell}{2}-2}_n\, \left(\curl{\boldsymbol{U}}_3^n  \mid {\boldsymbol{U}}_4^n \right) \rightarrow 0,\quad \text{as } n\to +\infty.
\label{sca U4}
\end{equation}
On the other hand, using the Green formula~\eqref{green rot rot} and the condition ${\boldsymbol{U}}_3^n\times\boldsymbol{n}=0$ on~$\Gamma$, we obtain: 
\begin{equation}
\beta^{\frac{\ell}{2}-2}_n\, \left(\curl{\boldsymbol{U}}_3^n  \mid {\boldsymbol{U}}_4^n \right) =  \left({\boldsymbol{U}}_3^n  \Bigm| \beta^{\frac{\ell}{2}-2}_n\curl{\boldsymbol{U}}_4^n \right)
\leq  \|{\boldsymbol{U}}_3^n\| \, \|\beta^{\frac{\ell}{2}-2}_n\curl{\boldsymbol{U}}_4^n\|.
\label{integ U4 rotU3}
\end{equation}
Assuming $\ell\geq 2$, we deduce from~\eqref{U3 poly} that:
\begin{equation}
\|{\boldsymbol{U}}_3^n\|\rightarrow 0,\quad \text{as } n\to +\infty.
\label{U3 converge sup2}
\end{equation}
Thus, from Equations \eqref{integ U4 rotU3}, \eqref{U3 converge sup2} and~\eqref{rotU4 poly}, we deduce
\begin{equation}
\beta^{\frac{\ell}{2}-2}_n\, \left(\curl{\boldsymbol{U}}_3^n  \mid {\boldsymbol{U}}_4^n \right) \rightarrow 0,\quad \text{as } n\to +\infty.
\label{integ U4 rotU3 conv}
\end{equation}
Together with \eqref{sca U4}, the latter property implies
\begin{equation}
\beta^{\frac{\ell}{2}-1}_n\, \|{\boldsymbol{U}}_4^n\|^2 \rightarrow 0,\quad \text{as } n\to +\infty 
\label{U4 poly}
\end{equation}
if $\ell\geq2$. In this case, the exponent $\frac{\ell}{2}-1\geq 0$, and there holds \afortiori: 
\begin{equation}
\|{\boldsymbol{U}}_4^n\| \rightarrow 0,\quad \text{as } n\to +\infty. 
\label{U4 converge sup2}
\end{equation}
As \eqref{Us poly} implies, for all $\ell>0$, that
\begin{equation}
\|{\boldsymbol{U}}_s^n\| \rightarrow 0,\quad \text{as } n\to +\infty,\quad s=1,\ 2,
\label{Us converge poly}
\end{equation}
taking \eqref{U3 converge sup2} and~\eqref{U4 converge sup2} into account and using the equivalence of norms, we obtain $\| {\boldsymbol{U}}^n\|_{\mathbf{X}}\rightarrow 0$ as $n\to +\infty$, which contradicts~\eqref{unitaire poly}.
\end{proof}

\medbreak

Hence, according to Theorem \Rref{stab poly theor} we conclude\dots
\begin{thrm}
\label{thm-poly-decay-cp}
The semigroup of contractions $(\check{\mathit{T}_1}(t))_{t\geq0}$, with generator  $-{\widetilde{\xA}}_{1}$, is polynomially stable on~${\mathbf{\widetilde{X}}_{1}}$, \ie, there exist a  constant $C >0$ such that 
\begin{equation}
\forall t>1,\quad
\| \check{\mathit{T}_1}(t) \widetilde{{\boldsymbol{U}}}_0 \|_{\mathbf{\widetilde{X}}_{1}} \leq C\, t^{-\frac1{2}} \| \widetilde{{\boldsymbol{U}}}_0  \|_{D({\widetilde{\xA}}_{1})},\quad \forall\widetilde{{\boldsymbol{U}}}_0 \in D({\widetilde{\xA}}_{1}).
\label{poly til X1}
\end{equation}
Furthermore, under the assumptions of Theorem~\Rref{divergence B} on~$\boldsymbol{B}_0$, there exists a constant $M>0$ such that the solution to Problem~\eqref{evol prob cp} satisfies 
\begin{equation}
\forall t>1,\quad
\| \mathit{T}_1(t){{\boldsymbol{U}}_0}- \sum_{1\leq j \leq J} \xi_j \, (0,0,0,\widetilde{\grad \dot{q}_j})^\top \|_{\mathbf{X}_1} \leq M\, t^{-\frac1{2}} \| {\boldsymbol{U}}_{0} \|_{D(\xA_{1})\cap\mathbf{X}_1},
\label{poly X1}
\end{equation}
for all ${\boldsymbol{U}}_0 \in D(\xA_{1})\cap\mathbf{X}_1$, where $\xi_j=\langle \boldsymbol{B}_0 \cdot \boldsymbol{n}, 1 \rangle_{\Sigma_j},\quad \text{for } j=1,\ \ldots,\  J$.
\end{thrm}
\begin{proof}
Equation~\eqref{poly X1} is a consequence of \eqref{poly til X1} and Corollary~\Rref{strong stab X1}.
\end{proof}

\subsection{Polynomial stability, homogeneous Silver--M\"uller case}
\label{poly-sm}
\begin{prpstn}
Let $\rho(-\widetilde{{\xA}}_2)$ denote the resolvent set of $-\widetilde{{\xA}}_2$. Then,  $0\in\rho(-\widetilde{{\xA}}_2)$.
\label{0 resolvent sm}
\end{prpstn}
\begin{proof}
We follow the same argument in the proof of Proposition~\Rref{0 resolvent cp}, and we prove that the resolvent of~$-\widetilde{{\xA}}_2$ is uniformly bounded in the neighborhood of~$0$ by using a contradiction argument. Suppose it is not the case, 
then there exists a sequence $(\beta_n)_{n\in\xN}$ on~$\xR\setminus\lbrace 0 \rbrace$ with $\beta_n\rightarrow 0$ as $n\rightarrow +\infty$, and a sequence of vectors fields $(\boldsymbol{U}^n)_{n\in\xN}= \left( (\boldsymbol{U}_1^n,\boldsymbol{U}_2^n,\boldsymbol{U}_3^n,\boldsymbol{U}_4^n)^\top \right)_n$ on~$D(\widetilde{{\xA}}_2)$, with
\begin{equation}
\| \boldsymbol{U}^n\|_{\mathbf{X}}=1,\quad \forall n,
\label{A-unitaire}
\end{equation}
such that 
\begin{equation}
\| (\ii\beta_n\mathbb{I} + \xA_2)\, \boldsymbol{U}^n \|_{\mathbf{X}} \rightarrow  0 ~\text{ as } n \to +\infty,
\label{A-resol uni}
\end{equation}
this again implies the system~\eqref{1exp}--\eqref{4exp}, with different boundary conditions.

\medbreak

By the monotonicity of $\xA_2$ (Equation~\eqref{monotonie A2 sm homo}) and~\eqref{A-unitaire}:
\begin{eqnarray*}
\sum_s \left( \frac{\nu_s\, \boldsymbol{U}_s^n}{{\varepsilon_0}\omega_{ps}} \Biggm| \frac{\boldsymbol{U}_s^n}{\omega_{ps}} \right)  + \varepsilon_0 c^3\, \| \boldsymbol{U}_{4\top}^n \|_{\mathbf{L}^2(\Gamma_A)}^2 &=& \Re\left((\ii\beta_n\mathbb{I} +\xA_2)\,\boldsymbol{U}^n,\boldsymbol{U}^n\right)_{\mathbf{X}}
\nonumber\\
&\leq & \| (\ii\beta_n\mathbb{I} +\xA_2)\,\boldsymbol{U}^n \|_{\mathbf{X}},
\end{eqnarray*}
we obtain from (\Rref{A-resol uni})  
\begin{equation}
\left( \frac{\nu_s\, \boldsymbol{U}_s^n}{{\varepsilon_0}\omega_{ps}} \Biggm| \frac{\boldsymbol{U}_s^n}{\omega_{ps}} \right)  \rightarrow 0, \text{ as } n\to +\infty \quad \forall s=1,\ 2
\label{A-u1,2}
\end{equation} 
and 
\begin{equation} 
\| \boldsymbol{U}_{4\top}^n \|_{\mathbf{L}^2(\Gamma_A)}^2 \rightarrow 0, \text{ as } n\to +\infty. \qquad 
\label{A-bore}
\end{equation} 
As already said, the condition $\boldsymbol{U}_{4\top}^n\in\mathbf{L}^2(\Gamma_A)$ follows from the Silver--M\"uller boundary condition, in the absence of pathological vertices.

\medbreak

Reasoning as in Proposition~\Rref{0 resolvent cp}, we deduce that
\begin{eqnarray}
&&\|\boldsymbol{U}_s^n\|_{(s)} \rightarrow 0 , \text{ as }  n\to +\infty,
\quad s=1,\ 2,\ 3\,;
\label{0A-sn & converg uc}
\\
&&\curl \boldsymbol{U}_4^n \rightarrow  0 \text{ in } 
\mathbf{L}^2(\Omega), \text{ as } n\to +\infty.
\label{converg ud A}
\end{eqnarray}
Hence, $(\curl\boldsymbol{U}_4^n)_n$ is bounded in~$\mathbf{L}^2(\Omega)$. Taking account of \eqref{A-unitaire} and~\eqref{A-bore}, the sequence $(\boldsymbol{U}_4^n)_n$ is bounded in~$\mathcal{W} = \{ \boldsymbol{w} \in \mathbf{H}(\curl;\Omega) : \boldsymbol{w}\times\boldsymbol{n}_{|_{\Gamma_A}} \in \mathbf{L}^2_t(\Gamma_A) \}$, and more specifically in the closed subspace $\mathcal{J}_2(\Omega) := \mathcal{W} \cap {\mathbf{H}_{0,{\Gamma}_P;\text{flux},{\Gamma}_A,\Sigma}({\divg}0;\Omega)}$ to which all its terms belong given the definition of~$D(\widetilde{{\xA}}_2)$, see \eqref{eq-def-Atild2} and~\eqref{eq-def-Xtild2}. Yet, $\mathcal{J}_2(\Omega)$~also appears as a closed subspace of
\begin{equation*}
\left\{ \boldsymbol{w} \in \mathbf{H}(\curl;\Omega)\cap\mathbf{H}(\divg;\Omega) : \boldsymbol{w}\cdot\boldsymbol{n}_{|_{\Gamma_P}} \in \xLtwo(\Gamma_P) \text{ and } \boldsymbol{w}\times\boldsymbol{n}_{|_{\Gamma_A}} \in \mathbf{L}^2_t(\Gamma_A) \right\},
\end{equation*}
which is compactly embedded into~$\mathbf{L}^2(\Omega)$ \cite[Proposition~7.3]{FG97}.
Therefore, we can extract a subsequence, still denoted~$(\boldsymbol{U}_4^n)_n$, which converges strongly in $\mathbf{L}^2(\Omega)$, and weakly in~$\mathcal{J}_2(\Omega)$, to some $\boldsymbol{U}_4\in \mathcal{J}_2(\Omega)$. Combining the weak convergence in~$\mathcal{J}_2(\Omega)$  with \eqref{converg ud A} and~\eqref{A-bore}, we find
\begin{equation*}
\curl \boldsymbol{U}_4^n \rightarrow  \curl\boldsymbol{U}_4=0 \text{ in } \mathbf{L}^2(\Omega),
\quad \boldsymbol{U}_{4\top}^n \rightarrow {\boldsymbol{U}_4}_\top = 0 \text{ in } \mathbf{L}^2_t(\Gamma_A),
\quad \text{as } n\to +\infty.
\end{equation*}
So, $\boldsymbol{U}_4\in\mathbf{H}_{0,\Gamma_A}(\curl0;\Omega)$. Together with $\boldsymbol{U}_4\in \mathcal{J}_2(\Omega)$, this means  that $\boldsymbol{U}_4$ belongs both to $ \mathbf{Z}(\Omega;\Gamma_A) = \mathbf{H}_{0,\Gamma_A}(\curl\,0;\Omega)\cap\mathbf{H}_{0,\Gamma_P}(\divg0;\Omega)$ and to its orthogonal $\mathbf{H}_{0,\Gamma_P;\text{flux},\Gamma_A,\Sigma}(\divg 0;\Omega)$ (\cf~\eqref{orth decom Z GammaA}), whence $\boldsymbol{U}_4=0$.

\medbreak

On the other hand, \eqref{A-unitaire} and~\eqref{0A-sn & converg uc} imply that  
\begin{equation*}
1=\lim_{n \to +\infty} \| \boldsymbol{U}^n \|^2_{\mathbf{X}}=\lim_{n \to +\infty} {\varepsilon_0\,} c^2\| \boldsymbol{U}_4^n \|^2 ={\varepsilon_0\,} c^2\| \boldsymbol{U}_4 \|^2,
\end{equation*}
in particular, $\boldsymbol{U}_4\neq0$, and the above conclusion is contradicted. The proof is complete. 
\end{proof}

\medbreak

\begin{prpstn}
The resolvent of the operator $-\widetilde{{\xA}}_2$ satisfies the condition \eqref{borne resolvent ply} with $\ell=2$.
\label{betta resolvant borne sm}
\end{prpstn}
\begin{proof}
We follow the lines of the proof of Proposition~\ref{betta resolvant borne cp}, only insisting on the differences. Assume that \eqref{borne resolvent ply} does not hold, with $\ell\in\xN$ to be specified later, then there exists a sequence  $(\beta_n)_{n\in\xN}$ on~$\xR$ with $|\beta_n|\rightarrow \infty$ as $n\rightarrow +\infty$, and a sequence $({\boldsymbol{U}}^n)_{n\in\xN}= \left( ({\boldsymbol{U}}_1^n,{\boldsymbol{U}}_2^n,{\boldsymbol{U}}_3^n,{\boldsymbol{U}}_4^n)^\top \right)_n$ on~$D(\widetilde{{\xA}}_2)$, such that:
\begin{equation}
\| {\boldsymbol{U}}^n\|_{\mathbf{X}}=1,\quad \forall n,
\label{A-unitaire poly}
\end{equation}
and
\begin{equation}
\beta^{\ell}_n\, \| (\ii\beta_n\mathbb{I} + \xA_2)\, {\boldsymbol{U}}^n \|_{\mathbf{X}} \rightarrow  0 \quad \text{as } n \to +\infty;
\label{A-resol uni poly}
\end{equation}
again, the latter condition is equivalent to the system~\eqref{1poly}--\eqref{4poly}, with different boundary conditions.

\medbreak

\noindent Using the monotonicity of $\xA_2$ (Equation~\eqref{monotonie A2 sm homo}) and~\eqref{A-unitaire poly}:
\begin{eqnarray*}
|\beta^{\ell}_n|\, \sum_s \left( \frac{\nu_s\, {\boldsymbol{U}}_s^n}{{\varepsilon_0}\omega_{ps}} \Biggm| \frac{{\boldsymbol{U}}_s^n}{\omega_{ps}} \right)  + \varepsilon_0 c^3\,|\beta^{\ell}_n|\,  \| {\boldsymbol{U}}_{4\top}^n \|_{\mathbf{L}^2(\Gamma_A)}^2 
&=& |\beta^{\ell}_n|\, \Re\left((\ii\beta_n\mathbb{I} +\xA_2)\,{\boldsymbol{U}}^n,{\boldsymbol{U}}^n\right)_{\mathbf{X}}
\nonumber\\
&\leq & |\beta^{\ell}_n|\;  \| (\ii\beta_n\mathbb{I} +\xA_2)\,{\boldsymbol{U}}^n \|_{\mathbf{X}},
\end{eqnarray*}
we infer from~\eqref{A-resol uni poly} that:
\begin{equation}
\beta^{\frac{\ell}{2}}_n\, {\boldsymbol{U}}_s^n \rightarrow 0,\, \text{ in } \mathbf{L}^2(\Omega)\quad \text{as } n\to +\infty,\quad s=1,\ 2,
\label{Us poly sm}
\end{equation} 
and
\begin{equation} 
\beta^{\ell}_n\, \| {\boldsymbol{U}}_{4\top}^n \|_{\mathbf{L}^2(\Gamma_A)}^2 \rightarrow 0, \quad \text{as } n\to +\infty.  
\label{A-bore poly sm}
\end{equation}
Reasoning as in Proposition~\ref{betta resolvant borne cp}, we deduce:
\begin{eqnarray}
&&\beta^{\frac{\ell}{2}-1}_n\, {\boldsymbol{U}}_3^n \rightarrow 0\quad \text{in } \mathbf{L}^2(\Omega), \text{ as } n\to +\infty,
\label{U3 poly sm}
\\
&&\beta^{\frac{\ell}{2}-2}_n\, \curl{\boldsymbol{U}}_4^n \rightarrow 0,\quad \text{ in } \mathbf{L}^2(\Omega) \text{ as } n\to +\infty,
\label{rotU4 poly sm}
\end{eqnarray}
which yields~\eqref{sca U4} again. Then, using the Green formula~\eqref{green rot 0 Gamma p} and the Silver--M\"uller boundary condition, we find: 
\begin{eqnarray}
\beta^{\frac{\ell}{2}-2}_n\, \, \, \left(\curl{\boldsymbol{U}}_3^n  \mid {\boldsymbol{U}}_4^n \right) 
&=& \beta^{\frac{\ell}{2}-2}_n\,  \left({\boldsymbol{U}}_3^n  \mid \curl{\boldsymbol{U}}_4^n \right) - \beta^{\frac{\ell}{2}-2}_n{}_{{\gamma}^0_A}\langle {\boldsymbol{U}}_3^n \times \boldsymbol{n}, {\boldsymbol{U}}_{4\top}^n \rangle_{{\pi}_{A}}
\nonumber\\
&=&  \left({\boldsymbol{U}}_3^n  \Bigm| \beta^{\frac{\ell}{2}-2}_n\, \curl{\boldsymbol{U}}_4^n \right) + c\,\beta^{\frac{\ell}{2}-2}_n\,  \| {\boldsymbol{U}}_{4\top}^n \|^2_{\mathbf{L}^2(\Gamma_A)}.
\label{integ U4 rotU3 sm}
\end{eqnarray}
But, on the other hand, according to~\eqref{A-bore poly sm}
\begin{equation*}
\beta^{\frac{\ell}{2}-2}_n\,  \| {\boldsymbol{U}}_\top^n \|^2_{\mathbf{L}^2(\Gamma_A)}=\beta^{-\frac{\ell}{2}-2}_n \times \left(\beta^{\ell}_n\,  \| {\boldsymbol{U}}_{4\top}^n \|^2_{\mathbf{L}^2(\Gamma_A)}\right) \rightarrow 0\quad \text{as } n\to +\infty,
\end{equation*}
for all~$\ell>0$. Moreover, Equation~\eqref{U3 poly sm} implies that the first term on the right-hand side of~\eqref{integ U4 rotU3 sm} converges to~$0$ if~$\ell\geq 2$.  As a consequence,
\begin{equation}
\beta^{\frac{\ell}{2}-2}_n\, \left(\curl{\boldsymbol{U}}_3^n  \mid {\boldsymbol{U}}_4^n \right)\rightarrow 0,\quad \text{as } n\to +\infty
\label{integ U4 rotU3 conv sm}
\end{equation}
for $\ell\geq 2$, which together with~\eqref{sca U4} implies~\eqref{U4 converge sup2}.
Again, this implies $\| {\boldsymbol{U}}^n\|_{\mathbf{X}}\rightarrow 0$, contradicting the assumption~\eqref{A-unitaire poly}. 
\end{proof}

\medbreak

The above results allow us to conclude\dots
\begin{thrm}
\label{thm-poly-decay-sm}
The semigroup of contractions $(\check{\mathit{T}_2}(t))_{t\geq0}$, with generator  $-{\widetilde{\xA}}_{2}$, is polynomially  stable on~${\mathbf{\widetilde{X}}_{2}}$,\, i.e., there exist a  constant $C >0$ such that 
\begin{equation}
\forall t>1,\quad
\| \check{\mathit{T}_2}(t) \widetilde{{\boldsymbol{U}}}_0 \|_{\mathbf{\widetilde{X}}_{2}} \leq C\, t^{-\frac1{2}} \| \widetilde{{\boldsymbol{U}}}_0  \|_{D(\widetilde{{\xA}}_2)},\quad \forall\widetilde{{\boldsymbol{U}}}_0 \in D(\widetilde{{\xA}}_2).
\label{poly til X2}
\end{equation}
Furthermore, under the assumptions of Theorem~\Rref{divergence B} on~$\boldsymbol{B}_0$, there exists a constant $M>0$ such that the solution to Problem~\eqref{evol prob sm homo} satisfies 
\begin{equation}
\forall t>1,\quad
\| \mathit{T}_2(t){{\boldsymbol{U}}_0} -  (0,0,0,P_2\boldsymbol{B}_0)^\top \|_{\mathbf{X}_2} \leq M\, t^{-\frac1{2}} \| {\boldsymbol{U}}_{0} \|_{D({\xA}_2)\cap\mathbf{X}_2} ,\quad \forall {\boldsymbol{U}}_{0} \in D({\xA}_2)\cap\mathbf{X}_2.
\label{poly X2}
\end{equation}
\end{thrm}

\subsection{Conditional exponential stability in the Silver--M\"uller case}
\begin{prpstn}
Suppose that the divergence-free, source-free Maxwell system with Silver--M\"uller, or mixed, boundary condition:
\begin{eqnarray}
\left\lbrace \begin{array}{lll}
\partial_t\boldsymbol{E}=c^2\,\curl\boldsymbol{B},\quad \partial_t\boldsymbol{B}=-\curl\boldsymbol{E}, \quad \text{in } \Omega\times\xR_{>0},
\\[2pt]
\divg\boldsymbol{E}=0,\quad \divg\boldsymbol{B}=0, \quad \text{in } \Omega\times\xR_{>0},
\\[2pt]
\boldsymbol{E}\times\boldsymbol{n}=0,\quad \boldsymbol{B}\cdot\boldsymbol{n}=0, \quad \text{on } \Gamma_P\times\xR_{>0},
\\[2pt]
\boldsymbol{E}\times\boldsymbol{n}+ c\,\boldsymbol{B}_{\top}=0, \quad \text{on } \Gamma_A\times\xR_{>0}.
\end{array} \right.
\label{Maxwell SlM}
\end{eqnarray} 
is exponentially stable. Then, the resolvent of the operator~$-\widetilde{{\xA}}_2$ satisfies:
\begin{equation}
\sup_{\beta\in\xR} \III (\ii\beta + \widetilde{{\xA}}_2)^{-1} \III < \infty.
\label{borne resolvent 0}
\end{equation} 
\end{prpstn}
\begin{proof}
Again, we use a contradiction argument. Assume there exists a sequence $(\beta_n)_n$ on~$\xR$, with $|\beta_n| \to +\infty$ as $n\to +\infty$,
and a sequence of fields $({\boldsymbol{U}}^n)_n = \left( ({\boldsymbol{U}}_1^n,{\boldsymbol{U}}_2^n,{\boldsymbol{U}}_3^n,{\boldsymbol{U}}_4^n)^\top \right)_n$ on~$D(\widetilde{{\xA}}_2)$ satisfying \eqref{A-unitaire} and~\eqref{A-resol uni}, the latter being equivalent to the system~\eqref{1exp}--\eqref{4exp}. Following the same reasoning as in Proposition~\ref{0 resolvent sm}, we get 
\begin{eqnarray} 
&&\|{\boldsymbol{U}}_s^n\|\,\rightarrow 0 , \text{ as } n\to +\infty,\quad  s=1,\ 2 \,;
\label{A-sn correc}
\\
&&\| {\boldsymbol{U}}_{4\top}^n \|_{\mathbf{L}^2(\Gamma_A)}^2 \rightarrow 0, \text{ as } n\to +\infty. \qquad\qquad  
\label{AA-bore correc}
\end{eqnarray}
In order to use the exponential stability of~\eqref{Maxwell SlM}, we need to correct ${\boldsymbol{U}}_3^n$ since it does not satisfy $\divg{\boldsymbol{U}}_3^n=0$ in~$\Omega$. Let $\varphi\in\xHone_0(\Omega)$ be the unique solution to 
\begin{eqnarray}
\left(\grad \varphi_n\mid \grad \psi\right)=\left({\boldsymbol{U}}_3^n\mid \grad \psi\right), \quad \forall \psi\in\xHone_0(\Omega).
\label{var form divEE}
\end{eqnarray}
Now, define 
\begin{equation*}
{\hat{\boldsymbol{U}}}_3^n={\boldsymbol{U}}_3^n-\grad \varphi_n, \quad \text{in } \Omega.
\end{equation*}
Then, ${\hat{\boldsymbol{U}}}_3^n$ belongs to $\mathbf{H}_{0,\Gamma_P}(\curl;\Omega)$ and it satisfies
\begin{equation}
\divg{\hat{\boldsymbol{U}}}_3^n=0\quad \text{in } \Omega.
\end{equation} 
Introduce
\begin{equation*}
\boldsymbol{L}_n:=\ii\beta_n\varepsilon_0\,{\boldsymbol{U}}_3^n + \sum_s{\boldsymbol{U}}_s^n - \varepsilon_0 c^2\curl{\boldsymbol{U}}_4^n,
\end{equation*}
the l.h.s.~of~\eqref{3exp}; by assumption $\|\boldsymbol{L}_n\|\to0$.
By choosing $\psi=\varphi_n$ in \eqref{var form divEE} and using the Green formula~\eqref{green rot rot}, we find
\begin{eqnarray*}
\|\grad \varphi_n\|^2 &=& \frac1{\ii\beta_n\varepsilon_0}\int_{\Omega}(\boldsymbol{L}_n-\sum_s{\boldsymbol{U}}_s^n+\varepsilon_0 c^2\curl{\boldsymbol{U}}_4^n)\cdot \grad \overline{\varphi_n}\, \xdif\Omega\\
&=& \frac1{\ii\beta_n\varepsilon_0}\int_{\Omega}(\boldsymbol{L}_n-\sum_s{\boldsymbol{U}}_s^n)\cdot \grad \overline{\varphi_n}\, \xdif\Omega\\
&\leq & \frac{C}{|\beta_n|} (\|\boldsymbol{L}_n\| + \sum_s\|{\boldsymbol{U}}_s^n\|)\,\|\grad \varphi_n\|. 
\end{eqnarray*}
Then, by \eqref{3exp} and \eqref{A-sn correc} we deduce that 
\begin{equation}
\|\beta_n\grad\varphi_n\|\,\rightarrow 0 , \text{ as } n\to +\infty.
\label{lim gra fi n}
\end{equation}
We now introduce 
\begin{eqnarray*}
\hat{\boldsymbol{L}}_n &=& \boldsymbol{L}_n-\sum_s{\boldsymbol{U}}_s^n-\ii\beta_n\varepsilon_0\grad\varphi_n,\\
\boldsymbol{Q}_n &=& \ii\varepsilon_0 c^2\beta_n\,{\boldsymbol{U}}_4^n + \varepsilon_0 c^2 \curl{\boldsymbol{U}}_3^n.
\end{eqnarray*}
By \eqref{4exp}, \eqref{A-sn correc} and~\eqref{lim gra fi n}, it holds that:
\begin{equation}
\hat{\boldsymbol{L}}_n,\, \boldsymbol{Q}_n \,\rightarrow 0 \quad \in \mathbf{L}^2(\Omega), \text{ as } n\to +\infty.
\label{lim l Q}
\end{equation}
To summarize, the pair $({\hat{\boldsymbol{U}}}_3^n,{\boldsymbol{U}}_4^n)$ satisfies the perfectly conducting boundary condition on $\Gamma_P$, the Silver--M\"uller boundary condition on $\Gamma_A$, and the divergence-free harmonic Maxwell problem in~$\Omega$:
\begin{eqnarray}
\left\lbrace \begin{array}{lll}
\ii\beta_n\varepsilon_0\,{\hat{\boldsymbol{U}}}_3^n  - \varepsilon_0 c^2\curl{\boldsymbol{U}}_4^n=\hat{\boldsymbol{L}}_n,
\\[2pt]
\ii\varepsilon_0 c^2\beta_n\,{\boldsymbol{U}}_4^n + \varepsilon_0 c^2 \curl{\hat{\boldsymbol{U}}}_3^n = \boldsymbol{Q}_n .
\end{array} \right.
\label{sys L Q}
\end{eqnarray}
By assumption, the system~\eqref{Maxwell SlM} is exponentially stable; thus, according to Theorem~\ref{stab exp theor}, its resolvent is uniformly bounded on the imaginary axis. In other words, there exists a positive constant~$C$, independent of~$n$, such that the solution $({\hat{\boldsymbol{U}}}_3^n,{\boldsymbol{U}}_4^n)$ to~\eqref{sys L Q} satisfies 
\begin{equation*}
\|{\hat{\boldsymbol{U}}}_3^n\| + \|{\boldsymbol{U}}_4^n\| \leq C\, \left( \|\hat{\boldsymbol{L}}_n\| + \|\boldsymbol{Q}_n\| \right).
\end{equation*}
From~\eqref{lim l Q}, we deduce that 
\begin{equation*}
\|{\hat{\boldsymbol{U}}}_3^n\| \,+\, \|{\boldsymbol{U}}_4^n\|  \,\rightarrow 0 ,\quad \text{as } n\to +\infty,
\end{equation*}
and finally by \eqref{lim gra fi n} we get
\begin{equation*}
\|{\boldsymbol{U}}_3^n\| \,+\, \|{\boldsymbol{U}}_4^n\|  \,\rightarrow 0 ,\quad \text{as } n\to +\infty,
\end{equation*}
which together with \eqref{A-sn correc} gives the desired contradiction of~\eqref{unitaire}.
\end{proof}

\medbreak

Using again Theorem~\Rref{stab exp theor}, we obtain a conditional improved version of the decay Theorem~\ref{thm-poly-decay-sm}.
\begin{thrm}
\label{thm-expo-decay-sm}
Assume that the divergence-free Maxwell system~\eqref{Maxwell SlM} is exponentially stable. Then, the semigroup of contractions $(\check{\mathit{T}_2}(t))_{t\geq0}$, with generator  $-{\widetilde{\xA}}_{2}$, is exponentially stable on~${\mathbf{\widetilde{X}}_{2}}$,\, \ie, there exist two  constants $C,\gamma >0$ such that 
\begin{equation}
\forall t\ge0,\quad
\| \check{\mathit{T}_2}(t) \widetilde{\boldsymbol{U}}_0 \|_{\mathbf{\widetilde{X}}_{2}} \leq C\, \ee^{-\gamma t} \| \widetilde{\boldsymbol{U}}_0  \|_{\mathbf{\widetilde{X}}_{2}},\quad \forall\widetilde{\boldsymbol{U}}_0 \in \mathbf{\widetilde{X}}_{2}.
\label{expo til X2}
\end{equation}
Furthermore, under the assumptions of Theorem~\Rref{divergence B} on~$\boldsymbol{B}_0$, there exists a constant $M>0$ such that the solution to Problem~\eqref{evol prob sm homo} satisfies 
\begin{equation*}
\forall t\ge0,\quad
\| \mathit{T}_2(t){\boldsymbol{U}_0} -  (0,0,0,P_2\boldsymbol{B}_0)^\top \|_{\mathbf{X}_2} \leq M\, \ee^{-\gamma t} \| \boldsymbol{U}_{0} \|_{\mathbf{X}_2} \,,\quad \forall \boldsymbol{U}_{0} \in \mathbf{X}_2.
\end{equation*}
\end{thrm}
\begin{rmrk}
Examples of sufficient conditions for the exponential stability of the Maxwell system with pure Silver--M\"uller or mixed boundary conditions are given in the seminal papers~\cite{Komo94,Phung00}. 
\end{rmrk}
\begin{rmrk}
On the other hand, our model with perfectly conduction boundary condition everywhere ($\Gamma_A = \varnothing$) is \emph{never} exponentially stable: there exists no improved version of Theorem~\Rref{thm-poly-decay-cp}. In this case, the Maxwell operator has an infinite number of eigenvalues on the imaginary axis: the associated evolution operator cannot be exponentially stable.

\smallbreak

Consider the eigenvalue problem with perfectly conducting boundary condition:
\begin{equation}
c^2\,\curl\curl \boldsymbol{E}_k = \lambda_k^2\,\boldsymbol{E}_k, \quad \divg\boldsymbol{E}_k = 0 \quad \text{in } \Omega,\quad \boldsymbol{E}_k \times \boldsymbol{n} = 0 \quad \text{on } \Gamma.
\label{eigenproblem-maxwell-cp}
\end{equation}
It is well-known (see, \vg,~\cite[\S8.2.1]{ACL+17}) that it admits a nondecreasing sequence of eigenvalues tending to infinity. Assume that the corresponding eigenvectors are normalized by $\| \boldsymbol{E}_k \|=1$, and introduce the sequence $\left( \boldsymbol{U}^k \right)_{k\in\xN}$ on~$D({\widetilde{\xA}}_{1})$ as:
\begin{eqnarray*}
{\boldsymbol{U}}_s^k = (\ii\lambda_k\, \mathbb{I} + \mathbb{M}_s)^{-1}\, \varepsilon_0\omega_{ps}^2\,\boldsymbol{E}_k,\ s =1,\ 2,\quad
{\boldsymbol{U}}_3^k = \boldsymbol{E}_k,\quad
{\boldsymbol{U}}_4^k = -\frac{1}{\ii\lambda_k}\, \curl\boldsymbol{E}_k.
\end{eqnarray*}
Thanks to~\S\ref{sec-spectral}, it is easily shown that $\lambda_k\to+\infty$ implies $\III (\ii\lambda_k\, \mathbb{I} + \mathbb{M}_s)^{-1} \III_{\mathcal{M}} \to 0$. Therefore, $\| {\boldsymbol{U}}_s^k \| \to 0$ as $k\to+\infty$, and:
\begin{eqnarray*}
(\ii\lambda_k\, \mathbb{I} + {\widetilde{\xA}}_{1})\, {\boldsymbol{U}}^k = \left(0, 0, \tfrac1{\varepsilon_0}\,({\boldsymbol{U}}_1^k + {\boldsymbol{U}}_2^k), 0 \right)^\top
\to 0 \quad \text{in } \mathbf{L}^2(\Omega), \text{ as } k\to+\infty.
\end{eqnarray*}
On the other hand, taking $\boldsymbol{E}_k$ as a test function in~\eqref{eigenproblem-maxwell-cp}, one finds $\| {\boldsymbol{U}}_4^k \| = 1/c$. All in all, $0 < U_* \le \| {\boldsymbol{U}}^k \| \le U^* < +\infty$, for some $U_*,\ U^*$ independent of~$k$. This shows that the counterpart of \eqref{born'e resolvent} or~\eqref{borne resolvent 0} cannot hold for~$-{\widetilde{\xA}}_{1}$.
\end{rmrk}

\subsection{Convergence to the harmonic regime}
A time-harmonic solution to the model~\eqref{uu}--\eqref{lm} is a particular solution such that $\boldsymbol{U}(t,\boldsymbol{x}) = \Re \left[ \mathbf{U}(\boldsymbol{x})\, \ee^{-\ii\omega t} \right]$. Such a solution may only exist if two conditions are satisfied: 
(i)~the forcing or Silver--M\"uller data is time-harmonic ($\boldsymbol{g}(t,\boldsymbol{x})=\Re \left[ \mathbf{g}(\boldsymbol{x})\,\ee^{-\ii\omega t} \right]$) and (ii)~the initial data match ($\boldsymbol{U}_0(\boldsymbol{x}) = \Re \left[ \mathbf{U}(\boldsymbol{x}) \right]$). Of course, the general condition~\eqref{initial cond nonho} must also hold.

\medbreak

The time-harmonic version of~\eqref{uu}--\eqref{pp2}, \ie, with $\partial_t \mapsto - \ii\omega$, has been studied in~\cite{BHL+15}. Under Hypotheses \Rref{hyp-1} and~\Rref{hyp-2}, its well-posedness has been established with slightly different boundary conditions, but the adaptation to the Silver--M\"uller case is not difficult. Indeed, the solution to time-harmonic problem, supplemented with boundary conditions formally similar to~\eqref{lk}--\eqref{lm}, can be expressed with the tools introduced in this paper. The harmonic variables will be denoted with upright bold letters: $\mathbf{J}_s,\ \mathbf{E},\ \mathbf{B}$, \etc.

\medbreak

Let $\mathbf{g} \in \widetilde{\mathbf{TT}}(\Gamma_{A}) + \mathbf{TC}(\Gamma_{A})$. As in \S\Rref{sm-nonhomo}, set:
\begin{eqnarray*}
(\mathbf{g}_3,\mathbf{g}_4) \in \mathbf{H}_{0,\Gamma_P}(\curl;\Omega) \times \mathbf{H}(\curl;\Omega) \quad\text{s.t.}\quad \mathbf{g}_3\times \boldsymbol{n} + c\,\mathbf{g}_{4\top} = \mathbf{g} \text{ on } \Gamma_A, 
\\
\mathbf{J}^\star_1 = \mathbf{J}_1,\quad \mathbf{J}^\star_2 = \mathbf{J}_2,\quad \mathbf{E}^\star = \mathbf{E}-\mathbf{g}_3,\quad \mathbf{B}^\star = \mathbf{B}-\mathbf{g}_4.
\end{eqnarray*}
The variable $\mathbf{U}^\star \in D(\xA_2)$ is solution to
\begin{equation}
-\ii\omega\, \mathbf{U}^\star  + \xA_2 \mathbf{U}^\star = 
\left(\varepsilon_0\omega_{p1}^2\,\mathbf{g}_3,\, \varepsilon_0\omega_{p2}^2\,\mathbf{g}_3,\, \ii\omega\, \mathbf{g}_3 + c^2 \curl \mathbf{g}_4,\, \ii\omega\, \mathbf{g}_4 - \curl \mathbf{g}_3 \right)^\top \,,
\label{UU*-harmo}
\end{equation}
a well-posed equation according to Proposition~\Rref{surj I+A2}. Then $\mathbf{U} := \mathbf{U}^\star + (0,0,\mathbf{g}_3,\mathbf{g}_4)^\top$ is solution to the time-harmonic problem. Obviously, the difference of two solutions belongs to $D(\xA_2)$ and satisfies $-\ii\omega\, \mathbf{U}  + \xA \mathbf{U} = 0$, hence it vanishes by Proposition~\Rref{inject sm stab}.
\begin{dfntn}
For any $\mathbf{g} \in \widetilde{\mathbf{TT}}(\Gamma_{A}) + \mathbf{TC}(\Gamma_{A})$, we denote $H[\mathbf{g}] := \mathbf{U}$, the unique solution to the time-harmonic problem constructed by the above procedure.
\end{dfntn}
By uniqueness, \emph{any} lifting $(\mathbf{g}_3,\mathbf{g}_4)$ of the boundary data~$\mathbf{g}$ can actually be used. For instance, it is possible to take $\mathbf{g}_4 \in \mathbf{H}_{0,\Gamma_P;\text{flux},\Gamma_A,\Sigma}(\divg 0;\Omega)$. Starting with $(\mathbf{g}_3^0,\mathbf{g}_4^0) = R_A[\mathbf{g}]$, one defines $\varphi\in \xHone_{0,\Gamma_A}(\Omega) := \{ w \in \xHone(\Omega) : w_{|_{\Gamma_A}} = 0 \}$ as the solution to
$$( \grad \varphi \mid \grad \psi ) = ( \mathbf{g}_4^0 \mid \grad\psi), \quad \forall \psi\in \xHone_{0,\Gamma_A}(\Omega),$$
and $\mathbf{g}_4^1 := \mathbf{g}_4^0 - \grad\varphi$. By~\eqref{green div grad}, $\mathbf{g}_4^1 \in \mathbf{H}_{0,\Gamma_P}(\divg 0;\Omega)$ and $\mathbf{g}_{4\top}^1 = \mathbf{g}_{4\top}^0$ on~$\Gamma_A$. Then one sets:
\begin{equation*}
\mathbf{g}_3 = \mathbf{g}_3^0,\quad \mathbf{g}_4 = \mathbf{g}_4^1 - P_2\, \mathbf{g}_4^1  \,;
\end{equation*}
recall that $P_2$ is the orthogonal projection onto~$\mathbf{Z}(\Omega;\Gamma_A)$. By~\eqref{orth decom Z GammaA}, $\mathbf{g}_4 \in \mathbf{H}_{0,\Gamma_P;\text{flux},\Gamma_A,\Sigma}(\divg 0;\Omega)$, and $\mathbf{g}_{4\top} = \mathbf{g}_{4\top}^1 = \mathbf{g}_{4\top}^0$ on~$\Gamma_A$, \ie, $\mathbf{g}_3\times \boldsymbol{n} + c\,\mathbf{g}_{4\top} = \mathbf{g}$.

\medbreak

\begin{prpstn}
The range of the mapping $H$ is included in~$\mathbf{\widetilde{X}}_{2}$.
\label{pro-range-H}
\end{prpstn}
\begin{proof}
Let $\mathbf{g} \in \widetilde{\mathbf{TT}}(\Gamma_{A}) + \mathbf{TC}(\Gamma_{A})$.
Take $\mathbf{g}_4 \in \mathbf{H}_{0,\Gamma_P;\text{flux},\Gamma_A,\Sigma}(\divg 0;\Omega)$ as above. As $\mathbf{g}_3 \in \mathbf{H}_{0,\Gamma_P}(\curl;\Omega)$ by definition, its curl also belongs to~$\mathbf{H}_{0,\Gamma_P;\text{flux},\Gamma_A,\Sigma}(\divg 0;\Omega)$. Hence, the right-hand side of~\eqref{UU*-harmo} actually belongs to~$\mathbf{\widetilde{X}}_{2}$. By Proposition~\Rref{spec-Atild2} the solution also belongs to this space, and so does finally $\mathbf{U} = \mathbf{U}^\star + (0,0,\mathbf{g}_3,\mathbf{g}_4)^\top$. 
\end{proof}

\medbreak

By uniqueness of the solution to the time-dependent model~\eqref{uu}--\eqref{lm}, $\boldsymbol{U}^\omega(t,\boldsymbol{x}) = \Re\left[ H[\mathbf{g}](\boldsymbol{x})\, \ee^{-\ii\omega t} \right]$ is the solution to the said system, with the Silver--M\"uller data $\boldsymbol{g}(t,\boldsymbol{x}) = \Re\left[ \mathbf{g}(\boldsymbol{x})\,\ee^{-\ii\omega t} \right]$ and the ``well-prepared'' initial condition $\boldsymbol{U}^\omega_0(\boldsymbol{x}) = \Re\left[ H[\mathbf{g}](\boldsymbol{x}) \right]$. Thus, the necessary conditions stated at the beginning of this Subsection are actually sufficient.

\medbreak

On the other hand, if the forcing is still time-harmonic, but the initial condition is arbitrary, the solution to~\eqref{uu}--\eqref{lm} does not have a time-harmonic form. However, if the initial condition satisfies both the compatibility condition and the physical requirements for a magnetic field, then the solution converges to the time-harmonic one as fast as the solution to the homogeneous system converges to~$0$.
\begin{thrm}
Let $\boldsymbol{U} = (\boldsymbol{J}_1,\boldsymbol{J}_2,\boldsymbol{E},\boldsymbol{B})^\top$ be the solution to~\eqref{uu}--\eqref{lm} with the time-harmonic  Silver--M\"uller data $\boldsymbol{g}(t,\boldsymbol{x}) = \Re\left[ \mathbf{g}(\boldsymbol{x})\,\ee^{-\ii\omega t} \right]$, where $\mathbf{g} \in \widetilde{\mathbf{TT}}(\Gamma_{A}) + \mathbf{TC}(\Gamma_{A})$, and the initial data $\boldsymbol{U}_0 = (\boldsymbol{J}_{1,0},\boldsymbol{J}_{2,0},\boldsymbol{E}_0,\boldsymbol{B}_0)^\top$ satisfying~\eqref{initial cond nonho} and $\boldsymbol{B}_0 \in \mathbf{H}_{0,\Gamma_P;\text{flux},\Gamma_A,\Sigma}(\divg 0;\Omega)$. There exists $K(\mathbf{g},\boldsymbol{U}_0)$ such that:
\begin{equation}
\left\| \boldsymbol{U}(t) - \boldsymbol{U}^\omega(t) \right\|_{\mathbf{X}} \leq K(\mathbf{g},\boldsymbol{U}_0)\, \phi(t),
\end{equation}
where $\boldsymbol{U}^\omega(t,\boldsymbol{x}) = \Re\left[ H[\mathbf{g}](\boldsymbol{x})\, \ee^{-\ii\omega t} \right]$, and the decay function $\phi(\cdot)$ can be taken in all cases as $\phi(t) = t^{-\frac12}$ for $t>1$, and as $\phi(t) = \ee^{-\gamma\,t}$ if the divergence-free Maxwell system~\eqref{Maxwell SlM} is stable.
\end{thrm}
\begin{proof}
The difference $\boldsymbol{U}-\boldsymbol{U}^\omega$ is solution to~\eqref{uu}--\eqref{lm} with homogeneous Silver--M\"uller boundary condition. Furthermore, the  initial data $\boldsymbol{U}_0 - \Re H[\mathbf{g}]$ satisfies the same boundary condition; by construction, it belongs to~$\mathbf{\widetilde{X}}_{2}$ by Proposition~\ref{pro-range-H}; while its third and fourth components belong to~$\mathbf{H}_{0,\Gamma_P}(\curl;\Omega) \times \mathbf{H}(\curl;\Omega)$. All in all, $\boldsymbol{U}_0 - \Re H[\mathbf{g}] \in D(\widetilde{{\xA}}_2)$: one can apply the estimate~\eqref{poly til X2}, and even~\eqref{expo til X2} under the assumptions of Theorem~\ref{thm-expo-decay-sm}.
\end{proof}

\begin{acknowledgement}
The authors are indebted to Serge Nicaise for useful remarks and discussions.
\end{acknowledgement}


\end{document}